\documentclass[preprint,12pt]{elsarticle}

\usepackage{amsmath, bm}
\usepackage{amsthm}
\usepackage{enumitem}
\usepackage{graphicx}
\usepackage{tikz}
\usepackage{natbib}
\usepackage{subcaption}
\usepackage[ruled, linesnumbered]{algorithm2e}
\usepackage[title]{appendix}
\usepackage{epstopdf}
\newtheorem{theorem}{Theorem}
\newtheorem{lemma}{Lemma}
\newtheorem{corollary}{Corollary}

\newtheorem{observation}{Observation}

\newtheorem{conjecture}{Conjecture}

\newtheorem{claim}{Claim}

\usetikzlibrary{positioning}

\newcommand{\rev}[1]{\textcolor{black}{#1}}

\theoremstyle{definition}

\newcommand{\D}{\mathcal{D}}

\journal{Advances in Applied Mathematics}

\begin{document}

\begin{frontmatter}



\title{Level-$2$ networks from shortest and longest distances}


\author[inst1]{Katharina T. Huber\fnref{fn1}}
\ead{K.Huber@uea.ac.uk}
\author[inst2]{Leo van Iersel\fnref{fn1}}
\ead{L.J.J.vanIersel@tudelft.nl}
\author[inst2]{Remie Janssen\fnref{fn1}}
\ead{R.Janssen-2@tudelft.nl}
\author[inst2]{Mark Jones\fnref{fn1}}
\ead{M.E.L.Jones@tudelft.nl}
\author[inst1]{Vincent Moulton\fnref{fn1}}
\ead{V.Moulton@uea.ac.uk}
\author[inst2]{Yukihiro Murakami\fnref{fn1}}
\ead{Y.Murakami@tudelft.nl}

\affiliation[inst1]{organization={School of Computing Sciences, University of East Anglia},
            addressline={Norwich Research Park}, 
            city={Norwich},
            postcode={NR4 7TJ}, 
            country={United Kingdom}}

\affiliation[inst2]{organization={Delft Institute of Applied Mathematics, Delft
    University of Technology},
            addressline={Mekelweg 4}, 
            city={Delft},
            postcode={2628 CD}, 
            country={The Netherlands}}


\fntext[fn1]{Research funded in part by the Netherlands Organization for 
Scientific Research (NWO) Vidi grant 639.072.602 and Klein grant 
OCENW.KLEIN.125, and partly by the 4TU Applied Mathematics Institute. Mark 
Jones was also supported by the Netherlands Organisation for Scientific 
Research (NWO) through Gravitation Programme Networks 024.002.003.}


\begin{abstract}
    Recently it was shown that a certain class of phylogenetic networks, called 
	level-$2$ networks, cannot be reconstructed from their associated distance 
	matrices.
	In this paper, we show that they can be reconstructed from their induced 
	shortest and longest distance matrices. That is, if two level-$2$ networks 
	induce the 
	same shortest and longest distance matrices, then they must be isomorphic. 
	We further show that level-$2$ networks are reconstructible from their 
	shortest distance matrices if and only if they do not contain a subgraph 
	from a family of graphs.
	A generator of a network is the graph 
	obtained by 
	deleting all pendant subtrees and suppressing degree-$2$ vertices. We also 
	show that networks with a leaf on every generator side is reconstructible 
	from their induced shortest distance matrix.
\end{abstract}




\end{frontmatter}



\section{Introduction}

Finding a weighted undirected graph that realizes a distance matrix has 
applications in 
phylogenetics~\citep{huson2010phylogenetic,morrison2011introduction}, 
psychology~\citep{cunningham1978free, schvaneveldt1989network}, electricity 
networks~\citep{hakimi1965distance, forcey2020phylogenetic}, 
information theory~\citep{dewdney1979diagonal}, and other areas. 
In their seminal paper,~\cite{hakimi1965distance} showed that a necessary and 
sufficient condition for a distance matrix to be realizable on a graph is for 
it to be a metric space. While this gave existence for graph realizations on 
any metric spaces, such realizations were not necessarily unique. Many of the 
existing distance methods, including the one from this paper, take the 
following approach. 
Assuming a graph~$G$ realizes some distance matrix~$M$, we first 
identify pendant structures of~$G$ from the information provided by~$M$. The 
elements involved in such structures are clustered into one element in the 
newly updated distance matrix~$M'$. This process is repeated until all 
structures of~$G$ have been identified, at which point we have essentially 
constructed~$G$ (should such a graph exist).


We consider a restriction of the distance matrix realizability problem 
in the context of phylogenetics, where graphs such as phylogenetic networks are 
used to elucidate the evolutionary histories of taxa.
In recent years, phylogenetic networks have attracted increasing attention over 
phylogenetic trees, due to their generalized nature and ability to represent 
non-treelike evolutionary histories; this is suitable in visualising complex 
reticulate events such as hybridization events and introgression, found to be 
rife within plants and 
bacteria~\citep{huson2010phylogenetic,bapteste2013networks}.
In the context of phylogenetics, distances are defined between pairs of taxa to 
denote the number of character 
changes (in terms of the nucleotide bases in DNA) or the evolutionary / genetic 
distance between them.
These distances are generally obtained from multiple sequence alignments.
In this paper, we study the \emph{reconstructibility} of phylogenetic networks 
from certain 
distance matrices. We say that a network is \emph{reconstructible} from its 
induced distance matrix if it is the unique network that realizes the distance 
matrix.

Because networks contain undirected cycles, there can be many paths between two 
leaves (leaves are labelled by unique taxa, so leaves and taxa will be used 
interchangeably). This is in contrast to trees which contain only one path 
between 
each leaf pair. So while a tree induces only one metric, networks may induce 
many\footnote{A matrix consisting of inter-leaf shortest distances, and a 
matrix consisting of inter-leaf longest distances are two examples of a metric 
induced by a network.}.
To date, different matrices have been considered for 
phylogenetic network 
reconstruction. These include the shortest distances (the traditional distance 
matrix),
the sets of distances~\citep{bordewich2016algorithm}, and the
multisets of 
distances~\citep{bordewich2016determining,van2020reconstructibility}, where the 
latter two are not distance matrices in the traditional sense of the term, 
however, 
they contain information on the inter-leaf distances. In particular, an element 
of the multisets of distances is a multiset of all distances between a pair of 
leaves, together with the multiplicities associated to each length. The set of 
distances can be obtained from the multisets of distances by ignoring the 
multiplicities\footnote{The motivation for considering sets and multisets of 
distances is mostly combinatorial. It is not clear how such distances can be 
obtained from a multiple sequence alignment. A possibility is to divide the
alignment into \emph{blocks} depending on the parts of the chromosome 
responsible for encoding a particular gene, or by optimizing constraints such 
as the homoplasy score~\citep{jones2019cutting}. Treating each of these blocks 
as an alignment yields a distance matrix for each block, which can be collated 
to give
multisets of distances between pairs of taxa. The fundamental flaw in this 
technique would be that every multiset of distances between leaf pairs would be 
of the same size (in particular the number of blocks), which is not always the 
case in the results where these multisets are used.}.

Both the sets and multisets of distances were first introduced to prove 
reconstructibility results for distance matrices induced by particular 
phylogenetic 
network classes~\citep{bordewich2016determining, bordewich2016algorithm}; 
shortest distances have also been used to prove reconstructibility 
results~\citep{willson2006unique, 
bryant2007consistency}.
Recent results have shown unique realizability from sets and multisets of 
distances for certain rooted 
networks (tree-child and normal 
networks\footnote{\rev{Tree-child networks are directed networks with the property that every non-leaf vertex has a child that is a tree vertex or a leaf. Normal networks are tree-child networks with the additional constraint that given an edge~$uv$, there cannot be another path from~$u$ to~$v$.}})~\citep{bordewich2018recovering,bordewich2018constructing} and for 
certain
unrooted networks~\citep{van2020reconstructibility}. 
In particular, the results of
\citep{van2020reconstructibility} showed that unweighted binary level-$2$ 
networks are reconstructible from their multisets of distances. Binary 
means that each leaf is of degree-$1$ and every other vertex is of degree-$3$; 
the level of a network refers to the number of edges needed to be removed from 
every biconnected component to obtain a tree (described more in detail below).
They also showed that level-$1$ networks are reconstructible from their 
shortest distances, but that level-$2$ networks were not reconstructible in 
general from their shortest distances (Figure~\ref{fig:level2CE}). 

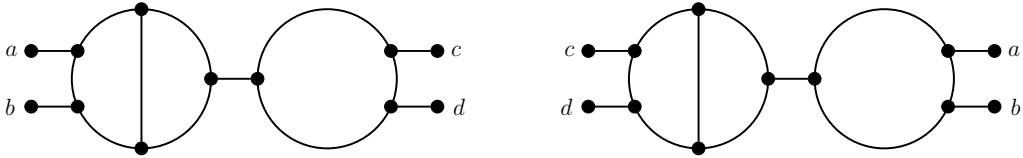
\begin{figure}[ht!]
	\centering
	\resizebox{\columnwidth}{!}{
		\begin{tikzpicture}[every node/.style={draw, circle, fill, inner 
			sep=0pt}]
		\tikzset{edge/.style={very thick}}
		\begin{scope}[xshift=0cm,yshift=0cm]
		\draw[black, very thick] (0,0) circle (1.5);
		\draw[black, very thick] (-4,0) circle (1.5);
		
		\node[] (v2)    at (1.3747,0.6)     {a};
		\node[] (v3)    at (1.3747,-0.6)    {a};
		\node[] (v4)    at (-1.5,0)         {a};
		\node[] (u1)    at (-2.5,0)         {a};
		\node[] (u2)    at (-4,1.5)         {a};
		\node[] (u3)    at (-4,-1.5)        {a};
		\node[] (u4)    at (-5.3747,0.6)    {a};
		\node[] (u5)    at (-5.3747,-0.6)   {a};
		
		\node[] (a)     at (-6.3747,0.6)    {a};
		\node[] (b)     at (-6.3747,-0.6)   {a};
		\node[] (c)     at (2.3747,0.6)     {a};
		\node[] (d)     at (2.3747,-0.6)    {a};
		
		\node[draw = none, fill = none, left=1mm of a]  (la) {\large $a$};
		\node[draw = none, fill = none, left=1mm of b]  (lb) {\large $b$};
		\node[draw = none, fill = none, right=1mm of c] (lc) {\large $c$};
		\node[draw = none, fill = none, right=1mm of d] (ld) {\large $d$};
		
		\draw[edge] (u2) -- (u3);
		\draw[edge] (u1) -- (v4);
		\draw[edge] (v2) -- (c);
		\draw[edge] (v3) -- (d);
		\draw[edge] (u4) -- (a);
		\draw[edge] (u5) -- (b);
		\end{scope}
		
		\begin{scope}[xshift=12cm,yshift=0cm]
		\draw[black, very thick] (0,0) circle (1.5);
		\draw[black, very thick] (-4,0) circle (1.5);
		
		\node[] (v2)    at (1.3747,0.6)     {a};
		\node[] (v3)    at (1.3747,-0.6)    {a};
		\node[] (v4)    at (-1.5,0)         {a};
		\node[] (u1)    at (-2.5,0)         {a};
		\node[] (u2)    at (-4,1.5)         {a};
		\node[] (u3)    at (-4,-1.5)        {a};
		\node[] (u4)    at (-5.3747,0.6)    {a};
		\node[] (u5)    at (-5.3747,-0.6)   {a};
		
		\node[] (a)     at (-6.3747,0.6)    {a};
		\node[] (b)     at (-6.3747,-0.6)   {a};
		\node[] (c)     at (2.3747,0.6)     {a};
		\node[] (d)     at (2.3747,-0.6)    {a};
		
		\node[draw = none, fill = none, left=1mm of a]  (la) {\large $c$};
		\node[draw = none, fill = none, left=1mm of b]  (lb) {\large $d$};
		\node[draw = none, fill = none, right=1mm of c] (lc) {\large $a$};
		\node[draw = none, fill = none, right=1mm of d] (ld) {\large $b$};
		
		\draw[edge] (u2) -- (u3);
		\draw[edge] (u1) -- (v4);
		\draw[edge] (v2) -- (c);
		\draw[edge] (v3) -- (d);
		\draw[edge] (u4) -- (a);
		\draw[edge] (u5) -- (b);
		\end{scope}
		
		\end{tikzpicture}}
	\caption{Two level-$2$ networks with the same shortest distances between 
		any pair of leaves. The shortest distance can be worked out by taking 
		the length of the shortest path between a pair of leaves, where each 
		edge is of length~$1$.}
	\label{fig:level2CE}
\end{figure}

It is interesting to know which level-$2$ networks are reconstructible from 
their shortest distances, or what additional information is needed to be able 
to do so.
Therefore, motivated by the results 
in~\citep{van2020reconstructibility}, we answer three open problems \rev{for binary networks}
from the paper on unique realizability of certain distance matrices.

%

\begin{enumerate}
	\item Networks with a leaf on every generator side are reconstructible from 
	their induced shortest distance matrix (Theorem~\ref{thm:GenSide}); 
	\item Level-$2$ networks are reconstructible from their induced shortest 
	and longest distance matrix \rev{(sl-distance matrix)} (Theorem~\ref{thm:L2SL});
	\item We characterize subgraphs of level-$2$ networks 
	that are responsible for the class to not be reconstructible from their 
	induced shortest distance matrix (Theorem~\ref{thm:NoAltPath}).
\end{enumerate}

\paragraph{Structure of the paper}
In Section~\ref{sec:Preliminaries}, we give formal definitions of phylogenetic 
terms.
In Section~\ref{sec:GeneratorSide}, we show that networks with a 
leaf on every generator side (i.e., every vertex on the network is at most 
shortest distance-2 away from a leaf) are reconstructible from their shortest 
distances 
(Theorem~\ref{thm:GenSide}). In Section~\ref{sec:SLDistance}, we show that 
level-$2$ networks are reconstructible from their sl-distance matrices 
(Theorem~\ref{thm:L2SL}). This is proven by first showing that the splits of 
the network (cut-edges that induce a partition on the labelled leaves) are 
determined by the shortest distances that they realize 
(Theorem~\ref{thm:CutEdgeSplit}). In Section~\ref{sec:Forbidden}, we show a 
construction for obtaining pairs of distinct level-$2$ networks from a binary 
tree that realize the same shortest distance matrix. We show that having such a 
network as a subgraph renders a level-$2$ network to be non-reconstructible 
from their shortest distances, thereby characterizing the family of subgraphs 
that are responsible for the non-reconstructibility 
(Theorem~\ref{thm:NoAltPath}). We close with a discussion in 
Section~\ref{sec:Discussion}, presenting ideas for possible future directions.



\section{Preliminaries}\label{sec:Preliminaries}
An \emph{(unrooted binary phylogenetic) network} on~$X$ (where~$|X|\ge2$) is a 
simple connected 
undirected graph with at least two leaves where the leaves are labelled 
bijectively by~$X$ and are of degree-$1$. All internal vertices are of 
degree-$3$. An \emph{(unrooted binary phylogenetic) tree} on~$X$ is a network 
with no cycles.

\subsection{Graph Theoretic Definitions}

Let~$N$ be a network. A set of two leaves~$\{x,y\}$ of~$N$ forms a 
\emph{cherry} if they share 
a common neighbor. Let~$a=(a_1,\ldots,a_k)$ be an ordered sequence of~$k$ 
leaves, and let~$p_a = (v_1,\ldots,v_k)$, where~$v_i$ is the neighbor of~$a_i$ 
for each~$i\in [k]=\{1,\ldots, k\}$\footnote{For consistency later on 
the section, we let~$[0]=\emptyset$, the empty set.}. We allow for~$v_1=v_2$ and~$v_{k-1} = 
v_k$. If~$p_a$ is a path in~$N$ then~$a$ is called a \emph{chain} of 
length~$k\ge0$. 
\rev{We call chains of length~$0$ an empty chain. Assume that all chains are non-empty unless stated otherwise. Letting chains be of empty length is to generalize some statements later on in the paper.}
We say that~$a$ is a \emph{maximal chain} if~$a$ does not form a 
subsequence for some other chain. We assume all chains to be maximal, unless 
stated otherwise. If~$a$ is a chain, then the vertices of~$p_a$ 
are called the \emph{spine vertices}, and the path~$p_a$ is called a 
\emph{spine}. The vertices~$a_1,a_k$ are called the \emph{end-leaves} of the 
chain~$a$, and the vertices~$v_1,v_k$ are called the \emph{end-spine vertices} 
of the chain.
\rev{For brevity, given a set~$S$, we shall write~$S\cup a$ to denote the set~$S\cup \{a_1,a_2,\ldots, a_k\}$.}

A \emph{blob} of a network is a maximal 2-connected subgraph with at least 
three vertices. A network is a \emph{level-$k$ network}, with~$k\ge0$, if at 
most~$k$ edges must be deleted from every blob to obtain a tree.
\rev{We denote an edge between~$u$ and~$v$ by~$uv$.}
We call a 
cut-edge 
\emph{trivial} if the edge is incident to a leaf, and 
\emph{non-trivial} otherwise. Given a cut-edge~$uv$ we say that a leaf~$x$ can 
be \emph{reached from}~$u$ (via~$uv$) if, upon deleting the edge~$uv$ without 
suppressing degree-$2$ vertices,~$x$ is in the same component as~$v$ in the 
resulting subgraph. We say that a leaf is 
\emph{contained} in a blob if the neighbor of the leaf 
is a vertex of the blob. We say that a chain is \emph{contained} in a blob if 
any of the leaves of the chain are contained in the blob (and therefore all 
leaves of the chain are contained in the blob).
An edge is 
\emph{incident} to a blob if exactly one of the endpoints of the edge is a 
vertex of the blob. A blob is \emph{pendant} if there is exactly one 
non-trivial cut-edge that is incident to the blob. We say that a leaf~$x$ can 
be \emph{reached} from a blob~$B$ via a cut-edge~$uv$ if~$u$ is a vertex of~$B$ 
and~$x$ can be reached from~$u$ via~$uv$. In this case, we also say that~$uv$ 
or~$u$ \emph{separates}~$x$ from~$B$.

Letting~$X$ be a set of taxa, a \emph{split} on~$X$ is a partition~$\{A,B\}$ 
of~$X$. We denote a split which induces the partition~$\{A,B\}$ of~$X$ by~$A|B$ 
where the order in which we list~$A$ and~$B$ does not matter. 
Observe that some cut-edges of a network on~$X$ naturally induce a split as 
there are exactly two parts of the network separated by the edge. We call this 
a \emph{cut-edge induced split}. We call a split~$A|B$ \emph{non-trivial} if 
both~$A$ and~$B$ contain at least two elements. Otherwise we call a split 
\emph{trivial}. Observe that non-trivial cut-edges induce non-trivial splits, 
and that trivial cut-edges induce trivial splits.

In this paper, we assume the restriction that every cut-edge must induce a 
unique split. Firstly, such a restriction eliminates the possibility for 
networks to contain~\emph{redundant} blobs, which are pendant blobs that 
contain no leaves. Secondly, the restriction removes all blobs that do not 
contain leaves, that are incident only to two non-trivial cut-edges. Such blobs 
can be interpreted as higher-level analogues of parallel edges.

The \emph{generator}~$G(N)$ of a network~$N$ is the multi-graph obtained by 
deleting all pendant subtrees (i.e., deleting all leaves from~$N$) and 
suppressing degree-$2$ vertices. The generator may contain loops and parallel 
edges. 
A vertex of~$N$ that is not deleted or suppressed in the process of 
obtaining~$G(N)$ is called a \emph{generator vertex}.
We call the edges 
of~$G(N)$ the \emph{sides} of~$N$. Observe that the sides of~$N$ correspond to 
paths of~$N$. Let~$s$ be a side of~$N$, and 
let~$e_0v_1v_2\cdots v_ke_1$ with~$k\ge0$ denote the path in~$N$ corresponding 
to~$s$, 
where~$e_0$ and~$e_1$ are vertices of the generator. If~$k=0$, then the path is 
simply the edge~$e_0e_1$. We call~$e_0$ and~$e_1$ 
the \emph{boundary vertices of side~$s$}. We say that a leaf~$x$ \emph{is on} 
side~$s$ if~$x$ is a neighbor of~$v_i$ for some~$i\in[k]$. We say that a 
chain \emph{is on} side~$s$ if all leaves of the chain are on the side~$s$. 
Observe that a leaf of a chain is on a side if and only if the chain is on 
the side. Observe also that~$v_1\cdots v_k$ is a spine of some chain on~$s$.
We say that a side is \emph{empty} if no leaves are on the side. A 
\emph{side of a blob~$B$} is an edge of~$G(N)$ which corresponds to 
a path in~$B$. Observe that level-$2$ blobs contain exactly two vertices that 
are not cut-vertices. We call these the \emph{poles} of the blob. There are 
exactly three edge-disjoint paths between the two poles. We call these three paths 
in~$N$ the \emph{main paths} of~$B$. The vertices in a main path~$s$ of~$B$ 
that are adjacent to the endpoints of~$s$ are called the \emph{main end-spine 
	vertices}. 

We adopt the following notation for pendant level-$2$ blobs 
from~\citep{van2020reconstructibility}. Let~$B$ be a pendant level-$2$ blob, 
and let~$a,b,c,d$ denote the four chains contained in~$B$ of 
lengths~$k,\ell,m,n \geq0$, respectively, such that chains~$c$ and~$d$ are on 
the same main path of~$B$ as the non-trivial cut-edge. Then we say that~$B$ is 
of 
the form~$(a,b,c,d)$ (see Figure~\ref{fig:PBlobs}). The order of the first two 
elements~$a,b$, and the 
order of the last two elements~$c,d$ do not matter. For ease of 
notation, a side without leaves is seen as a length-0 chain. Note that since 
every cut-edge induces a unique split, it is not possible to obtain the pendant 
blob of the form~$(1,0,0,0)$. 

\begin{figure}
	\centering
	\includegraphics[width=\textwidth]{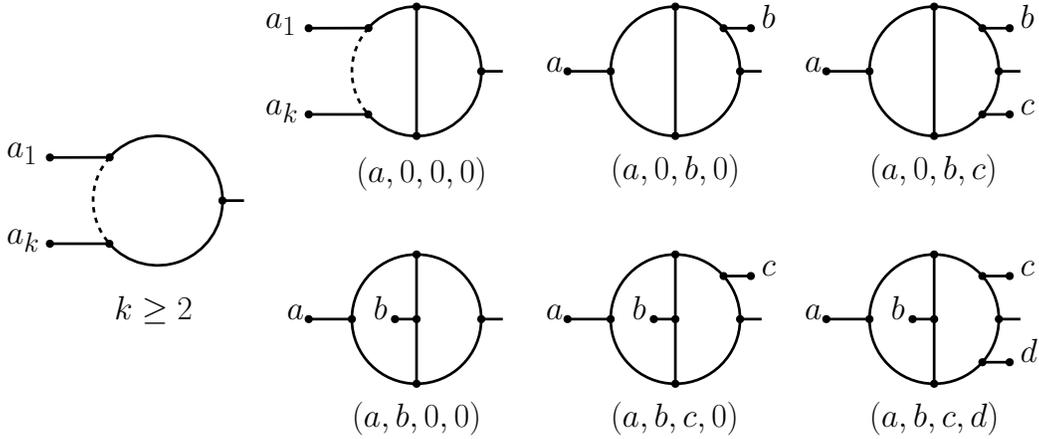}
	\caption{The seven possible pendant blobs in a level-$2$ network. The 
		leftmost pendant blob is level-$1$, and the other pendant blobs are all 
		level-$2$. Observe that for the pendant level-$1$ blob and the pendant 
		level-$2$ blob of the form~$(a,0,0,0)$, the length of the chain~$a$ 
		must be 
		of length at least~$2$. This is due to the fact that networks do not 
		contain parallel edges nor level-$2$ blobs with only two cut-edges 
		incident 
		to it. The dashed edges in these two blobs indicate that the chain can 
		be 
		longer than~$2$. For the five other pendant blobs, each of the 
		leaves~$a,b,c,d$ indicates a chain.
	}
	\label{fig:PBlobs}
\end{figure}
 
\begin{figure}[ht!]
	\centering
	\begin{minipage}[c]{.49\textwidth}
		\resizebox{\textwidth}{!}{
			\begin{tikzpicture}[every node/.style={draw, circle, fill, inner 
			sep=0pt}]
			\tikzset{edge/.style={very thick}}
			
			\draw[black, very thick, dashed, domain=0:90] plot 
			({1.5*cos(\x)},{1.5*sin(\x)});
			\draw[black, very thick, domain=90:360] plot 
			({1.5*cos(\x)},{1.5*sin(\x)});
			\draw[black, very thick] (4,0) circle (1.5);
			
			\node[] (v0)    at (0,0)            {a};
			\node[] (v1)    at (-1.5,0)    {a};
			\node[] (v2)    at (1.5,0)          {a};
			\node[] (cut) at (2,0) {a};
			\node[] (cut1) at (2,-0.75) {a}; 
			\node[] (f) at (1.5,-1.5) {a}
			node[draw=none,fill=none,below=1mm of f] {\large $f$};
			\node[] (g) at (2.5,-1.5) {a}
			node[draw=none,fill=none,below=1mm of g] {\large $g$};
			\node[] (v3)    at (2.5,0)          {a};
			\node[] (v4)    at (0,1.5)          {a}
			node[draw=none,fill=none,above=1mm of v4] {\large $p_1$};
			\node[] (v5)    at (0,-1.5)         {a}
			node[draw=none,fill=none,below=1mm of v5] {\large $p_2$};
			\node[] (a)     at (-2.25,0)         {a};
			\node[] (c)     at (-0.75,0)        {a};
			\node[] (d1) at ({4 + 1.5*cos(30)},{1.5*sin(30)}) {a};
			\node[] (d2) at ({4 + 1.5*cos(-30)},{1.5*sin(-30)}) {a};
			\node[] (c1) at ({1.5*cos(60)},{1.5*sin(60)}) {a};
			\node[] (c2) at ({1.5*cos(30)},{1.5*sin(30)}) {a};

			\node[] (lc1) at ({0.5 + 1.5*cos(60)},{1.5*sin(60)}) {a}
			node[draw=none,fill=none,right=1mm of lc1] {\large $c_1$};
			\node[] (lc2) at ({0.5 + 1.5*cos(30)},{1.5*sin(30)}) {a}
			node[draw=none,fill=none,right=1mm of lc2] {\large $c_2$};
			\node[] (ld1) at ({5 + 1.5*cos(30)},{1.5*sin(30)}) {a}
			node[draw=none,fill=none,right=1mm of ld1] {\large $d_1$};
			\node[] (ld2) at ({5 + 1.5*cos(-30)},{1.5*sin(-30)}) {a}
			node[draw=none,fill=none,right=1mm of ld2] {\large $d_2$};
			\node[draw = none, fill = none, left=1mm of a] (la) {\large $a$};
			\node[draw = none, fill = none, below=1mm of c] (lc) {\large $b$};
			
			\draw[edge] (v0) -- (c)
						(v1) -- (a)
						(v2) -- (v3)
						(cut1) -- (f)
						(cut1) -- (g)
						(c1) -- (lc1)
						(c2) -- (lc2)
						(d1) -- (ld1)
						(d2) -- (ld2)
						(v4) -- (v5);
			\draw[edge,dotted] (cut) -- (cut1)
			node[draw=none,fill=none,midway, right=1mm] {\large $e$};
			
			\node[draw=none,fill=none] at (-1.5,1.5) {\LARGE $N$};
			
			\end{tikzpicture}
		}
	\end{minipage}\hfill
	\begin{minipage}[c]{.49\textwidth}
		\resizebox{0.8\textwidth}{!}{
			\begin{tikzpicture}[every node/.style={draw, circle, fill, inner 
				sep=0pt}]
			\tikzset{edge/.style={very thick}}
			
			\draw[black, very thick] (0,0) circle (1.5);
			\draw[black, very thick] (4,0) circle (1.5);
			
			\node[] (v2)    at (1.5,0)          {a};
			\node[] (v3)    at (2.5,0)          {a};
			\node[] (v4)    at (0,1.5)          {a};
			\node[] (v5)    at (0,-1.5)         {a};
			
			\draw[edge] 
			(v2) -- (v3)
			(v4) -- (v5);
			
			\node[draw=none,fill=none] at (-1.6,1.5) {\LARGE $G(N)$};
			\end{tikzpicture}
		}
	\end{minipage}\hfill
	\vspace{1cm}
		\centering
		\small
		\resizebox{\columnwidth}{!}{
			\begin{tabular}{c|c|c|c|c|c|c|c|c}
				& $a$ & $b$     & $c_1$   & $c_2$   & $d_1$    & $d_2$    & 
				$f$        & $g$        \\ \hline
				$a$   &     & $(4,8)$ & $(4,8)$ & $(5,7)$ & $(7,12)$ & $(7,12)$ 
				& $(6,10)$ & $(6,10)$ \\ \hline
				$b$   &     &         & $(4,8)$ & $(5,7)$ & $(7,12)$ & $(7,12)$ 
				& $(6,10)$ & $(6,10)$ \\ \hline
				$c_1$ &     &         &         & $(3,7)$ & $(7,10)$ & $(7,10)$ 
				& $(6,8)$  & $(6,8)$  \\ \hline
				$c_2$ &     &         &         &         & $(6,11)$ & $(6,11)$ 
				& $(5,9)$  & $(5,9)$  \\ \hline
				$d_1$ &     &         &         &         &          & $(3,4)$  
				& $(5,6)$  & $(5,6)$  \\ \hline
				$d_2$ &     &         &         &         &          &          
				& $(5,6)$  & $(5,6)$  \\ \hline
				$f$   &     &         &         &         &          &          
				&          & $(2,2)$  \\ \hline
				$g$   &     &         &         &         &          &          
				&          &         
			\end{tabular}
		}

	\caption{A level-$2$ network~$N$ on the taxa 
	set~$\{a,b,c_1,c_2,d_1,d_2,f,g\}$, its generator~$G(N)$, and its 
	sl-distance matrix. $N$ contains a cherry~$\{f,g\}$ and four chains~$(a), 
	(b), (c_1,c_2)$ and~$(d_1,d_2)$. $N$ contains two pendant blobs: the 
	leftmost is 
	a level-$2$ blob of the form~$((a),(b),(c_1,c_2),\emptyset)$, and the 
	rightmost is a level-$1$ blob containing the leaves~$d_1$ and~$d_2$. The 
	poles of the pendant level-$2$ blob are labelled by~$p_1$ and~$p_2$. The 
	dotted
	cut-edge~$e$ induces the non-trivial 
	split~$\{a,b,c_1,c_2,d_1,d_2\}|\{f,g\}$. The blob side indicated by the 
	dashed path contains the 
	chain~$(c_1,c_2)$. The chains~$(a)$ and~$(c_1,c_2)$ are adjacent once. The 
	chains~$(a)$ and~$(b)$ are adjacent twice.
	The sl-matrix has~$ij$-th elements of the form~$(x,y)$, where~$x$ and~$y$ 
	denote the shortest and longest distances between~$i$ and~$j$ in~$N$. The 
	diagonal elements, which are all~$(0,0)$, and the lower triangular elements 
	are omitted as the matrix is symmetric.}
	\label{fig:Example}
\end{figure}



\subsection{Distances}

For a network~$N$ on~$X$, we let~$d^N_m(x,y)$ and~$d^N_l(x,y)$ denote the 
length of a shortest and a longest path between two vertices~$x,y$ in~$N$, 
respectively. We exclude the superscript~$N$ when there is no 
ambiguity on the network at hand. Let~$a=\{a_1,\ldots,a_k\}$ be a set of 
vertices in~$N$, and let~$u$ be a vertex in~$N$ that is not in~$a$. Then we 
define the shortest distance from~$u$ to~$a$ as the shortest distance from~$u$ 
to any of the vertices in~$a$, that is,~$d_m^N(a,u) = 
\min\{d_m(a_i,u):i\in[k]=\{1,\ldots, k\}\}$. Similarly, 
define the longest distance from~$u$ to~$a$ as the longest distance from~$u$ to 
any of the vertices in~$a$, that is,~$d_l^N(a,u) = \max\{d_l(a_i,u): i\in[k]\}$.

The \emph{shortest distance matrix}~$\D_m(N)$ of~$N$
is the~$|X|\times |X|$ matrix, \rev{where the rows and columns are indexed by the leaves of the network,} whose~$(x,y)$-th entry is~$d^N_m(x,y)$. A 
network~$N$ \emph{realizes} the shortest distance matrix~$\D_m$ 
if~$\D_m(N) = \D_m$. We say that a network~$N$ is 
\emph{reconstructible} from its shortest distance matrix if~$N$ is the only 
network, up to isomorphism,
that realizes~$\D_m(N)$.
Here, we say that two networks~$N$ and~$N'$ on~$X$ are \emph{isomorphic} if 
there exists a bijection~$f$ from the vertices of~$N$ to the vertices of~$N'$, 
such that~$uv$ is an edge of~$N$ if and only if~$f(u)f(v)$ is an edge of~$N'$, 
and the leaves of~$N$ are mapped to leaves of~$N'$ of the same label.
Similarly, we define the \emph{sl-distance matrix (shortest longest - distance 
matrix)}~$\D(N)$ as 
the~$|X|\times |X|$ matrix, \rev{where the rows and columns are indexed by the leaves of the network,} whose~$(x,y)$-th entry is~$d^N(x,y) = \{d^N_m(x,y), 
d^N_l(x,y)\}$. We say that a network~$N$ realizes the sl-distance 
matrix~$\D$ if~$\D(N) = \D$. A network~$N$ is 
\emph{reconstructible} from its sl-distance matrix if~$N$ is the only network, 
up to isomorphism, 
that realizes~$\D(N)$. 


\subsection{Reducing Cherries}\label{subsec:Che}

By definition, we may identify cherries from shortest distance matrices.

\begin{observation}
	Let~$\D_m$ be a shortest distance matrix. A network~$N$ on~$X$ 
	that realizes~$\D_m$ contains a cherry~$\{x,y\}$ if and only if~$d_m(x,y) = 
	2$.
\end{observation}

\emph{Reducing} a cherry~$\{x,y\}$ to a leaf~$z$ from~$N$ is the action of 
deleting both leaves~$x,y$ and labelling the remaining unlabelled degree-$1$ 
vertex as~$z$, assuming that~$z\notin X$ (this vertex was the neighbor of~$x$ 
and~$y$ in~$N$). As a result of reducing the cherry~$\{x,y\}$, observe that the 
shortest distance between two leaves that are both not~$z$ are unchanged; the 
shortest distance between~$z$ and another leaf~$l\in X -\{x,y\}$ is exactly one 
less than that of~$x$ and~$l$ in~$N$. 

\begin{observation}\label{obs:CherryReduction}
	Let~$N$ be a network on~$X$ containing a cherry~$\{x,y\}$. Upon reducing 
	the cherry to a leaf~$z$, we obtain a network~$N'$ 
	on~$X'=X\cup\{z\}-\{x,y\}$ such that 
	the shortest distance matrix for~$N'$ contains the elements
	\[d^{N'}_m(a,b) =
	\begin{cases}
		d^N_m(a,b) &\text{ if } a,b\in X-\{x,y\}\\
		d^N_m(a,x)-1 &\text{ if } a\in X-\{x,y\} \text{ and } b = z.
	\end{cases}
	\]
\end{observation}

In the setting of Observation~\ref{obs:CherryReduction}, one may obtain a 
network that is isomorphic to~$N$ from~$N'$ by adding two labelled vertices~$x$ 
and~$y$, adding the edges~$zx$ and~$zy$, and unlabelling the vertex~$z$. We 
call this \emph{replacing~$z$ by a cherry~$\{x,y\}$}. 

\begin{observation}\label{obs:AddCherry}
	Let~$N$ be a network on~$X$, and let~$z$ be a leaf in~$N$. Let~$x,y\notin 
	X$ be leaf labels that do not appear in~$N$. Then upon replacing~$z$ by a 
	cherry~$\{x,y\}$, we obtain a network~$M$ on~$Y=X\cup\{x,y\}-\{z\}$ that 
	realizes the shortest distance matrix with entries
	\[d^{M}_m(a,b) =
	\begin{cases}
		d^N_m(a,b) 		&\text{ if } a,b\in Y-\{x,y\}\\
		d^N_m(a,z)+1 	&\text{ if } a\in Y-\{x,y\} \text{ and } b\in\{x,y\}\\
		2				&\text{ if } a=x\text{ and }b=y.
	\end{cases}
	\]
\end{observation}

It is easy to see that replacing a leaf by a cherry and reducing a cherry are 
inverse operations of one another.

\begin{lemma}\label{lem:CherryReductionRecon}
	Let~$N$ be a network with a cherry~$\{x,y\}$, and let~$N'$ denote the 
	network obtained by reducing the cherry from~$N$ to a leaf~$z$. Then $N$ is 
	reconstructible from its shortest distance matrix if and only if~$N'$ is 
	reconstructible from its shortest distance matrix.
\end{lemma}
\begin{proof}
	Suppose first that the network~$N$ is reconstructible from its shortest 
	distance matrix. Suppose for a contradiction that the shortest distance 
	matrix~$\D_m(N')$ of~$N'$ is also realized by a network~$N''$ that is not 
	isomorphic to~$N'$. 
	Consider the networks~$M'$ and~$M''$ obtained from~$N'$ and~$N''$, 
	respectively, by replacing~$z$ by a cherry~$\{x,y\}$. By 
	Observation~\ref{obs:AddCherry}, the two distinct networks~$M'$ and~$M''$
	realize the same shortest distance matrix. However, this shortest distance 
	matrix is precisely~$\D_m(N)$, since~$M'$ is isomorphic to~$N$. This 
	contradicts the fact that~$N$ is reconstructible from its shortest distance 
	matrix. Therefore~$N'$ must be reconstructible from its shortest distance 
	matrix.
	
	Now suppose that the network~$N'$ is reconstructible from its shortest 
	distance matrix. If there were two distinct networks~$N$ and~$M$ 
	realizing~$\D_m(N)$, then these networks must both contain the 
	cherry~$\{x,y\}$. Reducing this cherry to a leaf~$z$, we see by 
	Observation~\ref{obs:CherryReduction} that both reduced networks, which are 
	distinct, realize the same shortest distance matrix, which is 
	exactly~$\D_m(N')$. However, this is not possible, as~$N'$ is 
	reconstructible from its shortest distance matrix. Therefore~$N$ is also 
	reconstructible from its shortest distance matrix.
\end{proof}

Let~$N$ be a network. \emph{Subtree reduction} refers to the action of reducing 
cherries of~$N$ until it is no longer possible to do so. We refer to the 
resulting network as the \emph{subtree reduced version} of~$N$. Note that the 
subtree reduced version of~$N$ is unique, and the order in which the cherries 
are reduced does not matter. 
The following corollary follows 
immediately by applying Lemma~\ref{lem:CherryReductionRecon} to every cherry 
that is reduced in the subtree reduction.

\begin{corollary}\label{cor:SubtreeReduction}
	A network~$N$ is reconstructible from its shortest distance matrix if and 
	only if the subtree reduced version of~$N$ is reconstructible from its 
	shortest distance matrix.
\end{corollary}

Note that Observations~\ref{obs:CherryReduction} and~\ref{obs:AddCherry}, 
Lemma~\ref{lem:CherryReductionRecon}, and Corollary~\ref{cor:SubtreeReduction} 
can 
naturally be extended to the sl-distances, with a single tweak for 
Observations~\ref{obs:CherryReduction} and~\ref{obs:AddCherry}, where 
the longest distances are adjusted exactly the same as done for the shortest 
distances (replace~$d_m$ by~$d_l$ wherever possible). This means we may assume 
for the rest of the paper, that all networks have undergone subtree reduction, 
and therefore that all networks contain no cherries.
\medskip

\subsection{Chains}\label{subsec:Chain}

Upon reducing all cherries from our networks, we may identify unique chains 
from shortest distance matrices. Recall that chains are written as 
sequences~$a=(a_1,\ldots,a_k)$ for some~$k\ge1$. We shall sometimes write these 
as~$(a,k)$. In 
what follows, we will often require a way of 
referring to leaves of the network that are not in a particular chain. So 
while~$a$ is a sequence of leaves, we shall sometimes treat~$a$ as a set of 
leaves, e.g., $X-a = \{l\in X: l\ne a_i \text{ for } i\in[k]\}$.

\begin{observation}\label{obs:Chain}
	Let~$\mathcal{D}_m$ be a shortest distance matrix. A network~$N$ on~$X$ 
	that realizes~$\mathcal{D}_m$ contains a chain~$a=(a_1,\ldots,a_k)$ 
	where~$k\ge1$ if and 
	only if~$d_m(a_i,a_{i+1}) = 3$ for all~$i\in [k-1]$ and there exists no 
	leaf~$l\in X-a$ such that~$d_m(a,l) =3$.
\end{observation}

Observation~\ref{obs:Chain} implies that the leaves in a network without 
cherries can be partitioned into chains. Indeed, no leaf can be contained in 
two distinct chains, as otherwise the chains would be non-maximal.
 Let~$(a,k)$ and~$(b,\ell)$ be two 
distinct chains. We say that~$(a,k)$ and~$(b,\ell)$ are \emph{adjacent} 
if~$d_m(a_i,b_j) = 
4$ for some combination of~$i\in\{1,k\}$ and~$j\in\{1,\ell\}$. Observe that 
adjacent chains of a network~$N$ can be identified from shortest distance 
matrices, by 
first partitioning the leaf set of~$N$ into chains and then checking for chain 
end-leaves that are distance-$4$ apart. 
\rev{We say that two chains~$(a,k)$ and~$(b,\ell)$ are \emph{adjacent once} if exactly one distinct
pair of~$(a,k)$ and~$(b,\ell)$ are distance-$4$ apart.}
We say that the chains~$(a,k)$ 
and~$(b,\ell)$ are 
\emph{adjacent twice} if two \rev{distinct} pairs of~$(a,k)$ and~$(b,\ell)$ 
end-leaves are distance-$4$ apart. Since we assume networks to be binary, two 
chains may be adjacent at most twice.
\rev{In the special case where~$k=\ell=1$, we can only tell whether the chains are adjacent from the shortest distances. We cannot tell whether they are adjacent twice. This can however be inferred from the sl-distance matrix.}

\subsection{Known results}

The following results appeared in~\cite{van2020reconstructibility}.

\begin{lemma}[Theorem 4.2; Lemma 
4.4; Lemma 5.3 of \citep{van2020reconstructibility}]\label{lem:L1}
	Let~$N$ be a level-$2$ network on~$|X|$. Then~$N$ is reconstructible from 
	its shortest distance matrix if~$N$ is also level-$1$, if~$|X|<4$, or 
	if~$N$ contains only one blob.
\end{lemma}


Therefore we may assume that the networks we consider are always at least 
level-$2$ on at least four leaves, and that the network contains at least two 
blobs. Furthermore, from Section~\ref{subsec:Che}, we may assume that the 
networks contain no cherries.

\section{Leaf on each generator side}\label{sec:GeneratorSide}

In this section, we consider networks with at least one leaf on each generator 
side, and show that such networks are reconstructible from their 
shortest distance matrices, regardless of level. Let~$N$ be one such network.
Since we may assume that~$N$ has no cherries, each side of~$N$ can be 
determined 
by the chain contained therein. Furthermore, two sides are adjacent (i.e., the 
sides share a common endpoint in~$G(N)$) if and only if the chains on the sides 
are adjacent. Since chains partition the leaf set of~$N$, this implies 
that the structure of the generator~$G(N)$, and therefore the structure of 
the network~$N$ is determined by the chains of~$N$ and their adjacency in~$N$.


%

Every vertex in~$N$ is either a leaf, a spine vertex of some chain, or a 
generator vertex.
Since networks considered here are binary, exactly two or three generator sides 
may be incident to the same vertex in~$G(N)$ (as per conventional graph theory, 
we say 
that an edge is \emph{incident} to its endpoints). 
If a vertex is incident to exactly two sides in~$G(N)$, then one of these sides 
must be a loop. Loops in~$G(N)$ correspond to pendant level-$1$ blobs in~$N$. 
Suppose that~$(a,k)$ is a chain (recall that this is a chain of length~$k$) 
and is adjacent to exactly one chain~$(b,\ell)$ 
twice, and 
that~$(a,k)$ is not adjacent to any other chains. Then~$(a,k)$ is contained 
in a pendant level-$1$ blob, since we may assume that~$N$ is a level-$2$ 
network 
with at least two blobs. Note that~$k\geq2$ as~$N$ contains no parallel edges. 
In such a case, we call the pair~$(a,b)$ the \emph{bulb} of~$a$ and~$b$. We 
say that~$a$ is contained in the bulb as the \emph{petal}. We say 
that~$N$ \emph{contains} the petal~$(a,b)$.

If a generator 
vertex is incident to three sides, then the three distinct
chains in the network, corresponding to these three 
sides must be pairwise adjacent. Now 
consider three pairwise adjacent distinct chains~$(a,k), (b,\ell), (c,m)$ 
in~$N$. Since we may assume~$N$ is not a level-$2$ network with a single 
blob (as we know such networks are reconstructible from their shortest 
distances by Lemma~\ref{lem:L1}), any three chains may be pairwise adjacent at 
most once. In particular, 
the end-leaves of~$a,b,c$ that are adjacent are unique. In this case, we 
say that~$(a,b,c)$ forms a \emph{pairwise adjacent triple} (see 
Figure~\ref{fig:BuPe} for examples of pairwise adjacent triples and 
petals). 
\rev{Therefore, if $(a,b,c)$ is a pairwise adjacent triple, then there is a generator vertex that is incident to three sides such that one side contains chain $a$, one chain $b$ and one chain $c$.}
We say that~$N$ \emph{contains} the pairwise adjacent 
triple~$(a,b,c)$.

\begin{figure}
	\centering
	\begin{tikzpicture}[every node/.style={draw, circle, fill, inner 
		sep=0pt}]
		\tikzset{edge/.style={very thick}}
			
		\node[] (origin)	at	(0,0)	{a};
		\node[] (north) 	at 	(0,2)	{a};
		\node[] (sw)		at	({2*cos(-135)},{2*sin(-135)})	{a};
		\node[] (se)		at	({2*cos(-45)},{2*sin(-45)})		{a};
		
		\draw[black, very thick, domain=90:140] plot 
		({2*cos(\x)},{2*sin(\x)});
		\draw[black, very thick, domain=140:175,dashed] plot 
		({2*cos(\x)},{2*sin(\x)});
		
		\node[] (a1)		at	({2*cos(140)},{2*sin(140)})		{a};
		\node[]	(a2)		at	({2*cos(175)},{2*sin(175)})		{a};
		\node[above left = 3mm of a1] (la1) {a}
		node[draw=none,fill=none,below left=2mm and 3mm of la1] {\large $a$};
		\node[above left = 3mm of a2] (la2) {a};
		
		\draw[edge]	(a1) -- (la1)
					(a2) -- (la2);

		\draw[black, very thick, domain=175:252.5] plot 
		({2*cos(\x)},{2*sin(\x)});
		\draw[black, very thick, domain=252.5:287.5,dashed] plot 
		({2*cos(\x)},{2*sin(\x)});
					
		\node[] (d1)		at	({2*cos(252.5)},{2*sin(252.5)})		{a};
		\node[]	(d2)		at	({2*cos(287.5)},{2*sin(287.5)})		{a};
		\node[below = 3mm of d1] (ld1) {a}
		node[draw=none,fill=none,right= 3mm of ld1] {\large 
		$d$};
		\node[below= 3mm of d2] (ld2) {a};		
		
		\draw[edge]	(d1) -- (ld1)
					(d2) -- (ld2);

		\draw[black, very thick, domain=62.5:90] plot 
		({2*cos(\x)},{2*sin(\x)});
		\draw[black, very thick, domain=27.5:62.5,dashed] plot 
		({2*cos(\x)},{2*sin(\x)});
		
		\node[] (f1)		at	({2*cos(62.5)},{2*sin(62.5)})		{a};
		\node[]	(f2)		at	({2*cos(27.5)},{2*sin(27.5)})		{a};
		\node[above right = 3mm of f1] (lf1) {a}
		node[draw=none,fill=none,below right=1mm and 3mm of lf1] {\large $f$};
		\node[above right = 3mm of f2] (lf2) {a};		
		
		\draw[edge]	(f1) -- (lf1)
					(f2) -- (lf2);

		\draw[black, very thick, domain=-72.5:-30] plot 
		({2*cos(\x)},{2*sin(\x)});
		\draw[black, very thick, domain=-30:-15,dashed] plot 
		({2*cos(\x)},{2*sin(\x)});
		
		\node[] (g1)		at	({2*cos(-15)},{2*sin(-15)})		{a};
		\node[]	(g2)		at	({2*cos(-30)},{2*sin(-30)})		{a};
		\node[below right = 3mm of g1] (lg1) {a}
		node[draw=none,fill=none,below right=2mm and 1mm of lg1] {\large $g$};
		\node[below right = 3mm of g2] (lg2) {a};		
		
		\draw[edge]	(g1) -- (lg1)
					(g2) -- (lg2);
					
		\draw[black, very thick, domain=-15:27.5] plot 
		({2*cos(\x)},{2*sin(\x)});
					
		\node[]	(b1)	at	(0,{2/3})	{a};
		\node[] (b2) 	at	(0,{4/3})	{a};
		\node[draw=none,fill=none] (b)	at (-1,1) {\large $b$};
		\node[left = 3mm of b1] (lb1)	{a};
		\node[left = 3mm of b2]	(lb2)	{a};
		
		\draw[edge]	(b1) -- (lb1)
					(b2) -- (lb2)
					(north) -- (b2)
					(b1) -- (origin);
		\draw[edge, dashed] (b1) -- (b2);
		
		\node[]	(c1)	at	({2*cos(225)/3},{2*sin(225)/3})	{a};
		\node[] (c2) 	at	({4*cos(225)/3},{4*sin(225)/3})	{a};
		\node[below right = 3mm of c1] 	(lc1)	{a}
		node[draw=none,fill=none,below =2mm of lc1] {\large $c$};
		\node[below right = 3mm of c2]	(lc2)	{a};
		
		\draw[edge]	(c1) -- (lc1)
					(c2) -- (lc2)
					(origin) -- (c1)
					(c2) -- (sw);
		\draw[edge,dashed] (c1) -- (c2);
		
		\node[]	(e1)	at	({2*cos(-45)/3},{2*sin(-45)/3})	{a};
		\node[] (e2) 	at	({4*cos(-45)/3},{4*sin(-45)/3})	{a};
		\node[above right = 3mm of e1] 	(le1)	{a};
		\node[above right = 3mm of e2]	(le2)	{a}
		node[draw=none,fill=none,above =1.2mm of le2] {\large $e$};
		
		\draw[edge]	(e1) -- (le1)
					(e2) -- (le2)
					(origin) -- (e1)
					(e2) -- (se);
		\draw[edge,dashed] (e1) -- (e2);
		
		\node[] (east) at (2,0)	{a};
		\node[] (west) at (6,0) {a};

		\draw[black, very thick, domain=20:340] plot 
		({8+2*cos(\x)},{2*sin(\x)});
		\draw[black, very thick, domain=-20:20,dashed] plot 
		({8+2*cos(\x)},{2*sin(\x)});
		
		\node[]	(i1)	at	({8+2*cos(20)},{2*sin(20)})		{a};
		\node[] (i2) 	at	({8+2*cos(-20)},{2*sin(-20)})	{a};
		\node[right = 3mm of i1] 	(li1)	{a}
		node[draw=none,fill=none,below =3.5mm of li1] {\large $i$};
		\node[right = 3mm of i2]	(li2)	{a};
		
		\draw[edge]	(i1) -- (li1)
					(i2) -- (li2);
					
		\node[]	(h1)	at	(3.5,0)		{a};
		\node[] (h2) 	at	(4.5,0)		{a};
		\node[below = 3mm of h1] 	(lh1)	{a}
		node[draw=none,fill=none,right=2mm of lh1] {\large $h$};
		\node[below = 3mm of h2]	(lh2)	{a};
		
		\draw[edge]	(h1) -- (lh1)
					(h2) -- (lh2)
					(east) -- (h1)
					(h2) -- (west);
		\draw[edge,dashed] (h1) -- (h2);
		
	\end{tikzpicture}
	\caption{A level-$3$ network with leaves on every generator side. Each 
	letter represents a chain on a side of the network. The network 
	contains the pairwise adjacent triples $(a,b,f), (a,c,d), (b,c,e), (d,e,g), 
	(f,g,h)$ and the bulb~$(i,h)$ with~$i$ as the petal. The dashed edges 
	indicate how each of the nine chains is of length at least~$1$.}
	\label{fig:BuPe}
\end{figure}
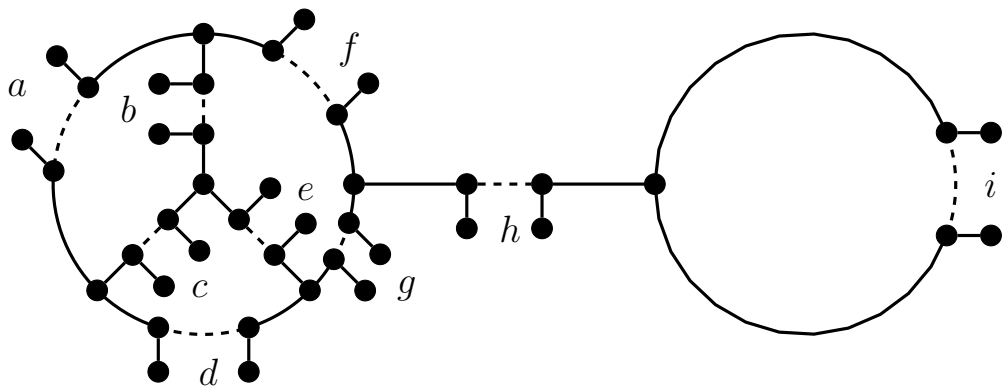

Note that bulbs and pairwise adjacent triples both consist of three leaves. In 
what follows, the notion of a \emph{median} vertex will be important. Given 
three 
vertices~$a,b,c$ of a network, a median of~$a,b,c$ is a vertex that belongs to 
a shortest path between each pair of~$a,b,c$. \rev{A median may not always exist for 
any three vertices (consider, for example, a cycle on three vertices). However, for
our purposes, we shall consider medians of three leaves, which will always exist.}
\rev{Moreover}, a median of three vertices is not 
necessarily unique, as there may be more than one shortest path between a pair 
of vertices in a network.

\begin{lemma}\label{lem:LESMatrixType}
	Let~$\D_m$ be a shortest distance matrix. Then~$\D_m$ can only be realized 
	by a network~$N$ with leaves on each side of the generator, where~$N$ is 
	not a single level-$2$ blob if and only if 
	each chain of~$\D_m$ is contained in either
	\begin{enumerate}[label=(\roman*)]
		\item two distinct pairwise adjacent triples; or
		\item one pairwise adjacent triple and one bulb as a non-petal; or
		\item one bulb as the petal; or
		\item two bulbs as non-petals.
	\end{enumerate} 
\end{lemma}
\begin{proof}
	Let~$N$ be a network with leaves on each generator side and suppose 
	that~$N$ realizes~$\D_m$. 
	Then each generator vertex of~$N$ is the median 
	of three 
	end-leaves of (not necessarily distinct) chains, such that these end-leaves 
	are pairwise shortest distance-$4$ apart. Any three chains may be pairwise 
	adjacent at most once (unless~$N$ is a network with a single level-$2$ 
	blob, but we have specifically excluded this case in the statement of the 
	lemma), and a chain contained in a pendant 
	level-$1$ blob is 
	adjacent twice to exactly one other chain. As~$N$ is binary, 
	these median vertices encode either pairwise adjacent triples or 
	bulbs. By \emph{encode}, we mean that for every three end-leaves that are 
	pairwise distance-$4$ apart, the median vertex corresponding to it is 
	unique. Each chain is contained in 
	exactly two of such constructs, 
	where being contained in a bulb as the petal counts as two, since each side 
	has two boundary vertices (except for the loop). The result follows 
	immediately.
	
	To show the other direction of the lemma, we prove the contrapositive. 
	Let~$N$ be a network that has at least one empty generator side, and 
	suppose that~$N$ realizes~$\D_m$. We want to show that at least one chain 
	of~$N$ does not satisfy any of the four properties $(i) - (iv)$ as stated 
	in 
	the statement of the lemma.
	Find adjacent sides~$s_1$ and~$s_2$ of~$N$, such that $s_1$ contains a 
	chain~$c$ while~$s_2$ is empty. 
	Clearly,~$c$ cannot be contained in a bulb as its petal, since~$s_2$ 
	contains no chains to which~$c$ can be adjacent twice (cannot 
	satisfy~$(iii)$). So we may 
	assume that~$c$ 
	is not contained in a pendant level-$1$ blob, and therefore that the 
	boundary vertices~$e_0,e_1$ of~$s_1$ are distinct. We may assume without 
	loss of generality that~$e_0$ is the boundary vertex of~$s_2$. Since~$s_2$ 
	is empty,~$e_0$ cannot be a median of three distinct end-leaves of chains. 
	This implies that~$c$ can only be contained in exactly one pairwise 
	adjacent 
	triple, or in exactly one bulb as a non-petal (which is encoded by~$e_1$) 
	(cannot 
	satisfy $(i),(ii)$, nor~$(iv)$).
\end{proof}


\begin{lemma}\label{lem:CMV}
	Let~$N$ and~$N'$ be networks with a leaf on each generator side, such that 
	neither~$N$ nor~$N'$ are level-$2$ and contain precisely one level-$2$ 
	blob. Then~$N$ and~$N'$ are isomorphic if and only if they contain the same 
	chains, the same pairwise adjacent triples, and the same bulbs, \rev{where it is known which end-leaves of the chains are adjacent}.
\end{lemma}
\begin{proof}
	Suppose first that~$N$ and~$N'$ contain the same chains, 
	the same pairwise adjacent triples, and the same bulbs. Then the networks 
	must 
	contain the same leaves and the same spine vertices (and also the edges 
	therein). The remaining vertices in~$N$ and~$N'$ are their generator 
	vertices, and the remaining edges are those incident on generator vertices 
	and the end-spine vertices. 
	
	\rev{We show first that~$G(N) = G(N')$.
	Every edge in the generator is a side that contains a chain.
	Since~$N$ and~$N'$ have the same chains, the number of edges in~$G(N)$ is the same as that in~$G(N')$.
	Every vertex in the generator is a median of end-leaf vertices of three (not-necessarily distinct) chains. 
	These generator vertices uniquely encode a pairwise adjacent triple or a bulb, since~$N$ 
	and~$N'$ are not level-$2$ networks that contain precisely one 
	level-$2$ blob.
	Since~$N$ and~$N'$ have the same pairwise adjacent triples and the same bulbs,~$G(N)$ and~$G(N')$ must have the same number of vertices.
	To see that~$G(N) = G(N')$, observe that each generator vertex that encodes the pairwise adjacent triple~$(a,b,c)$ or a bulb~$(a,b)$ links the generator edges that contain the chains~$a,b,c$ or~$a,b$, respectively. 
	This means that two generator edges share a common endpoint if and only if the chains that they contain are in the same pairwise adjacent triple or bulb.}

    \rev{To see that~$N$ is isomorphic to~$N'$, simply attach all chains to their corresponding generator sides, noting that the placement of the end-leaves are determined by the composition of the pairwise adjacent triples.
    Because we know which end-leaves of the chains are adjacent, the orientation of the chains are also determined.
    Since~$N$ and~$N'$ contain the same pairwise adjacent triples and bulbs, they must be isomorphic.}
	\medskip
	
	
	Conversely, if two networks are isomorphic, then they must have the same 
	chains, the same pairwise adjacent triples, and the same bulbs.
\end{proof}

\begin{theorem}\label{thm:GenSide}
	Networks with a leaf on each generator side are reconstructible 
	from its shortest distances.
\end{theorem}
\begin{proof}
	If~$N$ is a level-$2$ network with a single blob, then~$N$ is 
	reconstructible from its shortest distances by Lemma~\ref{lem:L1}. 
	Therefore we may assume~$N$ is not a level-$2$ network on a single blob, 
	and therefore we may call Lemmas~\ref{lem:LESMatrixType} and~\ref{lem:CMV}.
	
	Let~$N$ be a network with a leaf on each generator side. This means that 
	every chain in~$N$ satisfies \rev{one of} properties~$(i)-(iv)$ of 
	Lemma~\ref{lem:LESMatrixType}.
	As before, let~$\D_m(N)$ be the shortest distance matrix of~$N$.
	Suppose that~$N'$ is 
	another network that realizes~$\D_m(N)$. 
	Because each chain of~$\D_m(N)$ satisfies one of the four 
	properties~$(i)-(iv)$ of
	Lemma~\ref{lem:LESMatrixType},~$N'$ must be a network with a leaf on each 
	generator side. Furthermore, any network realizing~$\D_m(N)$ must contain 
	the same chains as~$N$, by Observation~\ref{obs:Chain}. Therefore~$N$ 
	and~$N'$ 
	have the same chains. To see that~$N$ and~$N'$ also have the same generator 
	vertices, observe that~$N$ and~$N'$ contain the same pairwise adjacent 
	triples and the same bulbs; these can indeed be inferred from chain 
	adjacencies, which can be inferred from~$\D_m(N)$ by definition of adjacent 
	chains. It follows by Lemma~\ref{lem:CMV} that~$N$ and~$N'$ must be 
	isomorphic.
\end{proof}

\section{Level-2 reconstructibility from sl-distance 
matrix}\label{sec:SLDistance}

As was pointed out in \citep{van2020reconstructibility}, level-$2$ networks are 
in general not reconstructible from their induced shortest distance matrix. 
Figure~\ref{fig:level2CE} illustrates two distinct level-$2$ networks on four 
leaves with the same shortest distance matrices (Figure 2 of 
\cite{van2020reconstructibility}).
In this section, we show that level-$2$ networks are reconstructible from their 
sl-distance matrix. 

\subsection{Cut-edges}\label{subsec:Splits}

First, we show that for a level-$2$ network, we may obtain all the cut-edge 
induced splits from its 
shortest distance matrix. Though the section is concerned with sl-distance 
reconstructibility, we show that the shortest distance matrix suffices in 
obtaining the cut-edge induced splits. 

\begin{theorem}\label{thm:CutEdgeSplit}
	All cut-edge induced splits of a level-$2$ network~$N$ may be obtained from 
	its shortest distance matrix~$\D_m(N)$. A split $A|B$ is induced by a 
	cut-edge of~$N$ if 
	and only if for all~$a,a'\in A$ and~$b,b'\in B$,
	\begin{enumerate}[label=(\roman*)]
		\item $d_m(a,b) + d_m(a',b') = d_m(a,b') + d_m(a',b)$; and
		\item $d_m(a,a') + d_m(b,b') \leq d_m(a,b) + d_m(a',b') - 2$.
	\end{enumerate}
\end{theorem}
\begin{proof}
    \rev{The first statement of the theorem follows from the second statement of the theorem.
    Here, we prove the second statement.}

	Let~$N$ be a level-$2$ network on~$X$. Suppose first that~$A|B$ is a split 
	induced by some cut-edge~$uv$ in~$N$. Let~$a,a'\in A$ and~$b,b'\in B$ be 
	arbitrarily chosen. Since every shortest path from a leaf of~$A$ to a leaf 
	of~$B$ contains the edge~$uv$, we must have that
	\begin{align*}
		d_m(a,b) + d_m(a',b') &= d_m(a,u) + d_m(u,b) + d_m(a',u) + d_m(u,b') \\
		&= d_m(a,b') + d_m(a',b).
	\end{align*}
	So property~$(i)$ holds. Since the length of the edge~$uv$ is~$1$, 
	property~$(ii)$ also holds because
	\begin{align*}
		d_m(a,b) + d_m(a',b') &= d_m(a,u) + 1 + d_m(v,b) + d_m(a',u) + 1 + 
		d_m(v,b')\\
		&\ge d_m(a,a') + d_m(b,b') + 2,
	\end{align*}
	where in particular, we obtain equality if there exist a shortest path 
	between~$a$ and~$a'$ and a shortest path between~$b$ and~$b'$ containing 
	the vertices~$u$ and~$v$, respectively.
	
	Now suppose that~$A|B$ \rev{is a partition of the leaf-set of~$N$, such that}
	properties~$(i)$ and~$(ii)$ hold. Let~$a\in A$, and let~$e=uv$ be a 
	cut-edge in~$N$ that is farthest from~$a$, such that~$e$ induces a split 
	which separates~$a$ from~$B$. If~$e$ induces the split~$A|B$, then we are 
	done. So suppose that there exists an~$a'\in A$ such that~$e$ induces a 
	split that separates~$a$ from~$B\cup\{a'\}$ (in particular, we may assume 
	that~$|A|\ge2$ as every trivial split is clearly induced by a cut-edge). 
	Without loss of generality, 
	suppose that~$u$ is closer to~$a$ than to~$v$. We consider several cases 
	(see Figure~\ref{fig:Lemma1Cases} for an illustration of the cases).
	
	\begin{figure}
		\centering
		\begin{subfigure}[t]{.25\textwidth}
			\centering
			\begin{tikzpicture}[every node/.style={draw, circle, fill, inner 
				sep=0pt, 
				minimum size=2mm}]
			\tikzset{edge/.style={very thick}}
			
			\node[] (a)	at (0,0) {} node[draw = none, fill = none, above=1mm of 
			a]{\large $a$};
			\node[] (u) at (0.5,-0.5) {} node[draw = none, fill = none, 
			above=1mm 
			of 
			u]{\large $u$};
			\node[] (v) at (1,-1) {} node[draw = none, fill = none, above=1mm 
			of 
			v]{\large $v$};
			\node[] (w) at (1.5,-0.5) {} node[draw = none, fill = none, 
			above=1mm 
			of 
			w]{\large $w$};
			\node[] (b) at (2,0) {} node[draw = none, fill = none, above=1mm of 
			b]{\large $b$};
			\node[] (x) at (1,-1.5) {} node[draw = none, fill = none, left=1mm 
			of 
			x]{\large $x$}; 
			\node[] (a') at (0.5,-2) {} node[draw = none, fill = none, left=1mm 
			of 
			a']{\large $a'$};	
			\node[] (b') at (1.5,-2) {} node[draw = none, fill = none, 
			right=1mm of 
			b']{\large $b'$};
			
			\draw[edge, dashed] (a) -- (u) (b) -- (w) (x) -- (a') (x) -- (b');
			\draw[edge] (u) -- (v) -- (w) (v) -- (x);
			\end{tikzpicture}
			\caption{Case 1.}
			\label{subcap:1}
		\end{subfigure}
		
		\begin{subfigure}[t]{.25\textwidth}
			\centering
			\begin{tikzpicture}
			[every node/.style={draw, circle, fill, inner sep=0pt, minimum 
				size=2mm}]
			\tikzset{edge/.style={very thick}}
			
			\draw[black, very thick] (0,0) circle (0.5);
			
			\node[] (a) at (0,1) {} 
			node[draw = none, fill = none, above =1mm of a]{\large $a$};
			\node[] (v) at (0,0.5) {} 
			node[draw = none, fill = none, above left =0.5mm and 1mm of 
			v]{\large $v$};
			\node[] (v2) at (210:0.5) {} 
			node[draw = none, fill = none, left=1mm of v2]{\large $v_2$};
			\node[] (b) at ({0.5*cos(30)}, -1.5) {} 
			node[draw = none, fill = none, below=1mm of b]{\large $b$};
			\node[] (v1) at (-30:0.5) {}
			node[draw = none, fill = none, right=1mm of v1]{\large $v_1$};
			\node[] (u2) at ({-0.5*cos(30)}, -1) {} 
			node[draw = none, fill = none, left=1mm of u2]{\large $u_2$};
			\node[] (a') at (-0.5, -1.5) {} 
			node[draw = none, fill = none, below=1mm of a']{\large $a'$};
			\node[] (b') at (0, -1.5) {} 
			node[draw = none, fill = none, below=1mm of b']{\large $b'$};
			
			\draw[edge, dashed] (a) -- (v) 
			(v1) -- (b) 
			(u2) -- (a') 
			(u2) -- (b');
			\draw[edge] 		(v2) -- (u2);
			
			\end{tikzpicture}
			\caption{Case 2. (a), when two leaves from the two partitions ($a'$ 
				and~$b'$) are reachable from the blob via the same cut-edge.}
			\label{subcap:2a}
		\end{subfigure}\hfill
		\begin{subfigure}[t]{.25\textwidth}
			\centering
			\begin{tikzpicture}
			[every node/.style={draw, circle, fill, inner sep=0pt, minimum 
				size=2mm}]
			\tikzset{edge/.style={very thick}}
			
			\draw[black, very thick] (0,0) circle (0.5);
			
			\node[] (a) at ({-0.5*cos(30)},1) {} 
			node[draw = none, fill = none, above =1mm of a]{\large $a$};
			\node[] (b) at ({0.5*cos(30)}, 1) {} 
			node[draw = none, fill = none, above=1mm of b]{\large $b$};
			
			\node[] (a') at ({0.5*cos(30)}, -1) {}
			node[draw=none, fill=none, below=1mm of a']{\large $a'$};
			\node[] (b') at ({-0.5*cos(30)}, -1) {}
			node[draw=none, fill=none, below=1mm of b']{\large $b'$};
			
			\node[] (v1) at (150:0.5) {} 
			node[draw = none, fill = none, left = 1mm of v1]{\large $v_1$};
			\node[] (v2) at (30:0.5) {} 
			node[draw = none, fill = none, right = 1mm of v2]{\large $v_2$};
			\node[] (v3) at (-30:0.5) {} 
			node[draw = none, fill = none, right = 1mm of v3]{\large $v_3$};
			\node[] (v4) at (210:0.5) {} 
			node[draw = none, fill = none, left = 1mm of v4]{\large $v_4$};
			
			\draw[edge, dotted] (a) -- (v1)
			(b) -- (v2)
			(a') -- (v3)
			(b') -- (v4);
			
			\end{tikzpicture}
			\caption{Case 2. (a) i.}
			\label{subcap:2ai}
		\end{subfigure}\hfill
		\begin{subfigure}[t]{.25\textwidth}
			\centering
			\begin{tikzpicture}
			[every node/.style={draw, circle, fill, inner sep=0pt, minimum 
				size=2mm}]
			\tikzset{edge/.style={very thick}}
			
			\draw[black, very thick] (0,0) circle (0.5);
			
			\node[] (a) at ({-0.5*cos(30)},1) {} 
			node[draw = none, fill = none, above =1mm of a]{\large $a$};
			\node[] (b) at ({0.5*cos(30)}, 1) {} 
			node[draw = none, fill = none, above=1mm of b]{\large $b$};
			
			\node[] (a') at ({0.5*cos(30)}, -1) {}
			node[draw=none, fill=none, below=1mm of a']{\large $b'$};
			\node[] (b') at ({-0.5*cos(30)}, -1) {}
			node[draw=none, fill=none, below=1mm of b']{\large $a'$};
			
			\node[] (v1) at (150:0.5) {} 
			node[draw = none, fill = none, left = 1mm of v1]{\large $v_1$};
			\node[] (v2) at (30:0.5) {} 
			node[draw = none, fill = none, right = 1mm of v2]{\large $v_2$};
			\node[] (v3) at (-30:0.5) {} 
			node[draw = none, fill = none, right = 1mm of v3]{\large $v_3$};
			\node[] (v4) at (210:0.5) {} 
			node[draw = none, fill = none, left = 1mm of v4]{\large $v_4$};
			
			\draw[edge, dotted] (a) -- (v1)
			(b) -- (v2)
			(a') -- (v3)
			(b') -- (v4);
			
			\end{tikzpicture}
			\caption{Case 2. (a) ii.}
			\label{subcap:2aii}
		\end{subfigure}
		
		\begin{subfigure}[t]{.25\textwidth}
			\centering
			\begin{tikzpicture}
			[every node/.style={draw, circle, fill, inner sep=0pt, minimum 
				size=2mm}]
			\tikzset{edge/.style={very thick}}
		
			\draw[black, very thick, domain=45:135] plot ({cos(\x)}, {sin(\x)});
			\draw[black, very thick, domain=135:170, dotted] plot ({cos(\x)}, 
			{sin(\x)});
			\draw[black, very thick, domain=170:190] plot ({cos(\x)}, 
			{sin(\x)});
			\draw[black, very thick, domain=-170:45, dotted] plot ({cos(\x)}, 
			{sin(\x)});
			
			\node[] (north) at (0,1) {};
			\node[] (south) at (0,-1) {};
			
			\node[] (v4) at (135:1) {}
			node[draw = none, fill = none, above = 1mm of v4]{\large $v_4$};
			\node[] (v2) at (170:1) {}
			node[draw = none, fill = none, right = 1mm of v2]{\large $v_2$};
			\node[] (v1) at (190:1) {}
			node[draw = none, fill = none, right = 1mm of v1]{\large $v_1$};
			\node[] (v3) at (45:1) {}
			node[draw = none, fill = none, above = 1mm of v3] {\large $v_3$};
			
			\node[] (b') at (-1.5, {sin(135)}) {}
			node[draw = none, fill = none, left = 1mm of b']{\large $b'$};
			\node[] (b) at (-1.5, {sin(170)}) {}
			node[draw = none, fill = none, left = 1mm of b]{\large $b$};
			\node[] (a) at (-1.5, {sin(190)}) {}
			node[draw = none, fill = none, left = 1mm of a]{\large $a$};
			\node[] (a') at (1.5, {sin(45)}) {}
			node[draw = none, fill = none, right = 1mm of a']{\large $a'$};	
			
			\draw[edge, dotted] (v4) -- (b')
			(v2) -- (b)
			(v1) -- (a)
			(v3) -- (a')
			(north) -- (south);
			
			\end{tikzpicture}
			\caption{Case 2. (b) i. The bottom-left dashed edge is a potential 
				type-$B$ cut-edge. This edge may or may not exist in the 
				network 
				(does 
				not affect the case).}
			\label{subcap:2bi}
		\end{subfigure}\hfill
		\begin{subfigure}[t]{.25\textwidth}
			\centering
			\begin{tikzpicture}
			[every node/.style={draw, circle, fill, inner sep=0pt, minimum 
				size=2mm}]
			\tikzset{edge/.style={very thick}}
			
			\draw[black, very thick, domain=45:135] plot ({cos(\x)}, {sin(\x)});
			\draw[black, very thick, domain=135:225, dotted] plot ({cos(\x)}, 
			{sin(\x)});
			\draw[black, very thick, domain=225:315] plot ({cos(\x)}, 
			{sin(\x)});
			\draw[black, very thick, domain=-45:45, dotted] plot ({cos(\x)}, 
			{sin(\x)});
						
			\node[] (north) at (0,1) {};
			\node[] (south) at (0,-1) {};
			
			\node[] (v1) at (135:1) {}
			node[draw = none, fill = none, above = 1mm of v1]{\large $v_1$};
			\node[] (v2) at (45:1) {}
			node[draw = none, fill = none, above = 1mm of v2]{\large $v_2$};
			\node[] (v3) at (225:1) {}
			node[draw = none, fill = none, below = 1mm of v3]{\large $v_3$};
			\node[] (v4) at (315:1) {}
			node[draw = none, fill = none, below = 1mm of v4]{\large $v_4$};
			
			\node[] (a) at (-1.5, {sin(135)}) {}
			node[draw = none, fill = none, left = 1mm of a]{\large $a$};
			\node[] (a') at (-1.5, {sin(225)}) {}
			node[draw = none, fill = none, left = 1mm of a']{\large $a'$};
			\node[] (b) at (1.5, {sin(45)}) {}
			node[draw = none, fill = none, right = 1mm of b]{\large $b$};
			\node[] (b') at (1.5, {sin(315)}) {}
			node[draw = none, fill = none, right = 1mm of b']{\large $b'$};
			
			\draw[edge, dotted] (v1) -- (a)
			(v3) -- (a')
			(v2) -- (b)
			(v4) -- (b')
			(north) -- (south); 
			
			\node[draw=none, fill=none] (a1) at (-1.5, {sin(160)}) {};
			\node[draw=none, fill=none] (a2) at (-1.5, {sin(200)}) {};
			
			\node[draw=none, fill=none] (b1) at (1.5, {sin(20)}) {};
			\node[draw=none, fill=none] (b2) at (1.5, {sin(-20)}) {};
			
			\end{tikzpicture}
			\caption{Case 2. (b) ii. A.}
			\label{subcap:2biiA}
		\end{subfigure}\hfill
		\begin{subfigure}[t]{.25\textwidth}
			\centering
			\begin{tikzpicture}
			[every node/.style={draw, circle, fill, inner sep=0pt, minimum 
				size=2mm}]
			\tikzset{edge/.style={very thick}}
			
			\draw[black, very thick, domain=45:315] plot ({cos(\x)}, {sin(\x)});
			\draw[black, very thick, domain=-45:45, dotted] plot ({cos(\x)}, 
			{sin(\x)});
			
			\node[] (north) at (0,1) {};
			\node[] (south) at (0,-1) {};
			
			\node[] (v1) at (180:1) {}
			node[draw = none, fill = none, above left = 1mm and 0.5mm of 
			v1]{\large 
				$v_1$};
			\node[] (v2) at (45:1) {}
			node[draw = none, fill = none, above = 1mm of v2]{\large $v_2$};
			\node[] (v3) at (0,0) {}
			node[draw = none, fill = none, right = 1mm of v3]{\large $v_3$};
			\node[] (v4) at (315:1) {}
			node[draw = none, fill = none, below = 1mm of v4]{\large $v_4$};
			
			\node[] (a) at (-1.5, {sin(180)}) {}
			node[draw = none, fill = none, left = 1mm of a]{\large $a$};
			\node[] (a') at (-0.5, {sin(0)}) {}
			node[draw = none, fill = none, above = 1mm of a']{\large $a'$};
			\node[] (b) at (1.5, {sin(45)}) {}
			node[draw = none, fill = none, right = 1mm of b]{\large $b$};
			\node[] (b') at (1.5, {sin(315)}) {}
			node[draw = none, fill = none, right = 1mm of b']{\large $b'$};
			
			\draw[edge, dotted] (v1) -- (a)
			(v3) -- (a');
			\draw[edge, dotted] (v2) -- (b)
			(v4) -- (b');
			\draw[edge] (north) -- (south); 
			
			\node[draw=none, fill=none] (b1) at (1.5, {sin(20)}) {};
			\node[draw=none, fill=none] (b2) at (1.5, {sin(-20)}) {};
			
			\end{tikzpicture}
			\caption{Case 2. (b) ii. B.}
			\label{subcap:2biiB}
		\end{subfigure}
		\caption{All cases examined in the proof of 
		Theorem~\ref{thm:CutEdgeSplit}. 
			The dotted edges represent a path between the two 
			vertices.}
		\label{fig:Lemma1Cases}
	\end{figure}
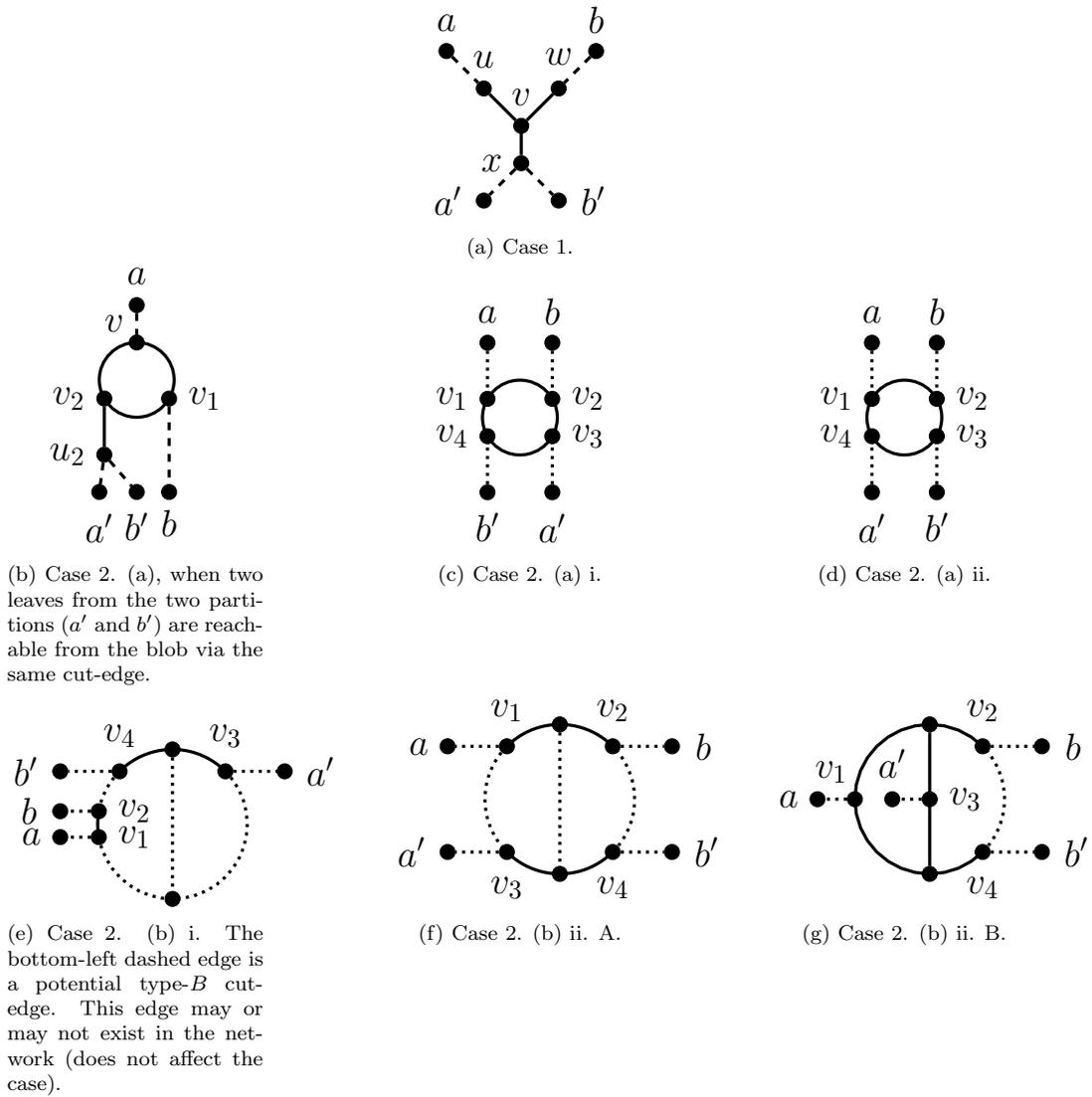
	
	\begin{enumerate}
		\item \textbf{$\bm{v}$ is not in a blob:} 
		Let~$w,x$ denote the two 
		neighbors of~$v$ that are not~$u$. By our choice of~$e$, there 
		must be a leaf~$b\in B$ that can be reached from the 
		edge~$vw$, and a leaf~$b'\in B$ that can be reached from the edge~$vx$. 
		Without loss of generality, assume that~$a'$ can be 
		reached from the edge~$vx$. But this means that
		\begin{align*}
		d_m(a,b) + d_m(a',b') &\le d_m(a,v) + d_m(v,b) + d_m(a',x) + d_m(x,b')\\
		&= [d_m(a,v) + d_m(v,x) + d_m(a',x)] - d_m(v,x)\\
		&\quad + [d_m(b,v) + d_m(v,x) + d_m(x,b')] - d_m(v,x)\\
		&= d_m(a,a') - d_m(v,x) + d_m(b,b') - d_m(v,x)\\
		&= d_m(a,a') + d_m(b,b') - 2\\
		&< d_m(a,a') + d_m(b,b'),
		\end{align*}
		where the first inequality may be strict since the shortest path 
		between~$a'$ and~$b'$ may not pass through~$x$. This contradicts the 
		second 
		condition of the claim.
		\item \textbf{$\bm{v}$ is a vertex of a blob~$\bm{C}$:} The blob~$C$ 
		must be 
		incident to at least two cut-edges~$e_1, e_2$ other than~$uv$, for 
		which there must be elements~$b$ and~$b'$ in~$B$ that are reachable 
		from~$e_1$ and~$e_2$ respectively. Otherwise, as before, this would 
		contradict our choice of a farthest~$uv$. We claim that if~$a'$ can be 
		reached from either~$e_1$ or~$e_2$, then we would reach a 
		contradiction. Without loss of generality, suppose that~$a'$ can be 
		reached from~$e_2$. Letting~$e_1=u_1v_1$ and~$e_2 = u_2v_2$ where~$v_1$ 
		and~$v_2$ are vertices on~$C$, we have that
		\begin{align*}
		d_m(a,b) + d_m(a',b') &< d_m(a,v) + d_m(v,v_1) + d_m(v_1,b) + 
		d_m(a',v_2) + d_m(v_2,b')\\
		&= d_m(a,a') - d_m(v,v_2) + d_m(b,b') - d_m(v_1,v_2) + d_m(v,v_1)\\
		&\le d_m(a,a') + d_m(b,b'),
		\end{align*}
		where the first inequality follows as the shortest path between~$a'$ 
		and~$b'$ does not contain~$v_2$, and the final inequality follows from 
		the triangle inequality. This contradicts the second condition of the 
		claim. Therefore, we may assume from now that there are at least four 
		cut-edges incident to the blob~$C$ and that no leaves from~$A$ and~$B$ 
		can be reached from the same cut-edge incident to~$C$. That is, every 
		cut-edge incident to~$C$ induces a split that has either a subset 
		of~$A$ or a subset of~$B$ as one of its parts. We refer to 
		these as \emph{type-$A$ 
		cut-edges} and \emph{type-$B$ cut-edges}, respectively. 
	
		We have another case that is common both for the instances when~$C$ is 
		either a level-$1$ or a level-$2$ blob. Suppose first 
		that there exist two pairs of cut-edges~$e_1, e_2$, and~$e_3,e_4$ 
		incident 
		to~$C$, whose endpoints are adjacent, respectively, 
		such that all four edges are distinct and~$a,b,a',b'$ are reachable 
		from $e_1,e_2,e_3,e_4$ respectively. We let~$v_i$ denote the 
		vertices of~$C$ that are endpoints of~$e_i$ for~$i=1,2,3,4$, 
		respectively. Then we have
		\begin{align*}
		d_m(a,b) + d_m(a',b') &= d_m(a,v_1) + d_m(v_1,v_2) + d_m(v_2,b) + 
		d_m(a',v_3)\\
		&\quad + d_m(v_3,v_4) + d_m(v_4,b')\\
		&= d_m(a,a') + d_m(b,b') - d_m(v_1,v_3) - d_m(v_2,v_4) + 2\\
		&\le d_m(a,a') + d_m(b,b'),
		\end{align*}
		where the second equality follows as~$d_m(v_1,v_2) = d_m(v_3,v_4) = 
		1$. This contradicts the second inequality of the 
		claim. 
		
		If~$C$ is a level-$1$ blob, then the above case always 
		applies. Indeed, there must be at least four cut-edges incident to~$C$, 
		of which at least two are type-$A$ and the remaining edges are 
		type-$B$. 
		\rev{If there was only one type-$B$ edge, then such a cut-edge induces the split~$A|B$, and we are done.}
		\rev{So} this implies that there are always two distinct pairs of  
		type-$A$ and type-$B$ edges, whose endpoints on~$C$ are adjacent.
		Thus we may assume that~$C$ is a level-$2$ blob.
	
		We may assume that each main path of $C$ contains only 
		type-$A$ cut-edges or only type-$B$ cut-edges, or a combination of 
		the two, for which such a main path contains one type of cut-edges, 
		a single cut-edge of the other type, and possibly cut-edges of the 
		first type. For example, a path corresponding to a main path of~$B$ 
		may be~$e_0v_1\cdots v_ke_1$ where~$k\ge2$ and~$e_0,e_1$ are 
		boundary vertices. For some integer~$j\le k$, we have $v_1,v_2, 
		\ldots, v_{j-1}, v_{j+1}, \ldots, v_k$ are incident to type-$A$ 
		cut-edges, and~$v_j$ is incident to a type-$B$ cut-edge. We call 
		such a main path a \emph{combination side}. Observe that a 
		combination side contains either one type-$A$ or one type-$B$ 
		cut-edge. Also note that the blob~$C$ contains at most one 
		combination side as otherwise there would be two distinct pairs of 
		type-$A$ and type-$B$ edges, whose endpoints on~$C$ are adjacent.
		\begin{enumerate}
			\item \textbf{$\bm{C}$ contains one combination side~$s$:}
			Suppose without loss of generality that~$s$ is a combination 
			side containing exactly one type-$A$ cut-edge. Let~$v_1$ denote 
			the endpoint of this cut-edge on~$C$, and let~$v_2$ be an 
			adjacent vertex on~$C$ that is incident to a type-$B$ cut-edge. 
			Since~$C$ is incident to at least two type-$A$ cut-edges, there 
			must be another main path~$s'$ of~$C$ that is incident to only 
			type-$A$ cut-edges. Similarly, since~$C$ is incident to at 
			least two type-$B$ cut-edges, there must be another type-$B$ 
			cut-edge~$e_4$ that is incident to~$C$. We may assume in 
			particular that an endpoint~$v_4$ of~$e_4$ is a main end-spine 
			vertex incident either to~$s$ or to the third main path of~$C$. 
			Either way, there must exist an end-spine vertex~$v_3$ on~$s'$ 
			such that~$d_m(v_3,v_4) = 2$. Observing that~$v_3$ is an 
			endpoint of a type-$A$ cut-edge, we may assume that the 
			leaves~$a,b,a',b'$ are separated from~$C$ by~$v_1,v_2,v_3,v_4$, 
			respectively. Then,
			\begin{align*}
			d_m(a,b) + d_m(a',b') -2 &= d_m(a,v_1) + d_m(v_1,v_2) + 
			d_m(v_2,b) + d_m(a',v_3)\\
			&\quad + d_m(v_3,v_4) + d_m(v_4,b') - 2\\
			&= d_m(a,a') + d_m(b,b') + d_m(v_1,v_2) + d_m(v_3,v_4) \\
			&\quad - d_m(v_1,v_3) - d_m(v_2,v_4)-2\\
			&\le d_m(a,a') + d_m(b,b') + 1+2-2 - d_m(v_1,v_3)\\
			&\quad - d_m(v_2,v_4)\\
			&< d_m(a,a') + d_m(b,b'),
			\end{align*} 
			since~$d_m(v_1,v_3)\ge 1$ and~$d_m(v_2,v_4)\ge 1$. This
			leads to a contradiction of the second condition.
			\item \textbf{\rev{Each} main path of~$\bm{C}$ contain cut-edges of 
				the same type:} Observe that at least one main path must 
			contain at least two cut-edges, since there are at least four 
			cut-edges incident to~$C$ and~$C$ has three main paths. Without 
			loss of generality, suppose that there is a main path~$s$ with 
			at least~$2$ type-$B$ edges.
			\begin{enumerate}
				\item \textbf{There is another main path~$s'$ with at least 
					two type-$\bm{A}$ edges:} Then choose~$v_1,v_3$ 
				and~$v_2,v_4$ to be the main end-spine vertices of~$s'$ 
				and~$s$, respectively, such that~$d_m(v_1,v_2) = 2$ 
				and~$d_m(v_3,v_4) = 2$. Supposing that the 
				leaves~$a,b,a',b'$ are separated from~$C$ 
				by~$v_1,v_2,v_3,v_4$, respectively, we have that
				\begin{align*}
				d_m(a,b) + d_m(a',b') &= d_m(a',b) + d_m(a,b') - 
				d_m(v_1,v_4) - 
				d_m(v_2,v_3)\\
				&\quad + d_m(v_1,v_2) + d_m(v_3,v_4)\\
				&\le d_m(a',b) + d_m(a,b') - 3 - 3 + 2 + 2\\
				&< d_m(a',b) + d_m(a,b'),
				\end{align*}
				since~$d_m(v_1,v_4)\ge 3$ and~$d_m(v_2,v_3)\ge 3$, where 
				this inequality is strict when the third main path contains 
				no cut-edges. This clearly contradicts the first condition 
				of the claim.
				\item \textbf{The other two main paths contain exactly one 
					type-$\bm{A}$ edge each:} Choose~$v_1,v_3$ to be the two 
				possible vertices incident to the type-$A$ cut-edges, 
				and~$v_2,v_4$ to be the main end-spine vertices of~$s$. 
				Supposing that the leaves~$a,b,a',b'$ are separated 
				from~$C$ by~$v_1,v_2,v_3,v_4$, respectively, we have that
				\begin{align*}
				d_m(a,b) + d_m(a',b') &= d_m(a,a') + d_m(b,b') + 
				d_m(v_1,v_2) + d_m(v_3,v_4) - d_m(v_1,v_3) - d_m(v_2,v_4)\\
				&= d_m(a,a') + d_m(b,b') + 4 - 2 - d_m(v_2,v_4)\\
				&< d_m(a,a') + d_m(b,b') +2,
				\end{align*}
				which contradicts the second condition of the claim.
			\end{enumerate}
		\end{enumerate}
	\end{enumerate}
	This covers all possible cases, for which we have obtained a contradiction 
	in each case. Therefore~$A|B$ must be a cut-edge induced split of the 
	network.
\end{proof}

Note that Theorem~\ref{thm:CutEdgeSplit} does not hold for networks of level at 
least~$3$ (see Figure~\ref{fig:CutEdgeSplitL3}). 
Let us call a split~$A|B$ 
\emph{minimal} if there exists no non-trivial split~$A'|B'$ of the same network 
such that \rev{one of~$A'$ and~$B'$ is a proper subset of~$A$ or~$B$.}
We say that~$A$ and~$B$ are \emph{minimal parts} 
of~$A|B$, respectively. Note that minimal parts of a split may not be unique as 
two pendant blobs may be connected by a non-trivial cut-edge~$e$, for which 
both parts of the split are minimal parts.

\begin{lemma}\label{lem:MinimalPendant}
	Let~$N$ be a level-$2$ network \rev{on~$X$} with at least two pendant blobs.
	Then~$N$ contains a pendant blob containing the set of leaves~$A$ if and 
	only if~$A|B$ is a minimal cut-edge induced non-trivial split where~$A$ is 
	a minimal part.
\end{lemma}
\begin{proof}
	Suppose first that~$N$ contains a pendant blob~$C$ containing the set of 
	leaves~$A$. Then there exists exactly one non-trivial cut-edge~$e$ incident 
	to~$C$, which induces the non-trivial split~$A|B$ (where~$B=X-A$). To see 
	that~$A$ is a minimal part, observe that for every cut-edge induced 
	split~$A'|B'$ where~$A'\subseteq A$, we have~$|A'| = 1$, since every
	cut-edge incident to~$C$ other than~$e$ is trivial. Therefore,~$A|B$ is a 
	minimal split, where~$A$ is a minimal part.
	
	Suppose now that~$A|B$ is a minimal non-trivial split induced by~$e$,
	where~$A$ is a minimal part. Suppose for a contradiction that~$N$ did not 
	contain a pendant blob with the set of leaves~$A$. Because we may 
	assume~$N$ contains no cherries, the part of~$N$ corresponding to the split 
	part~$A$ (i.e., the graph obtained by deleting~$e$ and taking the component 
	with the leaves from~$A$) must contain a pendant blob~$C$. Such a pendant 
	blob contains the set of leaves~$A'$, where~$A'\subseteq A$. The 
	non-trivial 
	cut-edge incident to~$C$ induces the split~$A'|B'$, where~$B' = X-A'$. By 
	definition,~$A'|B'$ must be a non-trivial split. But this contradicts the 
	fact that~$A|B$ was minimal. Therefore,~$N$ must contain a pendant blob 
	with the set of leaves~$A$.
\end{proof}

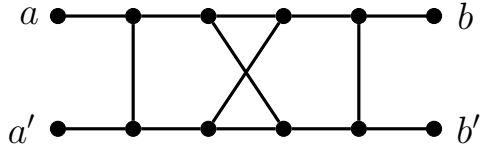
\begin{figure}
	\centering
	\begin{tikzpicture}[every node/.style={draw, circle, fill, inner sep=0pt, 
	minimum 
		size=2mm}]
	\tikzset{edge/.style={very thick}}
	
	\foreach \x in {0,...,5}
	\foreach \y in {0,1}
	\node[] (\x\y) at (-3+\x, -1.5*\y) {};
	
	\foreach \x [count = \xi] in {0,...,4}
	\foreach \y in {0,1}
	\draw[edge] (\x\y) -- (\xi\y);
	
	\draw[edge] (10) --  (11);
	\draw[edge] (20) -- (31);
	\draw[edge] (30) -- (21);
	\draw[edge] (40) -- (41);
	
	\node[draw = none, fill = none, left=1mm of 00]  	(1) {\large $a$};
	\node[draw = none, fill = none, left=1mm of 01]  	(2) {\large $a'$};
	\node[draw = none, fill = none, right=1mm of 50]  	(3) {\large $b$};
	\node[draw = none, fill = none, right=1mm of 51]  	(4) {\large $b'$};
	\end{tikzpicture}
	\caption{A level-$3$ network on leaf-set~$\{a,a',b,b'\}$. Observe that the 
		conditions for Theorem~\ref{thm:CutEdgeSplit} are satisfied 
		for~$A=\{a,a'\}$ 
		and~$B=\{b,b'\}$, but there is no cut-edge that induces the 
		split~$A|B$.}
	\label{fig:CutEdgeSplitL3}
\end{figure}

\subsection{sl-distance reconstructibility}

We show now that we can identify pendant blobs of level-$2$ networks from 
their sl-distance matrices.

\begin{lemma}\label{lem:IdentifyBlobsFromSplits}
	Let~$N$ be a level-$2$ network on~$X$ with at least two blobs. Let~$A|B$ be 
	a non-trivial split of~$N$ 
	\rev{where~$A$ is the minimal part.
	Then~$N$ contains a pendant blob containing the set of leaves~$A$, and} if~$A$ contains
	\begin{itemize}
		\item $1$ chain~$(a,k)$, then~$N$ contains
		\begin{itemize}
			\item a pendant level-$1$ blob containing~$(a,k)$ if and only if
			\begin{itemize}
				\item $2\le k\le 3$,~$d_m(a_1,a_k) = k+1$, and~$d_l(a_1,a_k) = 
				4$.
				\item $k\ge 4$ and~$d_m(a_1,a_k) = 4$.
			\end{itemize}
			\item a pendant level-$2$ blob of the form~$(a,0,0,0)$ if and only 
			if
			\begin{itemize}
				\item $2\le k\le 3$,~$d_m(a_1,a_k) = k+1$, and~$d_l(a_1,a_k) = 
				6$.
				\item $k\ge 4$ and~$d_m(a_1,a_k) = 5$.
			\end{itemize}
		\end{itemize}
		\item $2$ chains~$(a,k)$ and~$(b,\ell)$, then~$N$ contains
		\begin{itemize}
			\item a pendant level-$2$ blob of the form~$(a,b,0,0)$ if and 
			only if for all~$x\in X-(a\cup b)$, we have~$d_m(a,x) = d_m(b,x)$.
			\item a pendant level-$2$ blob of the form~$(a,0,b,0)$ if and 
			only if for all~$x\in X-(a\cup b)$, we have~$d_m(a,x) = d_m(b,x)+1$.
		\end{itemize}
		\item $3$ chains~$(a,k), (b,\ell),$ and~$(c,m)$, then~$N$ contains
		\begin{itemize}
			\item a pendant level-$2$ blob of the form~$(a,b,c,0)$ if and 
			only if for all~$x\in X-(a\cup b\cup c)$, we have~$d_m(a,x) = 
			d_m(b,x) = d_m(c,x) + 1$.
			\item a pendant level-$2$ blob of the form~$(a,0,b,c)$ if and 
			only if for all~$x\in X-(a\cup b\cup c)$, we have~$d_m(a,x) = 
			d_m(b,x) + \min\{\ell,m\}+1 = d_m(c,x) + \min\{\ell,m\} + 1$.
		\end{itemize}
		\item $4$ chains~$(a,k), (b,\ell), (c,m),$ and~$(d,n)$ then~$N$ 
		contains a pendant level-$2$ blob of the form~$(a,b,c,d)$ if and 
		only if~$(a,b,c)$ and~$(a,b,d)$ are both pairwise adjacent triples.
	\end{itemize}
\end{lemma}
\begin{proof}
    \rev{The fact that~$N$ contains a pendant blob containing the set of leaves~$A$ follows from Lemma~\ref{lem:MinimalPendant}.}
    
	Suppose first that~$N$ contains 
	\rev{either a pendant level-$1$ blob or a pendant level-$2$ blob of the form~$(a,b,c,d)$, where $a,b,c,d$ could be empty chains.}
	Then it is 
	easy to see by inspection that these distances hold and also that the 
	pairwise adjacent triple statement holds in the case of~$4$ chains (see 
	Figure~\ref{fig:PBlobs}).
	\medskip
	
	To show the other direction, note that within a level-$2$ network, there is 
	one level-$1$ pendant blob, and there are six possible level-$2$ \rev{pendant} blobs. 
	\rev{
	We know that~$N$ contains a pendant blob with the leaves of~$A$; it remains to show that if the conditions on the distances are satisfied, then~$N$ must contain the corresponding pendant blob.
	From the sl-distance matrix, we can infer the number of distinct chains contained in~$A$, as well as their adjacencies.
	Then, we can infer the type of this pendant blob by looking at the distance from the leaves of~$A$ to some leaf that is not in~$A$.
	We give one example here for the case when~$A$ consists of exactly three chains.
	The proof for the other cases follow in an analogous fashion.}
	
	We give a proof for the case when~$A$ contains~$3$ 
	chains~$(a,k),(b,\ell),$ and~$(c,m)$. Pendant level-$1$ blobs 
	contain exactly~$1$ chain; thus the pendant blob must be level-$2$. 
	Level-$2$ pendant blobs have three main paths, one of which contains the 
	endpoint of the incident non-trivial cut-edge. This main path, say~$s$,  
	contains at least~$1$ chain and at most~$2$ chains, whilst the other two 
	main paths contain at most~$1$ chain. Let~$x\in X-a\cup b\cup c$ be an 
	arbitrary leaf. Two of the chains, say~$a$ and~$b$, have the same minimal 
	distance to~$x$, and the other chain~$c$ has different minimal distance. If 
	the \rev{distance between~$c$ and~$x$ is shorter than that between~$a$ and~$x$}, 
	then we know that~$c$ must be contained in 
	the main path~$s$ of~$B$, and we have a pendant level-$2$ blob of the 
	form~$(a,b,c,0)$. On the other 
	hand, \rev{if the distance between~$c$ and~$x$ is longer than that between~$a$ and~$x$}, then we know that~$a$ and~$b$ must be 
	contained in the main path~$s$ of~$B$, and we have a pendant level-$2$ blob 
	of the form~$(c,0,a,b)$.
\end{proof}

Observe that in the proof of Lemma~\ref{lem:IdentifyBlobsFromSplits}, the 
longest distance information was used only to distinguish the pendant level-$1$ 
blob with a chain~$(a,k)$ and the pendant level-$2$ blob of the form~$(a,0,0,0)$  
for~$k\in\{2,3\}$. In other words, using only the shortest distances, the pendant 
level-$1$ blob containing~$2$ leaves cannot be distinguished from the pendant 
level-$2$ blob also containing~$2$ leaves on the same side; the pendant 
level-$1$ blob 
containing~$3$ leaves cannot be distinguished from the pendant level-$2$ blob 
of containing the same leaves on the same side. We shall denote these four 
subgraph structures as 
\emph{bad blobs}. That is, we say that a level-$1$ blob is \emph{bad} if it is 
incident to exactly three or four cut-edges. We say that a level-2 blob~$B$ is 
\emph{bad} if, of the three main paths~$s_1,s_2,s_3$ of~$B$, the main 
side~$s_1$ is incident to a single cut-edge,~$s_2$ is incident to no cut-edges, 
and~$s_3$ is incident to exactly two or three cut-edges.

The reason why we cannot discern these bad blobs is because the shortest 
distance between the end-leaves of the chain uses the path containing the spine 
of the chain, which is the same length for both pendant level-$1$ and pendant 
level-$2$ blobs. Whenever these chains contain at least~$4$ leaves, a
shortest path no longer contains the spine; since such paths differ in distance 
for pendant level-$1$ and pendant level-$2$ blobs with a single chain, we are 
able to identify such pendant blobs. We later show that level-$2$ networks that 
do not contain bad blobs are reconstructible from their shortest distances 
(Corollary~\ref{cor:L2Restricted}).

The following lemma states that if we can identify certain structures within 
level-$2$ networks, then we may replace them by a leaf, and we may obtain the 
distance matrix of the reduced network.

\begin{lemma}\label{lem:AdjustDistances}
	Let~$N$ be a level-$2$ network on~$X$ with a pendant blob~$B$, and 
	replace~$B$ by a leaf~$z\notin X$ to obtain the network~$N'$. 
	Letting~$Y$ denote the set of leaves contained in~$B$, we have that the 
	sl-distance matrix of~$N'$ contains the elements
	\[d^{N'}(p,q) = d^N(p,q)\]
	for all pair of leaves $p,q\in X-Y$. Now, for all~$p\in X-Y$, we have the 
	following.
	\begin{itemize}
		\item \textbf{$\bm{B}$ is a pendant level-$\bm{1}$ blob with the 
		chain~$\bm{(a,k)}$:} 
		\[d^{N'}(p,z) = \{d^N_m(p,a) - 2, d^N_l(p,a) - (k+1)\}\]
		\item \textbf{$\bm{B}$ is a pendant level-$\bm{2}$ blob of the 
		form~$\bm{F}$:}
		\[d^{N'}(p,z) = 
		\begin{cases}
			\{d^N_m(p,a)-3, d^N_l(p,a)-(k+3)\} 					
			& \text{ if } F = (a,0,0,0)\\
			\{d^N_m(p,a)-3, d^N_l(p,a)-(k+\ell+3)\} 			
			& \text{ if } F = (a,b,0,0)\\
			\{d^N_m(p,c)-2, d^N_l(p,c)-(k+m+3)\} 				
			& \text{ if } F = (a,0,c,0)\\
			\{d^N_m(p,c)-2, d^N_l(p,c)-(\max\{k,\ell\}+m+3)\} 	
			& \text{ if } F = (a,b,c,0)\\
			\{d^N_m(p,c)-2, d^N_l(p,c)-(k+m+n+3)\} 				
			& \text{ if } F = (a,0,c,d)\\
			\{d^N_m(p,c)-2, d^N_l(p,c)-(\max\{k,\ell\}+m+n+3)\}	
			& \text{ if } F = (a,b,c,d)
		\end{cases}
		\]
	\end{itemize}
\end{lemma}
\begin{proof}
	To obtain the inter-taxa distances for~$N'$, it suffices to simply subtract 
	the shortest / longest distances from the vertex of the pendant blob 
	incident to the non-trivial cut-edge to an end-spine leaf of a chain. These 
	distances are easy to obtain as we know exactly what the pendant blobs are 
	in all cases, due to Lemma~\ref{lem:IdentifyBlobsFromSplits}.
\end{proof}
The above two lemmas will now be combined to prove the following result.

\begin{theorem}\label{thm:L2SL}
	Level-2 networks are reconstructible from their sl-distance matrix.
\end{theorem}
\begin{proof}
	We prove by induction on the size of the network. For the base case, 
	\rev{a network on a single edge has two leaves, which is trivially reconstructible from its shortest distances.}
	In fact, we know by Lemma~\ref{lem:L1} that a network on a single blob is 
	reconstructible from its shortest distances. So suppose that we are given 
	a level-$2$ network~$N$ with~$|E(N)|$ edges, and that the result holds 
	for all level-$2$ networks with at most~$|E(N)|-1$ edges.
	
	We may assume that~$N$ contains at least two pendant blobs. By the results 
	in Section~\ref{subsec:Chain}, we can 
	partition the leaves into chains, and adjacency between chains can be 
	obtained from sl-distance matrices. By~Theorem~\ref{thm:CutEdgeSplit}, we 
	can 
	obtain all cut-edge induced splits of~$N$ from its shortest distance 
	matrix; by Lemma~\ref{lem:IdentifyBlobsFromSplits}, we can identify all 
	pendant blobs from these splits, by using the sl 
	distance matrix. We can also replace one of these pendant blobs by a 
	leaf~$z$ to 
	obtain a smaller level-$2$ network~$N'$, for which its shortest and longest 
	inter-taxa distances can be obtained by Lemma~\ref{lem:AdjustDistances}. By 
	induction hypothesis,~$N'$ is reconstructible. Then, we can obtain a 
	network isomorphic to~$N$ by replacing the leaf~$z$ with the pendant blob 
	that was originally present. 
	
	\rev{To see that this network is unique, consider another network~$M$ that is not isomorphic to~$N$ such that~$M$ induces the same sl-distance matrix as~$N$.
	Note that~$M$ must also contain a pendant blob~$P$, and upon replacing~$P$ in~$M$ by a leaf~$z$, we get by the induction hypothesis that the resulting network~$M'$ must be isomorphic to~$N'$.
	We obtain a network isomorphic to~$M$ by replacing the leaf~$z$ by~$P$ in~$M'$: but this operation yields a network that is also isomorphic to~$N$.
	It follows that~$N$ and~$M$ must be isomorphic.}
	
	Therefore, level-$2$ networks are 
	reconstructible from their sl-distance matrices.
\end{proof}

As stated before, it is possible to distinguish all pendant blobs from the 
shortest distances matrices if the 
networks do not contain the bad blobs. It follows then that the proof of 
Theorem~\ref{thm:L2SL} can be adapted to prove the following corollary, when we 
look at restricted level-$2$ networks.

\begin{corollary}\label{cor:L2Restricted}
	Let~$N$ be a level-$2$ network containing no bad blobs. Then~$N$ is 
	reconstructible from its shortest distance matrix.
\end{corollary}

A direct consequence of Theorem~\ref{thm:L2SL} and 
Corollary~\ref{cor:L2Restricted}, for restricted level-$2$ networks, is that by 
iteratively reducing pendant subtrees and pendant blobs from a network, it is 
possible to reconstruct the network from its sl-distance matrix
 and shortest distance matrix, respectively. Note that subtree reduction may be necessary 
after a few iterations of reducing pendant blobs from a network, as it is 
possible to obtain cherries from such reductions. Therefore the above results 
implicitly give an algorithm for reconstructing level-$2$ networks from 
their sl-distance matrices.

Completely excluding all bad blobs is quite restrictive. There can indeed 
exist networks that contain bad blobs that are still reconstructible from their 
shortest distances. For example, take a network in which there is exactly one 
bad blob. By 
reducing all cherries and all other pendant blobs before we reduce the bad 
blob, we are able to obtain a network on a single blob (which is necessarily 
the bad blob). Since networks on single blobs are reconstructible by 
Lemma~\ref{lem:L1}, it follows then that the original network is 
also reconstructible. Therefore, in an effort to weaken the restriction of 
completely disallowing bad blobs, we next aim to characterize level-$2$ 
networks 
that are not reconstructible from their shortest distances.


\section{Characterization of Level-$2$ networks that cannot be reconstructed 
from 
their shortest distances}\label{sec:Forbidden}

In this section we show that level-$2$ networks that cannot be reconstructed 
from their shortest distances can be categorized by a type of subgraph that 
they must contain. We only consider shortest distances in this section; we use 
use~$d^N(x,y)$ to denote the shortest distance between two vertices~$x$ 
and~$y$ in a network~$N$.

\subsection{Alt-path structures}

Let~$T$ be any binary tree with labelled leaves. Two-color \rev{the vertices of}~$T$ with colors 
black and red. Let~$G$ denote a graph obtained by 
\begin{itemize}
	\item replacing each black internal vertex by a certain level-$2$ blob. 
	That is, 
	for each internal vertex~$v$ with neighbors~$u_i$ for~$i\in[3]$, 
	delete~$v$, add vertices~$v_i,n_v,s_v$ and 
	edges~$u_iv_i, n_vv_i, s_vv_i$ for~$i\in[3]$;
	\item replacing each black leaf by a pendant level-$2$ blob of the 
	form~$(2,0,0,0)$ or~$(3,0,0,0)$; \rev{and}
	\item replacing each red leaf by a pendant level-$1$ blob with two or three 
	leaves.
\end{itemize}
The leaves of~$G$ are unlabelled. We call~$G$ an \emph{alt-path structure} 
of~$T$. We may obtain another alt-path structure~$H$ of~$T$ by swapping the 
roles of the red and black vertices in the blob replacement step. We say 
that~$H$ is \emph{similar} to~$G$ if every pendant blob of~$H$ that replaces a 
leaf~$l$ of~$T$ contains the same
number~$s$ of leaves as that of~$G$ that 
replaces~$l$. See Figure~\ref{fig:AltPath} for an example of obtaining two 
similar alt-path structures from the same binary tree. 
\rev{Note that every binary tree~$T$ on at least two leaves gives rise to exactly two alt-path structures, and these are similar to each other.}

We say that a network \emph{contains} an alt-path structure of some tree if the 
alt-path structure is a subgraph of the network \rev{up to deleting leaf labels.}
Suppose that~$N$ contains an alt-path structure~$G$ of some tree~$T$. 
\rev{Let~$H$ be the similar alt-path structure of~$G$. The operation of \emph{replacing}~$G$
by its similar alt-path structure is the action of replacing the subgraph~$G$ by~$H$ in~$N$.}




\begin{figure}
	\centering
	\begin{subfigure}{.45\textwidth}
		\centering
		\begin{tikzpicture}
		[every node/.style={draw, circle, fill, inner sep=0pt, minimum 
			size=2mm}]
		\tikzset{edge/.style={very thick}}
		
		\node[] (a) at (0,0) {}
		node[draw=none, fill=none, left=1mm of a] {\large $a$}
		node[draw=none, fill=none, above left=3mm and 2mm of a] {\LARGE $T$};
		\node[] (b) at (0,-2) {}
		node[draw=none, fill=none, left=1mm of b] {\large $b$};
		\node[fill=none] (c) at (1.5,-2) {}
		node[draw=none, fill=none, left=1mm of c] {\large $c$};
		\node[] (d) at (3,0) {}
		node[draw=none, fill=none, right=1mm of d] {\large $d$};
		\node[] (e) at (3,-2) {}
		node[draw=none, fill=none, right=1mm of e] {\large $e$};
		
		\node[fill=none] (1) at (0.75,-1) {};
		\node[] (2) at (1.50,-1) {};
		\node[fill=none] (3) at (2.25,-1) {};
		
		\draw[edge] 
		(a) -- (1)
		(1) -- (b)
		(1) -- (3)
		(d) -- (3)
		(3) -- (e)
		(2) -- (c);
		\end{tikzpicture}
	\end{subfigure}\hfill
	\begin{subfigure}{.45\textwidth}
		\centering
		\begin{subfigure}[t]{\textwidth}
			\centering
			\begin{tikzpicture}
			[every node/.style={draw, circle, fill, inner sep=0pt, minimum 
				size=2mm}]
			\tikzset{edge/.style={very thick}}
			
			\draw[very thick, black] (0, 0) circle (0.4);
			\node[] (la1) at (-1,{0.4*sin(140)}) {}
			node[draw=none, fill=none, left=1mm of la1] {\large $l_{a1}$}
			node[draw=none, fill=none, above left=3mm and 2mm of la1] {\LARGE 
				$N_1$};
			\node[] (la2) at (-1,{0.4*sin(220)}) {}
			node[draw=none, fill=none, left=1mm of la2] {\large $l_{a2}$};
			\node[] (lav1) at ({0.4*cos(140)},{0.4*sin(140)}) {};
			\node[] (lav2) at ({0.4*cos(220)},{0.4*sin(220)}) {};
			\node[] (lavc) at ({0.4*cos(315)}, {0.4*sin(315)}) {};
			\node[] (lavn) at ({0.4*cos(90)}, {0.4*sin(90)}) {};			
			\node[] (lavs) at ({0.4*cos(-90)}, {0.4*sin(-90)}) {};
			
			\draw[very thick, black] (0,-2) circle (0.4);
			\node[] (lb1) at (-1,{-2 + 0.4*sin(140)}) {}
			node[draw=none, fill=none, left=1mm of lb1] {\large $l_{b1}$};
			\node[] (lb2) at (-1,{-2 + 0.4*sin(220)}) {}
			node[draw=none, fill=none, left=1mm of lb2] {\large $l_{b2}$};
			\node[] (lbv1) at ({0.4*cos(140)},{-2 + 0.4*sin(140)}) {};
			\node[] (lbv2) at ({0.4*cos(220)},{-2 + 0.4*sin(220)}) {};
			\node[] (lbvc) at ({0.4*cos(45)}, {-2 + 0.4*sin(45)}) {};
			\node[] (lbvn) at ({0.4*cos(90)}, {-2 + 0.4*sin(90)}) {};			
			\node[] (lbvs) at ({0.4*cos(-90)}, {-2 + 0.4*sin(-90)}) {};
			
			\draw[very thick, black] (1.5, -2) circle (0.4);			
			\node[] (lc1) at ({1.5 + 0.4*cos(230)},-3) {}
			node[draw=none, fill=none, left=1mm of lc1] {\large $l_{c1}$};
			\node[] (lc2) at ({1.5 + 0.4*cos(310)},-3) {}
			node[draw=none, fill=none, right=1mm of lc2] {\large $l_{c2}$};
			\node[] (lcv1) at ({1.5 + 0.4*cos(230)},{-2 + 0.4*sin(230)}) {};
			\node[] (lcv2) at ({1.5 + 0.4*cos(310)},{-2 + 0.4*sin(310)}) {};
			\node[] (lcvc) at ({1.5 + 0.4*cos(90)}, {-2 + 0.4*sin(90)}) {};

			\draw[very thick, black] (3, 0) circle (0.4);			
			\node[] (ld1) at (4,{0.4*sin(140)}) {}
			node[draw=none, fill=none, right=1mm of ld1] {\large $l_{d1}$};
			\node[] (ld2) at (4,{0.4*sin(220)}) {}
			node[draw=none, fill=none, right=1mm of ld2] {\large $l_{d2}$};
			\node[] (ldv1) at ({3 + 0.4*cos(40)},{0.4*sin(40)}) {};
			\node[] (ldv2) at ({3 + 0.4*cos(-40)},{0.4*sin(-40)}) {};
			\node[] (ldvc) at ({3 + 0.4*cos(225)}, {0.4*sin(225)}) {};
			\node[] (ldvn) at ({3 + 0.4*cos(90)}, {0.4*sin(90)}) {};			
			\node[] (ldvs) at ({3 + 0.4*cos(-90)}, {0.4*sin(-90)}) {};
			
			\draw[very thick, black] (3,-2) circle (0.4);
			\node[] (le1) at (4,{-2 + 0.4*sin(140)}) {}
			node[draw=none, fill=none, right=1mm of le1] {\large $l_{e1}$};
			\node[] (le2) at (4,{-2 + 0.4*sin(220)}) {}
			node[draw=none, fill=none, right=1mm of le2] {\large $l_{e2}$};
			\node[] (lev1) at ({3 + 0.4*cos(40)},  {-2 + 0.4*sin(40)}) {};
			\node[] (lev2) at ({3 + 0.4*cos(-40)}, {-2 + 0.4*sin(-40)}) {};
			\node[] (levc) at ({3 + 0.4*cos(135)}, {-2 + 0.4*sin(135)}) {};
			\node[] (levn) at ({3 + 0.4*cos(90)}, {-2 + 0.4*sin(90)}) {};
			\node[] (levs) at ({3 + 0.4*cos(-90)}, {-2 + 0.4*sin(-90)}) {};
			
			\draw[very thick, black] (1.5, -0.2) circle (0.4);
			\node[] (r1) at ({1.5 + 0.4*cos(180)},  {-0.2 + 0.4*sin(180)}) {};
			\node[] (r2) at ({1.5 + 0.4*cos(0)}, {-0.2 + 0.4*sin(0)}) {};
			\node[] (rc) at (1.5,-0.2) {};
			\node[] (rn) at ({1.5 + 0.4*cos(90)}, {-0.2 + 0.4*sin(90)}) {};
			\node[] (rs) at ({1.5 + 0.4*cos(270)}, {-0.2 + 0.4*sin(270)}) {};
			
			\node[] (1) at (0.75,-1) {};
			\node[] (2) at (2.25,-1) {};
			
			\draw[edge] 
			(la1) -- (lav1)
			(la2) -- (lav2)
			(lb1) -- (lbv1)
			(lb2) -- (lbv2)
			(lc1) -- (lcv1)
			(lc2) -- (lcv2)
			(ld1) -- (ldv1)
			(ld2) -- (ldv2)
			(le1) -- (lev1)
			(le2) -- (lev2)
			(lavc) -- (1)
			(lbvc) -- (1)
			(ldvc) -- (2)
			(levc) -- (2)
			(1) -- (r1)
			(r2) -- (2)
			(lavn) -- (lavs)
			(lbvn) -- (lbvs)
			(ldvn) -- (ldvs)
			(levn) -- (levs)
			(rn) -- (rs);
			
			\draw[edge, bend left, dotted] (rc) edge (lcvc);
			
			\end{tikzpicture}
		\end{subfigure}
	
		\begin{subfigure}[t]{\textwidth}
			\centering
			\begin{tikzpicture}
			[every node/.style={draw, circle, fill, inner sep=0pt, minimum 
				size=2mm}]
			\tikzset{edge/.style={very thick}}
			
			\draw[very thick, black] (0, 0) circle (0.4);
			\node[] (la1) at (-1,{0.4*sin(140)}) {}
			node[draw=none, fill=none, left=1mm of la1] {\large $l_{a1}$}
			node[draw=none, fill=none, above left=3mm and 2mm of la1] {\LARGE 
				$N_2$};
			\node[] (la2) at (-1,{0.4*sin(220)}) {}
			node[draw=none, fill=none, left=1mm of la2] {\large $l_{a2}$};
			\node[] (lav1) at ({0.4*cos(140)},{0.4*sin(140)}) {};
			\node[] (lav2) at ({0.4*cos(220)},{0.4*sin(220)}) {};
			\node[] (lavc) at ({0.4*cos(315)}, {0.4*sin(315)}) {};			
			
			\draw[very thick, black] (0,-2) circle (0.4);
			\node[] (lb1) at (-1,{-2 + 0.4*sin(140)}) {}
			node[draw=none, fill=none, left=1mm of lb1] {\large $l_{b1}$};
			\node[] (lb2) at (-1,{-2 + 0.4*sin(220)}) {}
			node[draw=none, fill=none, left=1mm of lb2] {\large $l_{b2}$};
			\node[] (lbv1) at ({0.4*cos(140)},{-2 + 0.4*sin(140)}) {};
			\node[] (lbv2) at ({0.4*cos(220)},{-2 + 0.4*sin(220)}) {};
			\node[] (lbvc) at ({0.4*cos(45)}, {-2 + 0.4*sin(45)}) {};			
			
			\draw[very thick, black] (1.5, -2) circle (0.4);			
			\node[] (lc1) at ({1.5 + 0.4*cos(230)},-3) {}
			node[draw=none, fill=none, left=1mm of lc1] {\large $l_{c1}$};
			\node[] (lc2) at ({1.5 + 0.4*cos(310)},-3) {}
			node[draw=none, fill=none, right=1mm of lc2] {\large $l_{c2}$};
			\node[] (lcv1) at ({1.5 + 0.4*cos(230)},{-2 + 0.4*sin(230)}) {};
			\node[] (lcv2) at ({1.5 + 0.4*cos(310)},{-2 + 0.4*sin(310)}) {};
			\node[] (lcvc) at ({1.5 + 0.4*cos(90)}, {-2 + 0.4*sin(90)}) {};
			\node[] (lcw) at ({1.5 + 0.4*cos(180)}, {-2 + 0.4*sin(180)}) {};
			\node[] (lce) at ({1.5 + 0.4*cos(0)}, {-2 + 0.4*sin(0)}) {};
			
			\draw[very thick, black] (3, 0) circle (0.4);			
			\node[] (ld1) at (4,{0.4*sin(140)}) {}
			node[draw=none, fill=none, right=1mm of ld1] {\large $l_{d1}$};
			\node[] (ld2) at (4,{0.4*sin(220)}) {}
			node[draw=none, fill=none, right=1mm of ld2] {\large $l_{d2}$};
			\node[] (ldv1) at ({3 + 0.4*cos(40)},{0.4*sin(40)}) {};
			\node[] (ldv2) at ({3 + 0.4*cos(-40)},{0.4*sin(-40)}) {};
			\node[] (ldvc) at ({3 + 0.4*cos(225)}, {0.4*sin(225)}) {};			
			
			\draw[very thick, black] (3,-2) circle (0.4);
			\node[] (le1) at (4,{-2 + 0.4*sin(140)}) {}
			node[draw=none, fill=none, right=1mm of le1] {\large $l_{e1}$};
			\node[] (le2) at (4,{-2 + 0.4*sin(220)}) {}
			node[draw=none, fill=none, right=1mm of le2] {\large $l_{e2}$};
			\node[] (lev1) at ({3 + 0.4*cos(40)},  {-2 + 0.4*sin(40)}) {};
			\node[] (lev2) at ({3 + 0.4*cos(-40)}, {-2 + 0.4*sin(-40)}) {};
			\node[] (levc) at ({3 + 0.4*cos(135)}, {-2 + 0.4*sin(135)}) {};
			
			\draw[very thick, black] (0.75, -1) circle (0.4);
			\node[] (l1) at (0.75, -1) {};
			\node[] (l2) at ({0.75 + 0.4*cos(-135)}, {-1 + 0.4*sin(-135)}) {};
			\node[] (lc) at ({0.75 + 0.4*cos(0)}, {-1 + 0.4*sin(0)}) {};
			\node[] (ln) at ({0.75 + 0.4*cos(90)}, {-1 + 0.4*sin(90)}) {};
			\node[] (ls) at ({0.75 + 0.4*cos(270)}, {-1 + 0.4*sin(270)}) {};
			
			\draw[very thick, black] (2.25, -1) circle (0.4);
			\node[] (r1) at (2.25, -1) {};
			\node[] (r2) at ({2.25 + 0.4*cos(-45)}, {-1 + 0.4*sin(-45)}) {};
			\node[] (rc) at ({2.25 + 0.4*cos(180)}, {-1 + 0.4*sin(180)}) {};
			\node[] (rn) at ({2.25 + 0.4*cos(90)}, {-1 + 0.4*sin(90)}) {};
			\node[] (rs) at ({2.25 + 0.4*cos(270)}, {-1 + 0.4*sin(270)}) {};
			
			\node[] (origin) at (1.50,-1) {};
			
			\draw[edge] 
			(la1) -- (lav1)
			(la2) -- (lav2)
			(lb1) -- (lbv1)
			(lb2) -- (lbv2)
			(lc1) -- (lcv1)
			(lc2) -- (lcv2)
			(ld1) -- (ldv1)
			(ld2) -- (ldv2)
			(le1) -- (lev1)
			(le2) -- (lev2)
			(lbvc) -- (l2)
			(levc) -- (r2)
			(ln) -- (ls)
			(rn) -- (rs)
			(lc) -- (origin)
			(origin) -- (rc)
			(origin) -- (lcvc)
			(lcw) -- (lce);
			
			\draw[edge, dashed, bend right]
			(lavc) edge (l1);
			
			\draw[edge, dashed, bend left]
			(ldvc) edge (r1);

			\end{tikzpicture}
		\end{subfigure}
	\end{subfigure}
	\caption{An example of obtaining two alt-path structures from a binary 
	tree~$T$. The network~$N_1$ is obtained by replacing all filled internal 
	vertices by a level-$2$ blob with each cut-edge subdividing the three 
	different sides, filled leaf vertices by a pendant level-$2$ blob of the 
	form~$(k,0,0,0)$ for~$k=2$ (note that this can also be~$k=3$), and unfilled 
	leaf vertices by a pendant level-$1$ blob of~$2$ or~$3$ leaves. The 
	similar alt-path structure~$N_2$ is obtained by the same construction, with 
	the roles of 
	filled and unfilled vertices reversed. Observe that~$N_1$ and~$N_2$ has the 
	same shortest distance matrix, as stated in Lemma~\ref{lem:SimilarAltPath}.}
	\label{fig:AltPath}
\end{figure}
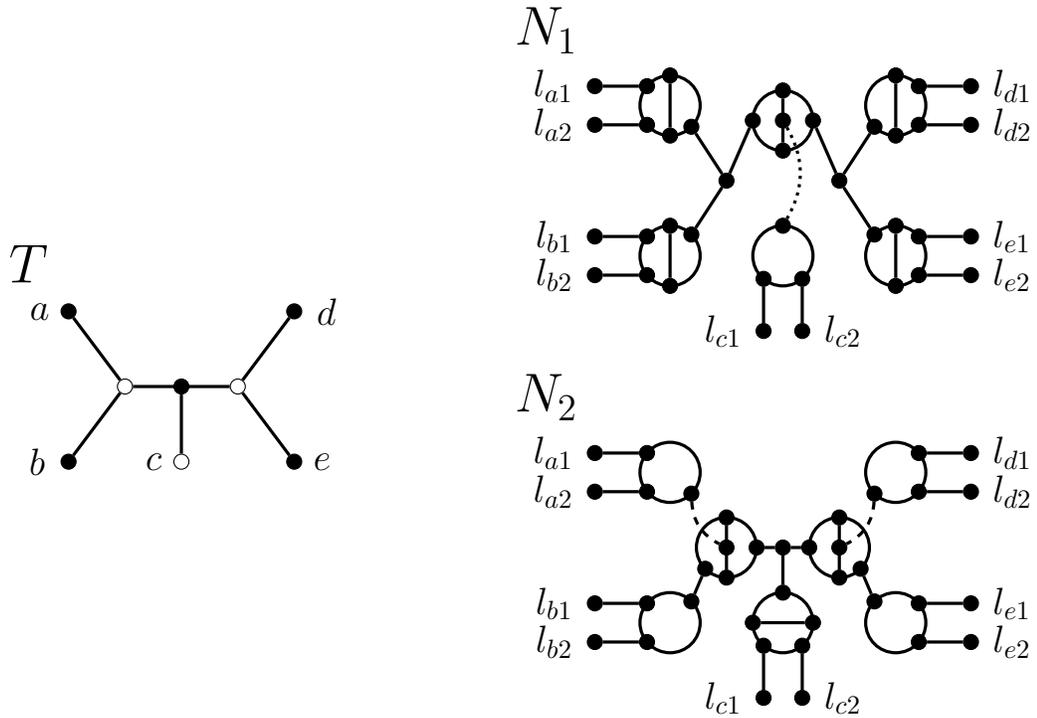

\begin{lemma}\label{lem:SimilarAltPath}
	Similar alt-path structures of a given binary tree realize the same 
	shortest distance matrix.
\end{lemma}
\begin{proof}
	Let~$N$ be a level-$2$ network that is an alt-path structure~$G$ of some 
	binary tree~$T$. Let~$N'$ be a network obtained from~$N$ by replacing~$G$ 
	by its similar alt-path structure.
	
	Let~$x$ and~$y$ denote two leaves in~$N$. The two networks~$N$ and~$N'$ 
	both contain the same chains by construction. Furthermore, each chain is of 
	length at most~$3$. Thus we have that~$d^N(x,y) = d^{N'}(x,y)$ if~$x$ 
	and~$y$ are contained in the same chains. So we may assume that~$x$ 
	and~$y$ are contained in different chains. We wish to show that~$d^N(x,y) = 
	d^{N'}(x,y)$. 
	Consider the leaves~$l_x$ and~$l_y$ of~$T$ that were replaced by 
	the pendant blobs containing~$x$ and~$y$, respectively. If~$d_T(l_x,l_y)$ 
	is odd, then there is an even number of internal vertices in the path 
	between~$l_x$ and~$l_y$ in~$T$. This means that~$N$ and~$N'$ contain the 
	same number of non-pendant level-$2$ blobs and the same number of non-leaf 
	vertices not contained in blobs in the shortest path between~$x$ 
	and~$y$. Moreover, due to parity, exactly one of the two leaves
	will be contained in a pendant level-$1$ blob in~$N$ and the other leaf in 
	a pendant level-$2$ blob. The reverse is true for~$N'$. Thus it follows 
	that if~$d_T(l_x,l_y)$ is odd, then~$d^N(x,y) = d^{N'}(x,y)$. 
	Now if~$d_T(l_x,l_y)$ is even, the number of non-pendant level-$2$ blobs in 
	the shortest path between~$x$ and~$y$ will be greater by one in either~$N$ 
	or in~$N'$. Without loss of generality, suppose that~$N$ has this 
	property. But this difference is offset by the fact that the pendant blobs 
	in this path are both level-$1$ in~$N$, whereas they are both level-$2$ 
	in~$N'$. Therefore $d^N(x,y) = d^{N'}(x,y)$.
\end{proof}

\begin{corollary}\label{lem:CEShortest}
	Let~$N$ be a level-$2$ network containing an alt-path structure~$G$ of some 
	binary tree~$T$ as a subgraph. Then~$N$ is not reconstructible from its 
	shortest distance matrix.
\end{corollary}
\begin{proof}
	Let~$N'$ denote the network obtained from~$N$ by replacing~$G$ by its 
	similar alt-path structure. We claim that the distinct networks~$N$ 
	and~$N'$ must realize the same shortest distance matrix, thereby proving 
	that~$N$ is not reconstructible from its shortest distance matrix. Consider 
	any two leaves~$x$ and~$y$ of~$N$, and let~$P$ denote a shortest path 
	between~$x$ and~$y$ in~$N$. If~$P$ does not contain any edges of~$G$, then 
	a path between~$x$ and~$y$ on the same length must exist in~$N'$, since 
	only the subgraph~$G$ of~$N$ was changed to obtain~$N'$. On the other hand, 
	if~$P$ contains an edge of~$G$, then~$P$ must contain exactly one path 
	of~$G$ that starts and ends at two leaves~$l_1,l_2$ of~$G$. Since only the 
	subgraph~$G$ of~$N$ was changed to obtain~$N'$, we have that~$d^N(x,l_1) = 
	d^{N'}(x,l_1)$, and that~$d^N(l_2,y) = d^{N'}(l_2,y)$. By 
	Lemma~\ref{lem:SimilarAltPath}, we have that~$d^N(l_1,l_2) = 
	d^{N'}(l_1,l_2)$. So a path between~$x$ and~$y$ on the same length must 
	also exist in~$N'$. It follows that~$d^{N'}(x,y) \le d^N(x,y)$.
	
	Now consider a shortest path~$Q$ between~$x$ and~$y$ in~$N'$. By applying 
	the same arguments to~$Q$, but with the alt-path structure that is similar 
	to~$G$, we 
	conclude that~$d^N(x,y) \le d^{N'}(x,y)$. This proves that~$d^N(x,y) = 
	d^{N'}(x,y)$.
\end{proof}

It is now easy to explain why the two networks in Figure~\ref{fig:level2CE} 
realize the same shortest distance matrix; they contain similar alt-path 
structures of a binary tree on two leaves.
We show in the next subsection that the converse of 
Corollary~\ref{lem:CEShortest}, that a level-$2$ network containing no alt-path 
structure is reconstructible, is also true. 

\subsection{Level-$2$ networks without alt-path structures are reconstructible}

We introduce some more terminology. Let~$N$ be a level-$2$ network. A 
\emph{blob tree} of~$N$ is the graph obtained by contracting all edges of 
blobs, deleting all labelled leaves, and suppressing all degree-$2$ vertices. A 
vertex of a blob-tree is called a 
\emph{blob-vertex}. 
We define the \emph{connection} of a pendant blob as the endpoint of the 
non-trivial cut-edge incident to the blob that is not on the blob.
We say that a blob~$B$ \emph{contains} a pendant blob~$C$ 
if the connection of~$C$ is a vertex of~$B$. For any blob~$B$ in~$N$, we 
let~$l(B)$ denote the level of~$B$.

Let~$P_1$ be a bad pendant blob on two leaves~$l_1,l_2$, and let~$u$ be a 
vertex 
of~$N$ \rev{that is not a neighbor of~$l_1$ nor~$l_2$}. We let~$d^N(P_1,u) = d^N(l_1,u)$ denote the shortest distance 
between~$P_1$ and~$u$. This is well-defined for all bad pendant blobs 
containing 
exactly two leaves, since the shortest distance from either of the two leaves 
to any other vertex in the network is the same. Let~$P_2$ be another bad 
pendant 
blob on two leaves~$l'_1,l'_2$. Then we may similarly define the shortest 
distance between~$P_1$ and~$P_2$ by~$d^N(P_1,P_2) = d^N(l_1,l'_1)$. This again 
is well-defined as the two bad pendant blobs both contain two leaves.

Let~$P_1$ and~$P_2$ be two pendant blobs that are contained in the same 
blob~$B$. Let~$p_1$ and~$p_2$ be the connections of~$P_1$ and~$P_2$, 
respectively. We say that~$P_1$ and~$P_2$ are \emph{adjacent} if~$p_1$ 
and~$p_2$ 
are adjacent. Let~$l$ be a leaf that is not contained 
in~$P_1$. We say that~$l$ and~$P_1$ are \emph{adjacent} if the neighbor of~$l$ 
is adjacent to~$p_1$. We say that~$P_1$ is \emph{adjacent} to a chain of 
leaves~$(a,k)$ if~$P_1$ is adjacent to an end-leaf of~$(a,k)$.
%

\begin{lemma}\label{lem:BadBlob2}
	Let~$N$ be a level-$2$ network on~$X$ containing a bad pendant blob with~$3$ 
	leaves~$(a_1,a_2,a_3)$, and let~$N'$ denote the network obtained by 
	deleting~$a_2$ from~$N$. Then the shortest distance matrix realized by~$N'$ 
	is given by
	\[ d^{N'}(x,y) = 
	\begin{cases} 
	d^N(x,y)      & \text{if }x,y\in X-\{a_2\} \text{ and }\{x,y\}\ne 
	\{a_1,a_3\}; \\
	3  & \text{if }\{x,y\} = \{a_1,a_3\}.
	\end{cases}
	\]
\end{lemma}
\begin{proof}
	The only shortest paths containing the edges incident to the neighbor 
	of~$a_2$ in~$N$ were those involving~$a_2$, or a path between~$a_1$ 
	and~$a_3$. Since~$a_2$ is no longer a leaf in~$N'$, the only path that is 
	affected in the leaf deletion is the shortest path between~$a_1$ and~$a_3$, which is 
	now of length~$3$ in~$N'$.
\end{proof}

We are now ready to prove the main theorem of the section. Because the proof 
exhaustively checks for contradictions within each case, it is rather long, and 
so we split the three main cases of the proof into 
subsubsections. In each 
subsubsection, a short summary will be given to clarify the proof steps. 

\begin{theorem}\label{thm:NoAltPath}
	A level-$2$ network containing no alt-path structure is reconstructible 
	from its shortest distances.
\end{theorem}

\begin{proof}
	Suppose for a contradiction that there exist distinct networks~$N$ and~$N'$ 
	with no alt-path structures that realize the same shortest distance matrix. 
	In particular, choose~$N$ and~$N'$ to be minimal counter-examples with 
	respect to the size of the networks. This means that every pendant blob of 
	$N$ and~$N'$ must be a bad blob; indeed, all other pendant 
	blobs are identifiable from the shortest distance matrix by 
	Lemma~\ref{lem:AdjustDistances}, and thus can be reduced 
	otherwise, allowing for smaller counter-examples to exist. We may 
	further assume that all pendant bad blobs contain exactly two leaves, as 
	otherwise we can find a smaller counter-example by calling 
	Lemma~\ref{lem:BadBlob2}. Finally, observe that if~$N$ contains a pendant 
	level-$1$ blob with some leaves~$l_1,l_2$, then~$N'$ must contain a pendant 
	level-$2$ blob of the form~$(2,0,0,0)$ containing the leaves~$l_1,l_2$, as 
	again we would be able to obtain a smaller counter-example otherwise. 
	Similarly, if~$N$ contains a pendant level-$2$ blob of the form~$(2,0,0,0)$ 
	with leaves~$l_1,l_2$ then~$N'$ must contain a pendant level-$1$ blob with 
	leaves~$l_1,l_2$. If~$N$ contains a pendant blob~$P_i$, we say that~$N'$ 
	contains the \emph{corresponding} pendant blob~$P'_i$ on the same leaves 
	such that~$l(P_i)\neq l(P'_i)$. For each pendant blob~$P_i, P'_i$ 
	in~$N,N'$, we let~$p_i,p'_i$ denote their connections, respectively.
		
	
	Since~$N$ and~$N'$ realize the same shortest distance matrix, the two 
	networks have the same cut-edge induced splits by 
	Theorem~\ref{thm:CutEdgeSplit}. Since each cut-edge in our networks induces 
	a unique split, this implies that their blob-trees must be identical. 
	\rev{This follows as every edge of the blob tree is a cut-edge of the network, and because trees are uniquely determined by their induced splits~\citep{buneman1971recovery}.}
	Note 
	that the blob-vertices of the blob-trees correspond to either a 
	degree-$3$ vertex, a level-$1$ blob, or a level-$2$ blob of the network. 
	Observe that the blob-tree of~$N$ must contain at least~$3$ blob-vertices, 
	as otherwise~$N$ 
	would be a level-$2$ network with a single blob -- which is reconstructible 
	from their shortest distances by Lemma~\ref{lem:L1} -- or a level-$2$ 
	network with two pendant 
	blobs, which implies that~$N$ must contain an alt-path structure as~$N$ 
	contains only bad pendant blobs, or~$N$ 
	and~$N'$ cannot realize the same shortest distance matrix. 
	
	Consider the blob-vertex~$u$ whose neighbors are all leaves, except 
	possibly for one neighbor; let~$uv$ denote the edge in the blob-tree to 
	this one neighbor. Since every edge of the blob-tree corresponds to a 
	non-trivial cut-edge in the network, the edge~$uv$ must be incident to 
	a degree-$3$ vertex / blob which correspond to~$u$ in~$N$ and~$N'$. 
	Let~$B,B'$ denote the corresponding structures in~$N,N'$, respectively.  
	Then~$B$ contains at least one pendant blob, possibly some chain of leaves, 
	and the cut-edge~$uv$ incident to it, with~$u$ as a vertex of~$B$. The same 
	can be said for~$B'$, but we use~$u'$ instead of~$u$ to be the vertex 
	of~$B'$ incident to the cut-edge for clarity. Note that it is possible 
	for~$u$ and~$u'$ to be a connection of some pendant blob. Let~$\bar{B}$ be 
	the graph obtained from~$N$ by deleting the edge~$uv$ and taking the 
	component containing~$u$. Similarly define~$\bar{B'}$ as the graph obtained 
	from~$N'$ by deleting the edge~$u'v$ and taking the component 
	containing~$u'$. 
	Since we have deleted the edges corresponding to the same edge in the 
	blob-tree,~$\bar{B}$ and~$\bar{B'}$ contain the same chains, and~$\bar{B}$ 
	contains a pendant blob if and only if~$\bar{B'}$ contains the 
	corresponding pendant blob.
	Observe 
	that~$\bar{B}$ and~$\bar{B'}$ are not 
	networks because they contain a degree-$2$ vertex~$u$ and~$u'$, 
	respectively. However, to avoid having to introduce new notation, we shall 
	still use terms defined for networks, such 
	as blobs containing pendant blobs, pendant blobs being adjacent to one 
	another in~$\bar{B}$ and~$\bar{B'}$.
	
	The rest of the proof will be as follows. We consider the cases where~$B$ 
	is 
	a degree-$3$ vertex, a level-$1$ blob, or a level-$2$ blob. Based on the 
	graph~$\bar{B}$ in comparison with the graph~$\bar{B'}$, we seek a 
	contradiction with regards to the 
	networks realizing the shortest distance matrix and the choice of the 
	minimum counter-example. Because the results are symmetric, it is worth 
	mentioning that once we have proven the case for when~$B$ is a 
	degree-$3$ vertex, then we may assume that~$B'$ is also not a degree-$3$ 
	vertex. After we prove the case for when~$B$ is a level-$1$ blob, then 
	we may assume that~$B'$ is also not a level-$1$ blob.
	
	The following claim will be used extensively throughout the proof.
	\begin{claim}
	\label{cla:DistToCE}
		Let~$x,y$ be leaves in~$\bar{B}$. 
		Then \[d^N(x,u)-d^{N'}(x,u') = d^N(y,u) - d^{N'}(y,u').\]
	\end{claim}
	\begin{proof}
		Let~$z$ be a leaf of~$N$ that is not in~$\bar{B}$. Such a leaf must 
		exist by our choice of~$B$, and in particular,~$z$ must be reachable 
		from~$B$ via~$uv$. Then we have
		\begin{align*}
		d^N(z,x) &= d^N(z,u) + d^N(u,x)\\
		d^N(z,y) &= d^N(z,u) + d^N(u,y),
		\end{align*}
		which gives
		\[d^N(z,x) - d^N(z,y) = d^N(x,u) - d^N(y,u).\]
		Similarly we have
		\[d^{N'}(z,x) - d^{N'}(z,y) = d^{N'}(x,u') - d^{N'}(y,u').\]
		Since the shortest distance matrices of~$N$ and~$N'$ are the same, we 
		have
		\[d^N(x,u)-d^{N'}(x,u') = d^N(y,u) - d^{N'}(y,u').\]
	\end{proof}
	
	\subsubsection{$\bm{B}$ is a degree-$\bm{3}$ vertex 
	in~$\bm{N}$:}\label{subsub:deg3} 
	Suppose first that a leaf~$l$ is a neighbor of~$u$ in~$\bar{B}$, and 
	let~$P_1$ be a pendant blob in~$\bar{B}$ whose connection is~$u$. 
	Since~$N'$ has the same cut-edge induced splits,~$\bar{B'}$ must either be 
	a degree-$3$ vertex or a blob that contains~$l$ and the corresponding 
	pendant blob~$P'_1$. But then
	\[d^N(P_1,l) = 4,\]
	whereas
	\[d^{N'}(P'_1,l) \ge 5,\]
	which contradicts the fact that~$N$ and~$N'$ must realize the same shortest 
	distance matrix.
	
	So now suppose that~$u$ is the connection of two pendant blobs~$P_1$ 
	and~$P_2$ in~$\bar{B}$. We check the three possible scenarios with regards 
	to the levels of~$P_1$ and~$P_2$.
	\begin{enumerate}
		\item \textbf{$\bm{l(P_1)=1}$ and $\bm{l(P_2)=1}$:} Then we 
		have~$d^N(P_1,P_2) = 6$. But since~$l(P'_1) = l(P'_2) = 
		2$, we must have that~$d^{N'}(P'_1,P'_2) \ge 8$. This contradicts the 
		fact that~$N$ and~$N'$ realize the same shortest distance matrix.
		\item \textbf{$\bm{l(P_1)=1}$ and $\bm{l(P_2)=2}$:} Then we 
		have~$d^N(P_1,P_2) = 7$. Since~$l(P'_1) = 2$ and~$l(P'_2) = 1$, we have 
		that~$d^{N'}(P'_1,P'_2) \ge 7$, where equality is achieved 
		whenever~$u'$ is a degree-$3$ vertex. But this would mean that
		\[d^N(P_1,u)-d^{N'}(P'_1,u') = -1,\]
		whereas
		\[d^N(P_2,u) - d^{N'}(P'_2,u') = 1,\]
		which contradicts Claim~\ref{cla:DistToCE}.
		\item \textbf{$\bm{l(P_1)=2}$ and $\bm{l(P_2)=2}$ (see 
		Figure~\ref{fig:thm1c} for an illustration of the cases):} 
		Then~$d^N(P_1,P_2) = 8$. Since~$l(P'_1) = l(P'_2) = 1$, and 
		because~$N'$ contains a split with one of the sets containing 
		exactly the leaves of~$P'_1$ and~$P'_2$,~$B'$ must be a level-$2$ 
		blob.
		In particular,~$P'_1$ and~$P'_2$ cannot be adjacent. There are two 
		possibilities for this -- $u'$ is a neighbor of one of~$p'_1$ 
		or~$p'_2$ but not the other, or all three vertices~$u',p'_1,p'_2$ are 
		pairwise non-adjacent (i.e., they all lie on different main paths 
		of~$\bar{B'}$). In the former case, we have that, assuming without loss 
		of generality that~$u'$ is adjacent to~$p'_1$,
		\[d^N(P_1,u)-d^{N'}(P'_1,u') = 4-4 = 0,\] 
		whereas
		\[d^N(P_2,u) - d^{N'}(P'_2,u') = 4-5 = -1,\]
		which contradicts Claim~\ref{cla:DistToCE}. 
		For the latter case, we claim that there is a smaller counter-example. 
		We replace~$\bar{B}$ in~$N$ by a pendant level-$1$ blob~$P_3$ 
		containing two leaves~$l_1,l_2$. We replace~$\bar{B'}$ in~$N'$ by a 
		pendant level-$2$ blob~$P'_3$ of the form~$(2,0,0,0)$ with the same 
		leaves~$l_1,l_2$. Then, we may adjust the shortest distance matrix by 
		first deleting elements containing the leaves of~$P_1$ and~$P_2$. And 
		for all other leaves~$z$ in the network, we add the elements
		\[d^N(l_i,z) = d^N(P_1,z) - 2,\]
		and
		\[d^{N'}(l_i,z) = d^N(P'_1,z) - 2\]
		for~$i=1,2$. All other matrix elements remain the same. Note that 
		before the replacement of the blobs,
		\[d^N(P_1,u) = d^N(P_2,u) = 4,\]
		and
		\[d^{N'}(P'_1,u') = d^{N'}(P'_2,u') = 5.\]
		The replacement of~$\bar{B}$ and~$\bar{B'}$ by pendant level-$1$ and 
		level-$2$ blobs, respectively ensures that the distance differences are 
		preserved. Therefore the modified networks both must satisfy this new 
		reduced shortest distance matrix. These modified networks~$N$ and~$N'$ 
		still contain no alt-path structures, as otherwise the original 
		networks also must have contained alt-path structures; all other parts 
		of the networks remain unchanged, and the two leaves~$l_1$ and~$l_2$ 
		are contained in pendant blobs of different level in~$N$ and~$N'$. 
		Therefore, this gives a counter-example on fewer leaves than that 
		of~$N$ and~$N'$, contradicting our original choice of~$N$ and~$N'$. 
	\end{enumerate}

	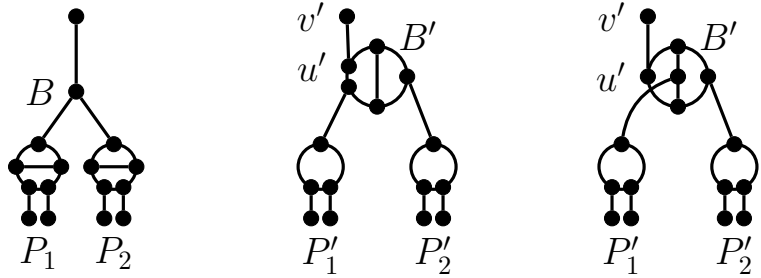
\begin{figure}
		\centering
		\begin{tikzpicture}
		[every node/.style={draw, circle, fill, inner sep=0pt, minimum 
		size=2mm}]
		\tikzset{edge/.style={very thick}}
		
		\node[] (v) at (0,1) {};
		\node[] (B) at (0,0) {}
		node[draw = none, fill = none, left=1mm of B]{\large $B$};
		
		\draw[very thick, black] (-0.5,-1) circle (0.3);
		\node[] (b1n) at ({-0.5 + 0.3*cos(90)}, {-1 + 0.3*sin(90)}) {};
		\node[] (b1w) at ({-0.5 + 0.3*cos(180)}, {-1 + 0.3*sin(180)}) 
		{};
		\node[] (b1e) at ({-0.5 + 0.3*cos(0)}, {-1 + 0.3*sin(0)}) {};
		\node[] (b1l1) at ({-0.5 + 0.3*cos(245)}, {-1 + 0.3*sin(245)}) 
		{};
		\node[] (b1l2) at ({-0.5 + 0.3*cos(295)}, {-1 + 0.3*sin(295)}) 
		{};
		\node[below=2mm of b1l1] (b1l1l) {};
		\node[below=2mm of b1l2] (b1l2l) {};
		\node[draw=none, fill=none, below=10mm of b1n] {\large $P_1$};
						
		\draw[very thick, black] (0.5,-1) circle (0.3);
		\node[] (b2n) at ({0.5 + 0.3*cos(90)}, {-1 + 0.3*sin(90)}) {};
		\node[] (b2w) at ({0.5 + 0.3*cos(180)}, {-1 + 0.3*sin(180)}) {};
		\node[] (b2e) at ({0.5 + 0.3*cos(0)}, {-1 + 0.3*sin(0)}) {};
		\node[] (b2l1) at ({0.5 + 0.3*cos(245)}, {-1 + 0.3*sin(245)}) 
		{};
		\node[] (b2l2) at ({0.5 + 0.3*cos(295)}, {-1 + 0.3*sin(295)}) 
		{};
		\node[below=2mm of b2l1] (b2l1l) {};
		\node[below=2mm of b2l2] (b2l2l) {};
		\node[draw=none, fill=none, below=10mm of b2n] {\large $P_2$};
		
		\draw[edge] 
		(B) -- (b1n)
		(B) -- (b2n)
		(b1w) -- (b1e)
		(b2w) -- (b2e)
		(b1l1) -- (b1l1l)
		(b1l2) -- (b1l2l)
		(b2l1) -- (b2l1l)
		(b2l2) -- (b2l2l)
		(B) -- (v);

		\begin{scope}[xshift=8cm,yshift=0cm]
		\draw[very thick, black] (0,0.2) circle (0.4)
		node[draw = none, fill = none, above right =3mm and 3mm]{\large 
		$B'$};
		\node[] (Bn) at ({0.4*cos(90)},{0.2+ .4*sin(90)}) {};
		\node[] (Bs) at ({0.4*cos(270)},{0.2+ .4*sin(270)}) {};
		\node[] (Bw) at ({0.4*cos(180)},{0.2+ .4*sin(180)}) {}
		node[draw=none, fill=none, left=1mm of Bw] {\large $u'$};
		\node[] (Be) at ({0.4*cos(0)},{0.2+ .4*sin(0)}) {};
		\node[] (center) at (0,0.2) {};
		\node[] (v) at (-0.4, 1) {}
		node[draw=none, fill=none, left=1mm of v] {\large $v'$};
		
		\draw[very thick, black] (-0.75,-1) circle (0.3);
		\node[] (b1n) at ({-0.75 + 0.3*cos(90)}, {-1 + 0.3*sin(90)}) {};
		\node[] (b1l1) at ({-0.75 + 0.3*cos(245)}, {-1 + 0.3*sin(245)}) 
		{};
		\node[] (b1l2) at ({-0.75 + 0.3*cos(295)}, {-1 + 0.3*sin(295)}) 
		{};
		\node[below=2mm of b1l1] (b1l1l) {};
		\node[below=2mm of b1l2] (b1l2l) {};
		\node[draw=none, fill=none, below=10mm of b1n] {\large $P'_1$};
		
		\draw[very thick, black] (0.75,-1) circle (0.3);
		\node[] (b2n) at ({0.75 + 0.3*cos(90)}, {-1 + 0.3*sin(90)}) {};
		\node[] (b2l1) at ({0.75 + 0.3*cos(245)}, {-1 + 0.3*sin(245)}) 
		{};
		\node[] (b2l2) at ({0.75 + 0.3*cos(295)}, {-1 + 0.3*sin(295)}) 
		{};
		\node[below=2mm of b2l1] (b2l1l) {};
		\node[below=2mm of b2l2] (b2l2l) {};
		\node[draw=none, fill=none, below=10mm of b2n] {\large $P'_2$};
		
		\draw[edge] 
		(Bn) -- (Bs)
		(Bw) -- (v)
		(Be) -- (b2n)
		(b1l1) -- (b1l1l)
		(b1l2) -- (b1l2l)
		(b2l1) -- (b2l1l)
		(b2l2) -- (b2l2l);
		
		\draw[edge, bend right] (center) edge (b1n);
		\end{scope}
		
		\begin{scope}[xshift=4cm,yshift=0cm]
		\draw[very thick, black] (0,0.2) circle (0.4)
		node[draw = none, fill = none, above right= 3mm and 3mm]{\large 
		$B'$};
		\node[] (Bn) at ({0.4*cos(90)},{0.2+ .4*sin(90)}) {};
		\node[] (Bs) at ({0.4*cos(270)},{0.2+ .4*sin(270)}) {};
		\node[] (Bw) at ({0.4*cos(200)},{0.2+ .4*sin(200)}) {};
		\node[] (Be) at ({0.4*cos(0)},{0.2+ .4*sin(0)}) {};
		\node[] (Bu) at ({0.4*cos(160)},{0.2+ .4*sin(160)}) {}			
		node[draw=none, fill=none, left=1mm of Bu] {\large $u'$};
		\node[] (v) at (-0.4, 1) {}			
		node[draw=none, fill=none, left=1mm of v] {\large $v'$};
		
		\draw[very thick, black] (-0.75,-1) circle (0.3);
		\node[] (b1n) at ({-0.75 + 0.3*cos(90)}, {-1 + 0.3*sin(90)}) {};
		\node[] (b1l1) at ({-0.75 + 0.3*cos(245)}, {-1 + 0.3*sin(245)}) 
		{};
		\node[] (b1l2) at ({-0.75 + 0.3*cos(295)}, {-1 + 0.3*sin(295)}) 
		{};
		\node[below=2mm of b1l1] (b1l1l) {};
		\node[below=2mm of b1l2] (b1l2l) {};
		\node[draw=none, fill=none, below=10mm of b1n] {\large $P'_1$};
		
		\draw[very thick, black] (0.75,-1) circle (0.3);
		\node[] (b2n) at ({0.75 + 0.3*cos(90)}, {-1 + 0.3*sin(90)}) {};
		\node[] (b2l1) at ({0.75 + 0.3*cos(245)}, {-1 + 0.3*sin(245)}) 
		{};
		\node[] (b2l2) at ({0.75 + 0.3*cos(295)}, {-1 + 0.3*sin(295)}) 
		{};
		\node[below=2mm of b2l1] (b2l1l) {};
		\node[below=2mm of b2l2] (b2l2l) {};
		\node[draw=none, fill=none, below=10mm of b2n] {\large $P'_2$};
		
		\draw[edge] 
		(Bn) -- (Bs)
		(Bw) -- (b1n)
		(Be) -- (b2n)
		(b1l1) -- (b1l1l)
		(b1l2) -- (b1l2l)
		(b2l1) -- (b2l1l)
		(b2l2) -- (b2l2l)
		(Bu) -- (v);
		
		\end{scope}
		\end{tikzpicture}
		\caption{Subsubsection~\ref{subsub:deg3} case 3 in the proof of 
		Theorem~\ref{thm:NoAltPath}. The 
		left figure is the part of the network~$N$ containing the 
		internal vertex~$B$ and its neighboring pendant blobs. The middle 
		and the right figures are the two subcases for what~$N'$ could look 
		like. In the middle network, two cut-edges are incident to the same 
		side; in the right network, the three cut-edges are incident to 
		distinct sides of~$B'$.}
		\label{fig:thm1c}
	\end{figure}

	\subsubsection{$\bm{B}$ is a level-$\bm{1}$ blob in~$\bm{N}$:}
	
	We may now also assume that~$B'$ is either a level-$1$ or a level-$2$ blob. 
	Note that either~$\bar{B}$ or~$\bar{B'}$ must contain a pendant level-$1$ 
	blob, and that we shall obtain a contradiction in each of those cases. 
	Before we do so, 
	we prove a claim that will be used in many of the arguments to 
	come. In the following claim, we assume that $B$ is either a level-$1$ or a 
	level-$2$ blob.
	
	\begin{claim}\label{cla:4Statements}
		Suppose~$l(B),l(B')\in\{1,2\}$, and suppose that~$\bar{B}$ contains a 
		pendant level-$1$ blob~$P_1$. Then,
		\begin{enumerate}[label=(\roman*)]
			\item $P_1$ cannot be adjacent to a chain of leaves~$(a,k)$ 
			in~$\bar{B}$;
			\item $P_1$ cannot be adjacent to another pendant level-$1$ blob 
			in~$\bar{B}$;
			\item $P_1$ is adjacent to a pendant level-$2$ blob~$P_2$ 
			in~$\bar{B}$ if and only if~$P'_1$ and~$P'_2$ are adjacent 
			in~$\bar{B'}$;
			\item $P_1$ is adjacent to at most one pendant level-$2$ blob 
			in~$\bar{B}$. In particular, this means that every pendant 
			level-$2$ blob in~$\bar{B}$ is adjacent to at most one pendant 
			level-$1$ blob.
			\item $P_1$ can be shortest distance~$6$ away from at most two 
			end-leaves of distinct chains in~$\bar{B}$.
		\end{enumerate}
	\end{claim}
	\begin{proof}
		\begin{enumerate}[label=(\roman*)]
			\item If~$P_1$ is adjacent to a chain of leaves~$(a,k)$, then one 
			of the end-leaves of~$(a,k)$ must be shortest distance~$5$ away 
			from~$P_1$. Without loss of generality, suppose that~$d^N(P_1,a_1) 
			= 5$. In~$B'$, since~$l(P') = 2$, we must have 
			that~$d^{N'}(P'_1,a_1)\ge 6$, which contradicts the fact that~$N$ 
			and~$N'$ must realize the same shortest distance matrix.
			\item If~$P_1$ is adjacent to another pendant level-$1$ blob~$P_2$, 
			then~$d^N(P_1,P_2) = 7$. In~$\bar{B'}$, both corresponding pendant 
			blobs are of level-$2$. This means that~$d^{N'}(P'_1,P'_2) \ge 8$, 
			which contradicts the fact that~$N$ and~$N'$ must satisfy the same 
			shortest distance matrix. 
			\item If~$P_1$ is adjacent to~$P_2$, then we have~$d^N(P_1,P_2) = 
			8$. Since~$N$ and~$N'$ satisfy the same shortest distance matrix, 
			one must also have~$d^{N'}(P'_1,P'_2) = 8$. The shortest distance 
			from~$P'_1$ to its connection~$p'_1$, and that from~$P'_2$ to its 
			connection~$p'_2$ is~$3$ and~$4$, respectively. Since~$l(B') 
			\in\{1,2\}$, the vertices~$p'_1$ and~$p'_2$ cannot be the same. 
			Therefore to satisfy~$P'_1$ and~$P'_2$ being shortest distance~$8$ 
			from one another in~$\bar{B'}$, one must have that~$p'_1$ 
			and~$p'_2$ are adjacent, which means that~$P'_1$ and~$P'_2$ must be 
			adjacent. The converse follows by symmetry.
			\item Suppose for a contradiction that~$P_1$ is adjacent to two 
			pendant level-$2$ blobs~$P_2$ and~$P_3$ in~$\bar{B}$. 
			Then~$d^N(P_2,P_3) = 10$. By 3., we know that~$P'_1$ must be 
			adjacent to both~$P'_2$ and~$P'_3$ in~$\bar{B'}$. We also know 
			that~$l(P'_2) = l(P'_3) = 1$. It follows that~$d^{N'}(P'_2,P'_3) = 
			8$. But this contradicts the fact that~$N$ and~$N'$ must 
			realize the same shortest distance matrix. If a pendant level-$2$ 
			blob is adjacent to two pendant level-$1$ blobs, then the 
			corresponding pendant level-$1$ blob in~$\bar{B'}$ is adjacent to 
			two pendant level-$2$ blobs, which we have just shown cannot be 
			true.
			\item Suppose~$l$ is an end-leaf of a chain such that~$d^N(P_1,l) 
			=6$. Since~$l(P'_1) = 2$, and since~$N$ and~$N'$ realize the same 
			shortest distance matrix, the leaf~$l$ must be adjacent to~$P'_1$ 
			in~$\bar{B'}$. So every leaf that is shortest distance~$6$ away 
			from~$P_1$ in~$\bar{B}$ must be adjacent to~$P'_1$ in~$\bar{B'}$.  
			Note that~$P'_1$ can be adjacent to at most two chains. These 
			chains must be distinct in~$\bar{B'}$, since~$u'$ is contained 
			in~$\bar{B'}$. This implies that in~$\bar{B}$,~$P_1$ can be 
			shortest distance~$6$ away from at most two end-leaves of distinct 
			chains.
		\end{enumerate}
	\end{proof}
	
	It follows that each main path of~$\bar{B}$ (and~$\bar{B'}$) may contain at 
	most two pendant level-$1$ blobs. Either~$\bar{B}$ or~$\bar{B'}$ must 
	contain a pendant level-$1$ blob. 
	\begin{enumerate}
		\item \textbf{$\bm{\bar{B}}$ has a pendant level-$\bm{1}$ 
		blob~$\bm{P_1}$:} Since~$N$ contains no parallel edges, and since 
		leaves cannot be adjacent to pendant level-$1$ blobs, $P_1$ must be 
		adjacent to a pendant level-$2$ blob~$P_2$ in~$\bar{B}$. By 
		Claim~\ref{cla:4Statements} $(iv)$,~$P_1$ can be adjacent to at most 
		one 
		pendant level-$2$ blob in~$\bar{B}$. This implies that~$p_1$ must be 
		adjacent to~$u$. In~$N'$, the corresponding pendant blobs~$P'_1$ 
		and~$P'_2$ are adjacent by Claim~\ref{cla:4Statements} $(iii)$. Then we 
		have 
		that 
		\[d^N(P_1,u) - d^N(P_2,u) \le 4-5 = -1.\]
		Observe that~$d^{N'}(P'_1,p'_1) = 4$ and~$d^{N'}(P'_2,p'_1) = 4$ 
		since~$l(P'_1) = 2$ and~$l(P'_2) = 1$. This implies that
		\begin{align*}
			d^{N'}(P'_1,u') - d^{N'}(P'_2,u') &\ge d^{N'}(P'_1,p'_1) + 
			d^{N'}(p'_1,u') - d^{N'}(P'_2,p'_1) - d^{N'}(p'_1,u')\\
			&= 4 - 4\\
			&= 0,
		\end{align*}
		which is a contradiction to Claim~\ref{cla:DistToCE}.
		\item \textbf{$\bm{\bar{B'}}$ has a pendant level-$\bm{1}$ 
		blob~$\bm{P'_1}$:} If~$l(B') = 1$, then we are done by 
		symmetry via case 1. So suppose that~$l(B') = 2$, and suppose in 
		addition that~$\bar{B}$ contains no pendant level-$1$ blobs. This 
		implies that~$\bar{B'}$ contains no pendant level-$2$ blobs. We claim 
		that~$\bar{B'}$ also contains no pendant level-$1$ blobs other 
		than~$P'_1$. Suppose for a contradiction that it did, so  
		that~$\bar{B'}$ contains another pendant level-$1$ blob~$P'_2$. 
		Because~$\bar{B'}$ contains no pendant level-$2$ blobs, and since 
		leaves cannot be adjacent to pendant level-$1$ blobs by 
		Claim~\ref{cla:4Statements}, we must have that~$p'_1$ and~$p'_2$ are 
		adjacent to the same pole, or that they must both be adjacent to~$u'$. 
		In any case, we must have~$d^{N'}(P'_1,P'_2) = 8$. But~$l(P_1) = l(P_2) 
		= 2$, and therefore~$d^N(P_1,P_2) \ge 9$, which is a contradiction. 
		So~$P'_1$ is the only pendant blob in~$\bar{B'}$.
		
		We now consider two possible cases: either~$p'_1$ is or is not adjacent 
		to~$u'$. 
		Either way, at least one main path of~$B'$ that does not contain~$p'_1$ 
		must contain a chain of leaves, as otherwise~$B'$ contains parallel 
		edges, or~$B'$ is a level-$2$ blob with only two cut-edges incident to 
		it. 
		
		\begin{enumerate}
			\item \textbf{$\bm{p'_1}$ is adjacent to~$\bm{u'}$:} 
			One of the main paths of~$\bar{B'}$ must contain a chain, 
			since~$N'$ does not contain parallel edges. So there must exist a 
			leaf~$l$ that is shortest distance~$6$ away from~$P'_1$. Then
			\[d^{N'}(P'_1,u')-d^{N'}(l,u') \le 4 - 3 = 1.\]
			On the other hand, in~$\bar{B}$, the leaf~$l$ is adjacent to~$P_1$; 
			then~$d^N(P_1,p_1) = 4$ and~$d^N(l,p_1) = 2$. It follows 
			that
			\begin{align*}
				d^N(P_1,u) - d^N(l,u) &\ge d^N(P_1,p_1) + d^N(p_1,u) - 
				d^N(l,p_1) - d^N(p_1,u)\\
				&= 2,
			\end{align*}
			which contradicts Claim~\ref{cla:DistToCE}.
			\item \textbf{$\bm{p'_1}$ is not adjacent to~$\bm{u'}$:} 
			Observe that~$\bar{B'}$ has five sides: two sides~$s'_1,s'_2$ which 
			have~$p'_1$ as one of their boundary vertices; two 
			sides~$s'_3,s'_4$ which have~$u'$ as one of their boundary 
			vertices; and one side~$s'_5$ that has neither~$p'_1$ nor~$u'$ as a 
			boundary vertex. By Claim~\ref{cla:4Statements} $(i)$ and~$(v)$, 
			the 
			sides~$s'_1$ and~$s'_2$ 
			are empty, and at most two of the three remaining sides 
			of~$\bar{B'}$ may contain chains. In particular, at least one of 
			these remaining three sides must contain a chain. We first show 
			that~$s'_5$ must be empty. Suppose not, and let~$(a,k)$ denote the 
			chain contained in~$s'_5$. Note that both end-leaves of~$(a,k)$ are 
			shortest distance~$6$ away from~$P'_1$, and so by 
			Claim~\ref{cla:4Statements} $(v)$, this chain must be of 
			length~$k=1$. 
			Since either~$s'_3$ or~$s'_4$ must be empty,
			\[d^{N'}(a_1,u') - d^{N'}(P'_1,u') = 3 - 5 = -2.\]
			In~$\bar{B}$, the blob~$P_1$ and the leaf~$a_1$ must be adjacent. 
			The vertex~$u$ must be adjacent to the neighbor of~$a_1$, 
			since~$\bar{B}$ contains no other pendant blobs. So we have
			\[d^N(a_1,u) - d^N(P_1,u) \le 2 - 5 = -3,\]
			which contradicts Claim~\ref{cla:DistToCE}. So~$s'_5$ is empty. 
			
			Now suppose that~$s'_3$ and~$s'_4$ contain the chains~$(b,\ell)$ 
			and~$(c,m)$, respectively, where at least one of~$\ell\ge 1$ 
			or~$m\ge 1$ holds. If~$\ell=1$ and~$m=0$, then~$\bar{B'}$, 
			together with the edge incident to~$u'$ is an alt-path structure of 
			a binary tree on two leaves. This contradicts our choice of~$N'$. 
			By symmetry, the case~$m=1$ and~$\ell=0$ is also not possible. So 
			we may assume that~$\ell\ge1$ and~$m\ge 1$. Suppose the chains are 
			arranged such that
			\[d^{N'}(b_1,u') = d^{N'}(c_1,u') = 2.\]
			If~$\ell>1$ or~$m>1$, then~$d^{N'}(b_\ell,c_m) = 5$, 
			whereas~$d^N(b_\ell,c_m) = 4$. This contradicts the fact that~$N$ 
			and~$N'$ realize the same shortest distance matrix. So we must
			have~$\ell=m=1$. But then~$\bar{B'}$, together with the edge 
			incident to~$u'$ is an alt-path structure of a binary tree on two 
			leaves. This contradicts our choice of~$N'$.

			\end{enumerate}
	\end{enumerate}

	\subsubsection{$\bm{B}$ is a level-$\bm{2}$ blob 
	in~$\bm{N}$:}\label{subsec:L2}
	
	Our only remaining case is if~$B$ and~$B'$ are both level-$2$ blobs. The 
	proofs of Claims \ref{cla:L2L2Bad}-\ref{cla:1L1Blob} are given in the 
	appendix.
	
	\begin{claim}\label{cla:L2L2Bad}
		Two pendant level-$2$ blobs cannot be adjacent to one another 
		in~$\bar{B}$ and in~$\bar{B'}$.
	\end{claim}

	An immediate consequence of Claim~\ref{cla:L2L2Bad} is that distinct 
	pendant level-1 blobs in~$\bar{B}$ or~$\bar{B'}$ must be distance at least 
	10 apart. In particular, they cannot be adjacent by 
	Claim~\ref{cla:4Statements} and they cannot be shortest distance-8 apart, 
	since two pendant level-$2$ blobs are shortest distance at least~$9$ apart. 
	The following claim dictates the placement of pendant blobs and 
	leaves in~$\bar{B}$.
	\begin{claim}\label{cla:L1LeafBad}
		A pendant level-$2$ blob may not be adjacent to both a pendant 
		level-$1$ blob and a leaf simultaneously in~$\bar{B}$ and in~$\bar{B'}$.
	\end{claim}

	Pendant level-$1$ blobs may be adjacent to at most one pendant 
	level-$2$ blob by Claim~\ref{cla:4Statements}. Pendant level-$2$ blobs 
	cannot be adjacent to other pendant level-$2$ blobs by 
	Claim~\ref{cla:L2L2Bad}. Pendant level-$1$ blobs cannot be adjacent to a 
	leaf by Claim~\ref{cla:L1LeafBad}. So the main path of~$\bar{B}$ 
	(and~$\bar{B'}$) that contains~$u$ ($u'$) contains at most two pendant
	level-$1$ blobs; the other two main paths of~$\bar{B}$ (and~$\bar{B'}$) 
	contain at most one pendant level-$1$ blob.
	
	\begin{claim}\label{cla:1L1Blob}
		$\bar{B}$ and~$\bar{B'}$ contain at most one pendant level-$1$ blob.
	\end{claim}
%
%
%
	We have now arrived at the two final cases for this proof. In summary, the 
	current 
	setting is as follows. Both~$\bar{B}$ and~$\bar{B'}$ are level-$2$ blobs, 
	and they both contain at most one pendant level-$1$ blob. In fact, this 
	implies that~$\bar{B}$ and~$\bar{B'}$ also contain at most one pendant 
	level-$2$ blob. We split into the cases for when~$\bar{B}$ does not, or 
	does contain a pendant level-$2$ blob.
	\begin{enumerate}
		\item \textbf{$\bm{\bar{B}}$ contains no pendant level-$\bm{2}$ blob:} 
		By assumption,~$\bar{B}$ must contain a pendant level-$1$ 
		blob~$P_1$. Let~$s$ denote the main path of~$B$ containing~$P_1$. 
		Note that~$s$ may contain at most one chain. Indeed,~$P_1$ cannot be 
		adjacent to a chain of leaves by Claim~\ref{cla:4Statements} $(i)$, 
		so~$p_1$ 
		must be adjacent to a pole of~$\bar{B}$; if~$p_1$ is adjacent to~$u$, 
		then~$s$ can contain a chain of leaves~$(a,k)$ such that an end-spine 
		vertex of~$(a,k)$ is adjacent to~$u$. The other two main paths 
		of~$\bar{B}$ may contain at most one chain of leaves each. So in 
		total,~$\bar{B}$ may contain at most three chains. 
		
		For each chain contained in~$\bar{B}$, we have that one end-leaf of a 
		chain
		is shortest distance-$6$ from~$P_1$. This means that each chain 
		contained in~$\bar{B}$ must be adjacent to~$P'_1$ in~$\bar{B'}$. 
		But~$P'_1$ may be adjacent to at most two chains. So~$\bar{B}$ may 
		contain up to two chains. This also implies 
		that~$\bar{B'}$ only contains leaves on the main path that 
		contains~$p'_1$. Since~$\bar{B'}$ contains no parallel edges, we must 
		then have that~$u'$ lies on a main path that does not contain~$p'_1$.
		
		Note that~$\bar{B}$ must contain at least one chain, as otherwise~$B$ 
		would be a level-$2$ blob incident only to two cut-edges. We now split 
		into subcases depending on the location of~$u$.
		\begin{enumerate}
			\item \textbf{$\bm{u}$ is adjacent to~$\bm{p_1}$:} 
			Then~$d^N(P_1,u) = 4$. Since~$u'$ is not on the same main path as 
			that containing~$P'_1$, we have~$d^{N'}(P'_1,u')\ge 6$. Now there 
			exists a leaf~$l$ such that~$d^{N'}(l,u') = 3$. Noting 
			that~$d^N(l,u) \ge 2$, we have 
			\[d^N(P_1,u) - d^N(l,u) \le 2\]
			and
			\[d^{N'}(P'_1,u') - d^{N'}(l,u') \ge 3,\]
			which contradicts Claim~\ref{cla:DistToCE}.
			\item \textbf{$\bm{u}$ is not adjacent to~$\bm{p_1}$:} 
			We let~$s,s_1,s_2$ denote the three main paths of~$\bar{B}$ such 
			that~$s$ contains~$p_1$,~$s_1$ contains~$u$, and~$s_2$ contains 
			neither~$p_1$ nor~$u$. We claim first that~$s_2$ contains no 
			chains. Suppose for a contradiction that it did contain some 
			chain~$(a,k)$. Note first that~$a_1$ and~$a_k$ are both shortest 
			distance $6$ from~$P_1$; this implies that~$P'_1$ is adjacent to 
			both~$a_1$ and~$a_k$ in~$\bar{B'}$, implying that~$a_1$ and~$a_k$ 
			are in different chains in~$\bar{B'}$. But this is not possible, so 
			we require~$k=1$. Note also that since~$\bar{B}$ contains at most 
			two chains,~$u$ must be adjacent to a pole; this implies 
			that~$d^N(P_1,u) = 5$ and~$d^N(a_1,u) = 3$. However in~$B'$, we 
			have~$d^{N'}(P'_1,u') \ge 6$ and~$d_{N'}(a_1,u') = 3$, which 
			contradicts Claim~\ref{cla:DistToCE} as
			\[d^N(P_1,u)-d^N(a_1,u) = 5-3 = 2\]
			whereas
			\[d^{N'}(P'_1,u')-d^{N'}(a_1,u') \ge 6-3 = 3.\] 
			So the main path~$s_2$ contains no chains; this leaves only~$s_1$ 
			to contain chains. 
			
			The main path~$s_1$ may contain two chains,~$(b,\ell)$ and~$(c,m)$, 
			such that~$\ell,m\ge 0$ and~$d^N(b_1,u) = d^N(c_1,u) = 2$, 
			whenever~$\ell>0$ and~$m>0$, respectively. We 
			require~$\ell+m\ge3$, as otherwise~$\bar{B}$ with the edge incident 
			to~$u$ is an alt-path structure that can be obtained from a binary 
			tree on two leaves. We fall into two subcases depending on the 
			value of~$\ell$.
			
			\begin{enumerate}
				\item \textbf{$\bm{\ell=0}$:} Then~$m\ge3$, and so
				\[d^N(P_1,c_1) = 7,\]
				whereas
				\[d^{N'}(P'_1,c_1) = 8,\]
				which contradicts the fact that~$N$ and~$N'$ realize the same 
				shortest distance matrix.
				\item \textbf{$\bm{\ell\ne0}$:} By symmetry, we may 
				assume~$m\ne0$. Now,
				\[d^N(b_1,c_1) = 4,\]
				whereas
				\[d^{N'}(b_1,c_1) = 5,\]
				since~$\ell+m\ge3$. This contradicts the fact that~$N$ and~$N'$ 
				realize the same shortest distance matrix.
			\end{enumerate}

		\end{enumerate}
		\item \textbf{$\bm{\bar{B}}$ contains one level-$\bm{2}$ blob:}
		If~$\bar{B}$ 
		did not contain a pendant level-$1$ blob, then we are done by applying 
		the arguments from the previous case to~$\bar{B'}$. So suppose 
		that~$\bar{B}$ contains a pendant level-$1$ blob~$P_1$ and a pendant 
		level-$2$ blob~$P_2$. We first show that~$P_1$ and~$P_2$ cannot be 
		adjacent in~$\bar{B}$. Suppose that they were adjacent. Then the vertex 
		on~$\bar{B}$ that is shortest distance-$2$ from~$p_1$ must be a pole 
		of~$\bar{B}$ or~$u$. Indeed, it cannot be a neighbor of some leaf~$l$; 
		this would mean that~$d^{N}(P_1,l) = 6$, implying that~$P'_1$ must 
		be adjacent to~$l$. But this is not possible by 
		Claim~\ref{cla:L1LeafBad}. In particular, it cannot be a connection 
		since~$\bar{B}$ contains only the pendant blobs~$P_1$ and~$P_2$. 
		
		Now, if~$u$ was adjacent to~$p_2$, 
		then this would mean that the two main paths of~$\bar{B}$ would be 
		empty, resulting in parallel edges in~$\bar{B}$. So~$u$ must either be 
		adjacent to~$p_1$ or~$u$ must be contained in one of the two other main 
		sides of~$\bar{B}$. Either way, we have				
		\[d^N(P_1,u) - d^N(P_2,u) \le -1.\]
		In~$\bar{B'}$, we have~$d^{N'}(P'_1,p'_1) = d^{N'}(P'_2,p'_1) = 4$.
		So
		\[d^{N'}(P'_1,u') - d^{N'}(P'_2,u') \ge 0,\]
		which contradicts Claim~\ref{cla:DistToCE}.
		So we may assume that~$P_1$ and~$P_2$ are not adjacent in~$\bar{B}$. We 
		split into cases depending on the position of~$u$ in~$\bar{B}$.
		\begin{enumerate}
			\item \textbf{$\bm{u}$ is on the same main path as~$\bm{p_1}$:} 
			Then~$u$ must be adjacent to~$p_1$, since~$P_1$ cannot be adjacent 
			to a chain by Claim~\ref{cla:4Statements}, and since~$P_1$ is not 
			adjacent to the only pendant level-$2$ blob in~$\bar{B}$. 
			Then~$d^N(P_1,u) = 4$ and~$d^N(P_2,u)\ge 5$. We split into 
			subcases depending on the position of~$u'$ in~$\bar{B'}$.
			\begin{enumerate}
				\item \textbf{$\bm{u'}$ is on the same main path 
				as~$\bm{p'_2}$:} Then~$u'$ must be adjacent to~$p'_2$. So we 
				have~$d^{N'}(P'_2,u')=4$ and~$d^{N'}(P'_1, u')\ge 5$, which 
				contradicts Claim~\ref{cla:DistToCE}.
				\item \textbf{$\bm{u'}$ is not on the same main path 
				as~$\bm{p'_2}$:} If~$u'$ is adjacent to an end-spine vertex of 
				some chain, then there exists a leaf~$l$ such 
				that~$d^{N'}(l,u') = 2$. We also have~$d^{N'}(P'_1,u') 
				\ge 5$. But in~$\bar{B}$, we have~$d^N(l,u) \ge 2$. This 
				contradicts Claim~\ref{cla:DistToCE}, as
				\[d^N(P_1,u) - d^N(l,u) \le 4-2 = 2,\]
				whereas
				\[d^{N'}(P'_1,u') - d^{N'}(l, u') \ge 5-2 = 3.\]
				So~$u'$ cannot be adjacent to an end-spine vertex of some 
				chain, which means that~$d^{N'}(P'_2,u') = 5$. But~$d^N(P_2,u) 
				\ge 5$, and so
				\[d^N(P_1,u) - d^N(P_2,u) \le 4-5 = -1,\]
				whereas
				\[d^{N'}(P'_1,u') - d^{N'}(P'_2, u')\ge 0,\]
				which contradicts Claim~\ref{cla:DistToCE}.
			\end{enumerate}
			\item \textbf{$\bm{u}$ is not on the same main path as~$\bm{p_1}$:} 
			We may assume that~$u'$ is not on the same main path as~$p'_2$ by 
			symmetry (apply the previous case to~$\bar{B'}$).
			\begin{enumerate}
				\item \textbf{$\bm{u}$ and~$\bm{p_2}$ are not on the same main 
				side:} 
				We let~$s_u,s_1,s_2$ denote the three main paths of~$\bar{B}$ 
				such that~$s_u$ contains~$u$ and~$s_i$ contains~$p_i$ 
				for~$i=1,2$. 
				
				Note that~$s_u$ may contain up to two chains: 
				denote these chains as~$(a,k)$ and~$(b,\ell)$ where~$d^N(a_1,u) 
				= d^N(b_1,u) = 2$, if~$k> 0$ and~$\ell>0$, respectively. 
				If~$k>0$ and~$\ell>0$, then~$d^N(P_2,a)\ge 7$ 
				and~$d^N(P_2,b)\ge 7$.
				Since~$d^N(P_1,a_k) = d^N(P_1,b_\ell) = 6$, the pendant 
				blob~$P'_1$ must be adjacent to the two chains~$a$ and~$b$ 
				in~$\bar{B'}$. Depending on the placement of~$u'$, at least 
				one, and at most two of the chain endpoints~$a_1,b_1$ are 
				shortest distance~$6$ away from~$P'_2$. But this contradicts 
				that~$N$ and~$N'$ realize the same 
				shortest distance matrix. Therefore,~$s_u$ contains at most one 
				chain; this means that~$d^N(P_1,u) = 5$. We also 
				have~$d^N(P_2,u) \ge 6$. 
				
				The same argument can be used in the case when~$u'$ and~$p'_1$ 
				are not on the same main path of~$\bar{B'}$. In that case, the 
				main path of~$\bar{B'}$ containing~$u'$ contains at most one 
				chain. We now split into subcases depending on the position 
				of~$u'$.
				\begin{enumerate}
					\item \textbf{$\bm{u'}$ and~$\bm{p'_1}$ are not on the same 
					main path of~$\bm{\bar{B'}}$:} Then, the main path 
					of~$\bar{B'}$ that contains~$u'$ contains at most one 
					chain. In particular, this means that~$u'$ must be adjacent 
					to a pole in~$\bar{B'}$. So~$d^{N'}(P'_2,u') = 5$. 
					Furthermore,~$d^{N'}(P'_1,u') 
					\ge 6$. But this contradicts Claim~\ref{cla:DistToCE}.
					\item \textbf{$\bm{u'}$ and~$\bm{p'_1}$ are on the same 
					main path of~$\bm{\bar{B'}}$:} Consider the generator 
					side~$s'$ of~$B'$ that contains $u'$ and a pole 
					of~$\bar{B'}$ as its boundary vertices. We claim that~$s'$ 
					is empty. If not, then~$s'$ contains a chain~$(a,k)$ such 
					that~$d^{N'}(a_1,u') = 2$. But then~$d^{N'}(P'_2, a_k) = 
					6$, meaning that~$a_k$ must be adjacent to~$P_2$ 
					in~$B$. This further implies that~$d^N(P_1,a_1) = 6$, 
					which leads to a contradiction as~$a_1$ is clearly not 
					adjacent to~$P'_1$ in~$B'$ (we have~$d^{N'}(P'_1,a_1)\ge 
					7$). Thus~$s'$ must be empty. But then
					\[d^N(P_1,u) - d^N(P_2,u) \le 5-6 = -1,\]
					whereas
					\[d^{N'}(P'_1,u') - d^{N'}(P'_2,u') \ge 5-5 = 0,\]
					which contradicts Claim~\ref{cla:DistToCE}.
				\end{enumerate}
				\item \textbf{$\bm{u}$ and~$\bm{p_2}$ are on the same main path 
				of~$\bar{B}$:} We may assume also that~$u'$ and~$p'_1$ are on 
				the same main path of~$\bar{B'}$. Consider the main path~$s$ 
				of~$\bar{B}$ that does not contain~$p_1$ nor~$p_2$. We claim 
				that~$s$ is empty. Suppose not, and suppose that~$s$ contains a 
				chain~$(a,k)$. In~$\bar{B'}$, we require~$P'_1$ to be adjacent 
				to both~$a_1$ and to~$a_k$. For this to be possible, 
				since~$(a,k)$ must also be a chain in~$\bar{B'}$, we 
				require~$k=1$. 
				Note that in~$\bar{B'}$, the leaf~$a_1$ is contained in the 
				side with boundary 
				vertices~$p'_1$ and~$u'$. If it had been contained in the 
				other side of~$\bar{B'}$ with~$p'_1$ as its boundary vertex, 
				then~$P_2$ must be adjacent to~$a_1$ in~$B$, which is 
				clearly not the case. Observe that a shortest path from~$P_1$ 
				to~$u$ contains the same pole contained in a shortest path 
				from~$a_1$ to~$u$. Noting that the shortest distance from~$P_1$ 
				to this pole, and the shortest distance from~$a_1$ to this pole 
				are~$4$ and~$2$, respectively, it follows that
				\[d^N(P_1,u) - d^N(a_1,u) = 4 - 2 = 2,\]
				whereas
				\[d^{N'}(P'_1,u') - d^{N'}(a_1,u') = 6-2 = 4,\]
				which contradicts Claim~\ref{cla:DistToCE}.
				So~$s$ is empty, and we may assume that the main path 
				of~$\bar{B'}$ that does not have~$p'_1$ nor~$p'_2$ is empty.
				
				If all three other sides of~$B$ are empty, then~$B$ 
				contains an alt-path structure formed by a binary tree on 
				two leaves, and we get a contradiction on the choice 
				of~$N$. In particular, the three sides must contain at 
				least two leaves. Let~$s_1, s_2, s_3$ denote the sides 
				of~$B$ that has~$p_2$ but not~$u$,~$p_2$ and~$u$, and~$u$ 
				but not~$p_2$ as its boundary vertices, respectively. 
				Similarly let~$s'_1,s'_2, s'_3$ denote the sides of~$B'$ 
				that has~$p'_1$ but not~$u'$,~$p'_1$ and~$u'$, and~$u$ but 
				not~$p'_1$ as its boundary vertices, respectively. It is 
				easy to see that a chain contained in~$s_1$ must be 
				contained in~$s'_1$; a chain contained in~$s_2$ must be 
				contained in~$s'_3$; a chain contained in~$s_3$ must be 
				contained in~$s'_2$.
				\begin{enumerate}
					\item \textbf{the side~$\bm{s_1}$ is non-empty:} 
					Let~$(a,k)$ denote the chain contained in~$s_1$, such 
					that~$d^N(a_1,P_2) = d^N(a_k,P_1) = 6$. We 
					claim that~$s_2$ and~$s_3$ must both be empty. If~$s_2$ 
					is non-empty, then it contains a chain~$(b,\ell)$, such 
					that~$d^N(b_1,P_2) = 6$. So the chains~$(a,k)$ 
					and~$(b,\ell)$ are adjacent. In~$\bar{B'}$, the 
					chain~$(b,\ell)$ is contained in the side~$s'_3$. But 
					then~$(a,k)$ and~$(b,\ell)$ cannot be adjacent in~$N'$, 
					which contradicts the fact that~$N$ and~$N'$ satisfy 
					the same shortest distance matrix. So~$s_2$ must be 
					empty. By applying the same argument to~$B'$, we see 
					that~$s'_2$ must also be empty; therefore, the 
					side~$s_3$ must be empty. So~$s_2$ and~$s_3$ must both be 
					empty.
					
					We require~$k\ge2$, as otherwise~$B$ contains an alt-path 
					structure obtained from a tree on two leaves. But then
					\[d^N(a_1,u) - d^N(a_k,u) = 3 - 4 = -1,\]
					whereas
					\[d^{N'}(a_1,u') - d^{N'}(a_k,u') = 4 - 3 = 1,\]
					since~$a_1$ is adjacent to~$P_2$ in~$\bar{B}$ and~$a_k$ is 
					adjacent to~$P'_1$ in~$\bar{B'}$. This contradicts 
					Claim~\ref{cla:DistToCE}.
					\item \textbf{the side~$\bm{s_1}$ is empty:} 
					Let~$(b,\ell)$ and~$(c,m)$ denote the chains contained 
					in~$s_2$ and~$s_3$, respectively, such 
					that~$d^N(P_2,b_1) = 6$ and~$d^N(b_\ell,c_m) = 4$, 
					whenever~$\ell>0$ and~$m>0$, respectively. 
					Note that~$\ell+ m \ge 2$, as otherwise~$\bar{B}$ together 
					with the edge incident to~$u$ is an alt-path structure 
					formed by a binary tree on two leaves. In particular, at 
					least one of~$\ell$ or~$m$ must be non-zero. By symmetry, 
					we may assume without loss of generality that~$m>0$. Then 
					in~$\bar{B'}$, the blob~$P'_1$ is adjacent to~$c_1$. 
					This means that
					\[d^{N'}(P'_2,c_1) = 7.\]
					In~$\bar{B}$, there are three paths from~$c_1$ to~$P_2$. 
					One 
					uses the empty main path and is of length~$8$. The second 
					uses 
					the main path with~$p_1$ and is of 
					length~$9$. These two paths cannot be altered by 
					deleting leaves. The third path contains the spine of 
					the chain~$(b,\ell)$, and is of length~$\ell+m+6$. 
					Since~$m\ge1$, we only obtain~$d^N(P_2,c_1) = 7$ if and 
					only if~$m=1$ and~$\ell=0$. But this is not possible, 
					since we require~$\ell+m \ge 2$. So the networks~$N$ 
					and~$N'$ cannot satisfy the same shortest distance matrix, 
					which is a contradiction.
				\end{enumerate}
			\end{enumerate}

		\end{enumerate}
	\end{enumerate}
	Therefore we reach a contradiction for the case when~$B$ and~$B'$ are 
	both level-$2$ blobs.

\end{proof}

%

The following corollary follows immediately from Corollary~\ref{lem:CEShortest} and 
Theorem~\ref{thm:NoAltPath}.

\begin{corollary}
	A level-$2$ network is reconstructible if and only if it does not contain 
	an alt-path structure.
\end{corollary}

\section{Discussion}\label{sec:Discussion}

The results of this paper build on, and answer three open problems 
presented in the paper by van Iersel et al.~\citep{van2020reconstructibility}. 
We have shown that networks with a leaf on each generator side are 
reconstructible from their shortest distance matrix 
(Theorem~\ref{thm:GenSide}).
We have shown that level-$2$ networks are reconstructible from their 
sl-distance matrix (Theorem~\ref{thm:L2SL}). We have 
characterized the family of subgraphs that prevent level-$2$ networks from 
being 
reconstructible from their shortest distances (Theorem~\ref{thm:NoAltPath}).

Previously, it was only known that level-$2$ networks were reconstructible from 
their multisets of distances, the full collection of lengths of all inter-taxa 
distances together with their multiplicities. An algorithm based on this result 
was recently presented and implemented in the Bachelor Thesis of Riche 
Mol~\citep{Mol2020}, where 
a major bottleneck originated from having to adjust the large 
multisets of distances upon identifying and reducing a particular pendant 
structure. As a result, the theoretical running time of the algorithm \rev{is}
exponential in the number of leaves in the network (though it is polynomial in 
the size of the input, the multisets of distances). The results presented in 
this paper point to a possibility of an alternative algorithm for constructing 
level-$2$ 
networks from their sl-distance matrix; since updating the sl-distance matrix 
can be done much quicker than for multisets of distances, we wonder if this 
could culminate in a polynomial time algorithm with respect to the number of 
leaves in the network. It would be of great interest to see the speed-up both 
theoretically and in practice.

In this paper, we have excluded all blobs incident to exactly two cut-edges. 
One of the consequences of excluding such blobs is that we never obtain pendant 
level-$2$ blobs of the form~$(1,0,0,0)$ in our networks. Conditions for 
identifying and reducing such pendant blobs from level-$2$ networks are 
outlined in Lemmas~5.9 and~5.10 of~\citep{van2020reconstructibility}. In fact, 
such pendant blobs can be inferred from only the shortest distance 
matrix. This means that Theorem~\ref{thm:L2SL}, which says that level-$2$ 
networks are reconstructible from their sl-distance matrix, holds in general 
when this restriction is not imposed. On the other hand, allowing for such 
blobs introduces a new level of complexity within alt-path 
structures. Call a level-$2$ blob with two cut-edges a \emph{macaron}, and 
consider an alt-path structure~$G$ obtained from some tree~$T$, and 
replace every cut-edge in~$G$ by a path of arbitrary many macarons. Call this 
graph~$G'$. Let~$H$ denote a similar alt-path structure to~$G$, and let us
replace the same cut-edges by paths consisting of the same number of macarons  
(where by the same cut-edge, 
we mean the cut-edge that induces the same split). Call this resulting 
graph~$H'$. It is easy to see that~$G'$ and~$H'$ realize the same shortest 
distance matrix. The converse is not immediately obvious. In other words, it is 
not clear 
whether excluding these `macaron-added' alt-path structures from level-$2$ 
networks guarantee reconstructibility from their shortest distances. 
Nevertheless, we make the following conjecture.

\begin{conjecture}
	A level-$2$ network is reconstructible from its shortest distance matrix if 
	and only if after suppressing macarons and degree-$2$ vertices, it does not 
	contain an alt-path structure.
\end{conjecture}


A potential shortcoming of our findings lies in the fact that the networks we 
consider are unweighted. In phylogenetic analysis, weighted 
edges are often used to indicate the extent on how two species may differ from 
one another - to depict the passage of time, or to indicate the 
amount of genetic divergence between two species. The major issue that 
arises from weighted edges is that the 
foundational structures such as cherries and chains can no longer be 
characterized by their distances. In the rooted weighted variant of the 
problem, this is overcome by simulating a `relative root' by imposing 
ultrametric conditions and through the use of 
outgroups~\citep{bordewich2018constructing}. This 
makes it possible 
to locate cherries and the rooted analogue of chains (\emph{reticulated 
cherries}), even when the network is weighted. While these techniques do not 
translate over to the unrooted setting, some additional 
conditions will almost certainly be required to obtain results for the weighted 
variant of the problem.

\bibliographystyle{apa}

\begin{appendices}
	\section{Proof of Claims from Section~\ref{subsec:L2}:}
	
	To prove these claims, we will use the following observation.
	
	\begin{observation}\label{obs:689}
		Let~$P_1$ be a pendant level-$1$ blob contained in~$\bar{B}$. Suppose 
		that~$p'_2p'_1p'_3p'_4$ is a path in~$\bar{B'}$, such that each~$p'_i$ 
		is a connection for a pendant blob~$P'_i$ for~$i\in [4]$. Suppose also 
		that~$l(P'_1) = l(P'_3) = 1$ and~$l(P'_2) = l(P'_4) = 2$. A vertex~$p$ 
		on the blob~$\bar{B}$ such that~$d^N(p_1,p) = 2$ must either be~$u$ or 
		a pole of~$\bar{B}$.
	\end{observation}
	\begin{proof}
		Suppose for a contradiction that~$p$ is either a neighbor of a leaf or 
		a connection of some pendant blob. Suppose first that~$p$ is a 
		neighbor of a leaf~$l$. Then~$d^N(P_1,l) = 6$. Since~$N$ and~$N'$ 
		have the same shortest distance matrix, this means that~$P'_1$ must 
		be adjacent to~$l$ in~$\bar{B'}$. But this is not possible as~$P'_1$ is 
		already adjacent to~$P'_2$ and~$P'_3$. So suppose that~$p$ is a 
		connection of a pendant blob~$P_5$. Suppose first that~$l(P_5) = 1$. 
		Then~$d^N(P_1,P_5) = 8$. Since~$l(P'_1) = l(P'_5) = 2$, such a distance 
		can be realized in~$N'$ if and only if~$p'_1=p'_5$. But this is 
		impossible as~$B'$ is a level-$2$ blob. Finally suppose that~$l(P_5) = 
		2$. Then since~$d^N(P_1,P_5) = 9$, the corresponding blob~$P'_5$ must 
		be adjacent to~$P'_2$ or to~$P'_4$, which is not possible as pendant 
		level-$1$ blobs cannot be adjacent to one another by 
		Claim~\ref{cla:4Statements}.
	\end{proof}

	\setcounter{claim}{2}

	\begin{claim}
		Two pendant level-$2$ blobs cannot be adjacent to one another 
		in~$\bar{B}$ (and in~$\bar{B'}$).
	\end{claim}
	\begin{proof}
		Suppose for a contradiction that~$\bar{B}$ contains two adjacent 
		pendant level-$2$ blobs~$P_1$ and~$P_2$ on the main path~$s$. Then~$B'$ 
		contains two corresponding pendant level-$1$ blobs~$P'_1$ and~$P'_2$ on 
		the same leaves. We split into cases depending on the locations 
		of~$p'_1$ and~$p'_2$ in~$B'$.
		\begin{enumerate}
			\item \textbf{$\bm{p'_1}$ and~$\bm{p'_2}$ are on the 
				same main path~$\bm{s'}$ of~$\bm{B'}$:} Since the shortest 
				distance 
			between~$p'_1$ and~$p'_2$ must be exactly~$3$, there must be at 
			least two vertices on the same main path in the path~$Q'$ 
			between~$p'_1$ and~$p'_2$. This path~$Q'$ may contain~$u'$ as a 
			vertex; we split into cases again.
			\begin{enumerate}
				\item \textbf{$\bm{u'}$ is a vertex of~$\bm{Q'}$:} Observe 
				first that if both~$P'_1$ and~$P'_2$ are adjacent to level-$2$ 
				blobs~$P'_3$ and~$P'_4$ respectively, then~$d^{N'}(P'_3,P'_4) 
				\ge10$. However, the counterparts of these blobs in~$\bar{B}$ 
				must be adjacent to~$P_1$ and~$P_2$. Since~$P_1$ and~$P_2$ are 
				adjacent, this implies that~$d^N(P_3,P_4) = 9$, which 
				contradicts the fact that~$N$ and~$N'$ satisfy the same 
				shortest distance matrix. Therefore only one of~$P'_1$ 
				or~$P'_2$ can be adjacent to a pendant level-$2$ blob. Note 
				that at least one of~$P'_1$ or~$P'_2$ must be adjacent to a 
				pendant level-$2$ blob, such that its connection is on the 
				path~$Q'$. So without loss of generality, suppose that~$P'_1$ 
				is adjacent to a pendant level-$2$ blob~$P'_3$, such 
				that~$p'_3$ is a vertex in~$Q'$. At this point, we have 
				that~$P_3,P_1$ and~$P_1,P_2$ are adjacent in~$\bar{B}$.
				
				We claim that~$P'_3$ cannot be adjacent to a leaf or to pendant 
				blobs other than~$P'_1$ in~$\bar{B'}$. Firstly, if~$P'_3$ was 
				adjacent to a leaf~$l$, then~$d^{N'}(P'_1,l) \in \{5,6\}$. The 
				distance~$d^{N'}(P'_1,l) = 5$ is impossible as~$l(B) = 2$; 
				if~$d^{N'}(P'_1,l) = 6$, then~$P_1$ must be adjacent to~$l$ 
				in~$\bar{B}$. But this is impossible as~$P_1$ is already 
				adjacent to~$P_2$ and~$P_3$. This is a contradiction. The 
				blob~$P'_3$ cannot be adjacent to a pendant 
				level-$1$ blob other than~$P'_1$, as this would contradict 
				Claim~\ref{cla:4Statements}. Finally, we claim that~$P'_3$ 
				cannot be adjacent to a pendant level-$2$ blob~$P'_4$. Since 
				this would mean that~$d^{N'}(P'_1,P'_4) = 9$, we must 
				in~$\bar{B}$ that either~$P_2$ or~$P_3$ is adjacent to~$P_4$. 
				The former is not possible as~$P'_2$ is not adjacent to~$P'_4$ 
				in~$B'$; the latter is not possible as two pendant level-$1$ 
				blobs cannot be adjacent by Claim~\ref{cla:4Statements}. 
				So~$P'_3$ cannot be adjacent to a leaf or to pendant blobs 
				in~$\bar{B'}$.
				
				By Claim~\ref{cla:4Statements}, the connection~$p'_1$ must be 
				adjacent to a pole of~$B'$. Since~$P'_2$ cannot be adjacent to 
				a leaf or to a pendant level-$1$ blob by 
				Claim~\ref{cla:4Statements}, and it also cannot be adjacent to 
				any pendant level-$2$ blobs by assumption,~$p'_2$ must also be 
				adjacent to the other pole of~$B'$ and to~$u'$. Since~$p'_3$ 
				must be adjacent to~$u'$, it follows that the main path~$s'$ is 
				the path~$p'_1p'_3u'p'_2$.
				
				Now~$B'$ must contain another leaf on one of the other two main 
				sides since they cannot contain parallel edges. We claim that 
				such a leaf cannot exist, thereby reaching a contradiction. 
				Let~$l$ denote such a leaf that is on one of these two main 
				sides, whose neighbor (if~$l$ is not in a pendant blob) / 
				connection (if~$l$ is in a pendant blob) is shortest 
				distance-$2$ to~$p'_1$. Suppose first that~$l$ is not in a 
				pendant blob. Then
				\[d^{N'}(P'_1,l) = 6,\]
				and so~$P_1$ must be adjacent to~$l$ in~$\bar{B}$, 
				which is not possible as~$P_1$ is already adjacent to~$P_2$ and 
				to~$P_3$. So now suppose that~$l$ is contained in a pendant 
				level-$1$ blob. Then
				\[d^{N'}(P'_1,l) = 8.\]
				But the level of the corresponding pendant blob in~$\bar{B}$ 
				is~$2$, and since~$l(P_1) = 2$, we must have that~$d^N(P'_1,l) 
				\ge 9$, which contradicts the fact that~$N$ and~$N'$ have 
				the same shortest distance matrix. Finally, if~$l$ is contained 
				in a pendant level-$2$ blob~$P'_4$, then
				\[d^{N'}(P'_1,l) = 9.\]
				This means that in~$\bar{B}$, the corresponding blob~$P_4$ must 
				be adjacent either to~$P_2$ or~$P_3$. The former is not 
				possible as~$P'_2$ is not adjacent to~$P'_4$ in~$\bar{B'}$; the 
				latter is not possible as two pendant level-$1$ blobs cannot be 
				adjacent by Claim~\ref{cla:4Statements}.
				
				\item \textbf{$\bm{u'}$ is not a vertex of~$\bm{Q'}$:} 
				Let~$p'_3$ and~$p'_4$ denote the neighbors of~$p'_1$ 
				and~$p'_2$ in this path~$Q'$, respectively. 
				Since~$l(P'_1)=l(P'_2) = 1$, the vertices~$p'_3$ and~$p'_4$ 
				must be connections of pendant level-$2$ blobs~$P'_3$ 
				and~$P'_4$, respectively. At this point, we have 
				that~$P_3,P_1$; $P_1,P_2$; $P_2,P_4$ are adjacent in~$B$. Now 
				suppose that there is another vertex~$p'_5$ that is a neighbor 
				of~$p'_3$ that is not~$p'_4$ in~$B'$. By 
				Observation~\ref{obs:689}, 
				the vertex~$p'_5$ cannot be a neighbor of a leaf nor a 
				connection of some pendant blob. By nature of~$B'$,~$p'_5$ must 
				be the vertex~$u'$, but this would contradict our assumption 
				that the path~$Q'$ does not include~$u'$. Therefore~$p'_3$ must 
				be adjacent to~$p'_4$. Thus,~$P'_1, P'_3$;~$P'_3,P'_4$; 
				and~$P'_4,P'_2$ are adjacent in~$B'$. 
				
				By 
				Claim~\ref{cla:4Statements}, since pendant level-$1$ blobs may 
				not be adjacent to leaves and they may be adjacent to at most 
				one pendant level-$2$ blob, either~$p'_1$ or~$p'_2$ must be 
				adjacent to a pole of~$B'$. Suppose without loss 
				of generality that~$p'_1$ is adjacent to a pole of~$B'$. 
				Consider the two main paths of~$B'$ that are not~$s'$. 
				Since~$N$ contains no parallel edges, one of the two main paths 
				must contain a vertex. In particular, there must exist a vertex 
				that is distance-$2$ away from~$p'_1$. By 
				Observation~\ref{obs:689}, 
				such a vertex cannot be a neighbor of a leaf nor a connection 
				of a pendant blob in~$\bar{B'}$. Then such a vertex must 
				be~$u'$. By invoking Observation~\ref{obs:689} again, for 
				both~$p'_1$ 
				and~$p'_2$, it is easy to see that these two main paths cannot 
				contain any leaves in~$\bar{B'}$. So~$\bar{B}$ and~$\bar{B'}$ 
				must contain only the eight leaves of these four pendant 
				blobs on the same main paths, with the vertices~$u$ and~$u'$ on 
				a different main path, respectively. But we find that
				\[d^N(P_1,u) - d^{N'}(P'_1,u') = 7-5 = 2,\]
				whereas
				\[d^N(P_3,u) - d^{N'}(P'_3,u') = 5-7 = -2,\]
				which contradicts Claim~\ref{cla:DistToCE}.
			\end{enumerate}

			\item \textbf{$\bm{p'_1}$ and~$\bm{p'_2}$ are incident to 
				different main paths of~$\bm{B'}$:} Let~$s'_1,s'_2$ denote the 
				main 
			sides of~$\bar{B'}$ that contains~$p'_1,p'_2$, respectively. 
			Let~$s'_3$ denote the third main path of~$\bar{B'}$. The 
			vertex~$u'$ is contained either in~$s'_1,s'_2,$ or in~$s'_3$. The 
			first two cases are equivalent by symmetry, so we split into two 
			cases.
			\begin{enumerate}
				\item \textbf{$\bm{u'}$ is in $\bm{s'_1}$:} If~$P'_1$ 
				and~$P'_2$ are both not adjacent to a pendant level-$2$ blob,  
				then~$d^{N'}(P'_1,P'_2) = 8$ and we reach a contradiction as we 
				have~$d^N(P_1,P_2) = 9$. Therefore~$P'_1$ or~$P'_2$ must be 
				adjacent to a pendant level-$2$ blob on this distance-8 
				path. If~$P'_1$ and~$P'_2$ are both adjacent to level-$2$ 
				blob~$P'_3$ and~$P'_4$, respectively, then~$d^{N'}(P'_3,P'_4) 
				\ge 10$. But since~$P_3,P_1$; $P_1,P_2$; and~$P_2,P_4$ would be 
				adjacent in~$\bar{B}$, we must have~$d^N(P_3,P_4) = 9$, which 
				contradicts the fact that~$N$ and~$N'$ satisfy the same 
				shortest distance matrix. Therefore, exactly one of~$P'_1$ 
				or~$P'_2$ must be adjacent to a pendant level-$2$ blob.
				\begin{enumerate}
					\item \textbf{$\bm{P'_1}$ is adjacent to a pendant 
						level-$\bm{2}$ blob~$\bm{P'_3}$:} Note that~$p'_1$ is 
					adjacent to~$u'$ and to~$p'_3$. In particular,~$P'_1$ must 
					be adjacent to~$u'$ to make sure that the shortest path 
					between~$P'_1$ and~$P'_2$ is of length~$9$. Then we have
					\[d^{N'}(P'_1,u') - d^{N'}(P'_3,u') = 4-6= -2.\]
					But in~$\bar{B}$, we have that~$d^N(P_1,p_1) = d^N(P_3,p_1) 
					= 4$. This implies that
					\begin{align*}
					d^N(P_1,u) - d^N(P_3,u) &\ge d^N(P_1,p_1) + d^N(p_1,u) 
					- d^N(P_3,p_1) - d^N(p_1,u)\\
					&= 0,
					\end{align*}
					which contradicts Claim~\ref{cla:DistToCE}.
					\item \textbf{$\bm{P'_2}$ is adjacent to a pendant 
						level-$\bm{2}$ blob~$\bm{P'_4}$:}
					Observe that~$p'_4$ must be placed on~$s'_2$ such 
					that~$d^{N'}(p'_1,p'_4) = 2$. Then we have that
					\[d^{N'}(P'_4,u') - d(P'_1,u') = 7 - 4 = 3.\]
					In~$\bar{B}$,~$P_2$ must be adjacent to both~$P_1$ 
					and~$P_4$ Then~$d^N(P_1,p_1) = 4$, whereas~$d^N(P_4,p_1) = 
					5$. We have that
					\[d^N(P_4,u) - d^N(P_1,u) \le 1,\]
					which contradicts Claim~\ref{cla:DistToCE}.
				\end{enumerate}
				\item \textbf{$\bm{u'}$ is in $\bm{s'_3}$:} Then to ensure that 
				the leaves of~$P'_1$ and the leaves of~$P'_2$ are shortest 
				distance-9 apart, we require~$P'_1$ and~$P'_2$ to be adjacent 
				to level-$2$ blobs~$P'_3$ and~$P'_4$, respectively. But then
				\[d^{N'}(P'_3,P'_4) \ge 10.\]
				In~$\bar{B}$, the pendant blobs~$P_3,P_1$; $P_1,P_2$; 
				and~$P_2,P_4$ are adjacent. So we have
				\[d^N(P_3,P_4) = 9,\]
				which contradicts Claim~\ref{cla:DistToCE}.
			\end{enumerate}
		\end{enumerate}
		This covers all cases for whenever two level-$2$ blobs are 
		adjacent. In all cases, we were able to find a contradiction with 
		regards to the inter-taxa distances or to Claim~\ref{cla:DistToCE}.
	\end{proof}
	
		\begin{figure}
		\centering
				\begin{subfigure}[b]{.49\textwidth}
			\centering
			\begin{tikzpicture}
			[every node/.style={draw, circle, fill, inner sep=0pt, minimum 
				size=2mm}]
			\tikzset{edge/.style={very thick}}
			
			\node[draw=none, fill=none] (v) at (0,1) {};
			\draw[very thick, black] (0,0) circle (0.6) {};
			\node[] (north) at (0,0.6) {};
			\node[] (west) at (180:0.6) {}
			node[draw=none, fill=none, left=8mm] {\large $B'$};
			\node[] (east) at (0:0.6) {};
			\node[] (u1) at (30:0.6) {};
			\node[] (u3) at (60:0.6) {};
			\node[] (u2) at (288:0.6) {};
			\node[] (u4) at (135:0.6) {};
			
			\node[fill=none] (P1) at (1, 1) {}
			node[draw=none, fill=none, above=1mm of P1] {\large $P'_1$};
			\node[] (P2) at (({0.6*cos(288)}, -1) {}
			node[draw=none, fill=none, below=1mm of P2] {\large $l$};
			\node[] (P3) at (0.33, 1) {}
			node[draw=none, fill=none, above=1mm of P3] {\large $P'_3$};
			\node[fill=none] (P4) at (-1, 1) {}
			node[draw=none, fill=none, above=1mm of P4] {\large $P'_2$};
			
			\draw[edge]
			(v) -- (north)
			(west) -- (east)
			(u3) -- (P3)
			(u4) -- (P4);
			
			\draw[edge, bend right]
			(u1) edge (P1);
			
			\draw[edge]
			(u2) edge (P2);
			
			\begin{scope}[xshift=4cm,yshift=0cm]
			\draw[very thick, black] (0,0) circle (0.6) {}
			node[draw=none, fill=none, above left = 4mm and 4mm] {\large 
				$B$};
			\node[] (west) at (180:0.6) {};
			\node[] (east) at (0:0.6) {};
			\node[] (u1) at (270:0.6) {};
			\node[] (u3) at (225:0.6) {};
			\node[] (u2) at (315:0.6) {};
			
			\node[] (P1) at (0, -1) {}
			node[draw=none, fill=none, below=1mm of P1] {\large $P_1$};
			\node[] (P2) at (1, -1) {}
			node[draw=none, fill=none, below=1mm of P2] {\large $P_2$};
			\node[fill=none] (P3) at (-1, -1) {}
			node[draw=none, fill=none, below=1mm of P3] {\large $P_3$};
			
			\draw[edge]
			(west) -- (east)
			(u1) -- (P1);
			
			\draw[edge, bend right]
			(u3) edge (P3);
			
			\draw[edge, bend left]
			(u2) edge (P2);
			\end{scope}
			
			\end{tikzpicture}
			\caption{Case 1. (a)}
		\end{subfigure}\hfill
		\begin{subfigure}[b]{.49\textwidth}
			\centering
			\begin{tikzpicture}
			[every node/.style={draw, circle, fill, inner sep=0pt, minimum 
				size=2mm}]
			\tikzset{edge/.style={very thick}}
			
			\node[draw=none, fill=none] (v) at (0,1) {};
			\draw[very thick, black] (0,0) circle (0.6) {}
			node[draw=none, fill=none, above left = 4mm and 4mm] {\large 
				$B'$};
			\node[] (north) at (90:0.6) {};
			\node[] (west) at (180:0.6) {};
			\node[] (east) at (0:0.6) {};
			\node[] (u1) at (216:0.6) {};
			\node[] (u3) at (252:0.6) {};
			\node[] (u4) at (288:0.6) {};
			\node[] (u2) at (324:0.6) {};
			
			\node[fill=none] (P1) at (-1, -1) {}
			node[draw=none, fill=none, below=1mm of P1] {\large $P'_1$};
			\node[fill=none] (P2) at (1, -1) {}
			node[draw=none, fill=none, below=1mm of P2] {\large $P'_2$};
			\node[] (P3) at (-0.33, -1) {}
			node[draw=none, fill=none, below=1mm of P3] {\large $P'_3$};
			\node[] (P4) at (0.33, -1) {}
			node[draw=none, fill=none, below=1mm of P4] {\large $P'_4$};
			
			\draw[edge]
			(v) -- (north)
			(west) -- (east)
			(u3) -- (P3)
			(u4) -- (P4);
			
			\draw[edge, bend right]
			(u1) edge (P1);
			
			\draw[edge, bend left]
			(u2) edge (P2);
			
			\begin{scope}[xshift=4cm,yshift=0cm]
			
			\node[draw=none, fill=none] (v) at (0,1) {};
			\draw[very thick, black] (0,0) circle (0.6) {}
			node[draw=none, fill=none, above left = 4mm and 4mm] {\large 
				$B$};
			\node[] (north) at (0,0.6) {};
			\node[] (west) at (180:0.6) {};
			\node[] (east) at (0:0.6) {};
			\node[] (u1) at (216:0.6) {};
			\node[] (u3) at (252:0.6) {};
			\node[] (u4) at (288:0.6) {};
			\node[] (u2) at (324:0.6) {};
			
			\node[fill=none] (P1) at (-1, -1) {}
			node[draw=none, fill=none, below=1mm of P1] {\large $P_3$};
			\node[fill=none] (P2) at (1, -1) {}
			node[draw=none, fill=none, below=1mm of P2] {\large $P_4$};
			\node[] (P3) at (-0.33, -1) {}
			node[draw=none, fill=none, below=1mm of P3] {\large $P_1$};
			\node[] (P4) at (0.33, -1) {}
			node[draw=none, fill=none, below=1mm of P4] {\large $P_2$};
			
			\draw[edge]
			(v) -- (north)
			(west) -- (east)
			(u3) -- (P3)
			(u4) -- (P4);
			
			\draw[edge, bend right]
			(u1) edge (P1);
			
			\draw[edge, bend left]
			(u2) edge (P2);
			\end{scope}
			
			\end{tikzpicture}
			\caption{Case 1. (b)}
		\end{subfigure}

		\begin{subfigure}[b]{.49\textwidth}
			\centering
			\begin{tikzpicture}
			[every node/.style={draw, circle, fill, inner sep=0pt, minimum 
				size=2mm}]
			\tikzset{edge/.style={very thick}}
			
			\node[draw=none, fill=none] (v) at (0,1) {};
			\draw[very thick, black] (0,0) circle (0.6) {};
			\node[] (north) at (0,0.6) {};
			\node[] (west) at (180:0.6) {}
			node[draw=none, fill=none, left=8mm] {\large $B'$};
			\node[] (east) at (0:0.6) {};
			\node[] (u1) at (30:0.6) {};
			\node[] (u3) at (60:0.6) {};
			\node[] (u2) at (-90:0.6) {};

			\node[] (P1) at (1, 1) {}
			node[draw=none, fill=none, above=1mm of P1] {\large $P'_3$};
			\node[fill=none] (P2) at ({0.6*cos(-90)}, -1) {}
			node[draw=none, fill=none, below=1mm of P2] {\large $P'_2$};
			\node[fill=none] (P3) at (0.33, 1) {}
			node[draw=none, fill=none, above=1mm of P3] {\large $P'_1$};
			
			\draw[edge]
			(v) -- (north)
			(west) -- (east)
			(u3) -- (P3);

			\draw[edge, bend right]
			(u1) edge (P1);
			
			\draw[edge]
			(u2) edge (P2);
			
			\begin{scope}[xshift=4cm,yshift=0cm]
			\draw[very thick, black] (0,0) circle (0.6) {}
			node[draw=none, fill=none, above left = 4mm and 4mm] {\large 
				$B$};
			\node[] (west) at (180:0.6) {};
			\node[] (east) at (0:0.6) {};
			\node[] (u1) at (270:0.6) {};
			\node[] (u3) at (225:0.6) {};
			\node[] (u2) at (315:0.6) {};
			
			\node[] (P1) at (0, -1) {}
			node[draw=none, fill=none, below=1mm of P1] {\large $P_1$};
			\node[] (P2) at (1, -1) {}
			node[draw=none, fill=none, below=1mm of P2] {\large $P_2$};
			\node[fill=none] (P3) at (-1, -1) {}
			node[draw=none, fill=none, below=1mm of P3] {\large $P_3$};
			
			\draw[edge]
			(west) -- (east)
			(u1) -- (P1);
			
			\draw[edge, bend right]
			(u3) edge (P3);
			
			\draw[edge, bend left]
			(u2) edge (P2);
			\end{scope}
			
			\end{tikzpicture}
			\caption{Case 2. (a) i.}
		\end{subfigure}\hfill
		\begin{subfigure}[b]{.49\textwidth}
			\centering
			\begin{tikzpicture}
			[every node/.style={draw, circle, fill, inner sep=0pt, minimum 
				size=2mm}]
			\tikzset{edge/.style={very thick}}
			
			\node[draw=none, fill=none] (v) at (0,1) {};
			\draw[very thick, black] (0,0) circle (0.6) {};
			\node[] (north) at (0,0.6) {};
			\node[] (west) at (180:0.6) {}
			node[draw=none, fill=none, left=8mm] {\large $B'$};
			\node[] (east) at (0:0.6) {};
			\node[] (u1) at (45:0.6) {};
			\node[] (u3) at (-60:0.6) {};
			\node[] (u2) at (-120:0.6) {};
			
			\node[fill=none] (P1) at ({0.6*cos(45)}, 1) {}
			node[draw=none, fill=none, above=1mm of P1] {\large $P'_1$};
			\node[fill=none] (P4) at (0.33, -1) {}
			node[draw=none, fill=none, below=1mm of P4] {\large $P'_4$};
			\node[] (P2) at (-0.33, -1) {}
			node[draw=none, fill=none, below=1mm of P2] {\large $P'_2$};
			
			\draw[edge]
			(v) -- (north)
			(west) -- (east)
			(u1) -- (P1)
			(u3) -- (P4)
			(u2) -- (P2);
			
			\begin{scope}[xshift=4cm,yshift=0cm]
			\draw[very thick, black] (0,0) circle (0.6) {}
			node[draw=none, fill=none, above left = 4mm and 4mm] {\large 
				$B$};
			\node[] (west) at (180:0.6) {};
			\node[] (east) at (0:0.6) {};
			\node[] (u1) at (270:0.6) {};
			\node[] (u3) at (225:0.6) {};
			\node[] (u2) at (315:0.6) {};
			
			\node[] (P1) at (0, -1) {}
			node[draw=none, fill=none, below=1mm of P1] {\large $P_2$};
			\node[fill=none] (P2) at (1, -1) {}
			node[draw=none, fill=none, below=1mm of P2] {\large $P_4$};
			\node[] (P3) at (-1, -1) {}
			node[draw=none, fill=none, below=1mm of P3] {\large $P_1$};
			
			\draw[edge]
			(west) -- (east)
			(u1) -- (P1);
			
			\draw[edge, bend right]
			(u3) edge (P3);
			
			\draw[edge, bend left]
			(u2) edge (P2);
			\end{scope}
			
			\end{tikzpicture}
			\caption{Case 2. (a) ii.}
		\end{subfigure}
		
		\begin{subfigure}[b]{0.49\textwidth}
			\centering
			\begin{tikzpicture}
			[every node/.style={draw, circle, fill, inner sep=0pt, minimum 
				size=2mm}]
			\tikzset{edge/.style={very thick}}
			
			\node[draw=none, fill=none] (v) at (0,0.4) {};
			\draw[very thick, black] (0,0) circle (0.6) {};
			\node[] (north) at (0,0) {};
			\node[] (west) at (180:0.6) {}
			node[draw=none, fill=none, left=8mm] {\large $B'$};
			\node[] (east) at (0:0.6) {};
			\node[] (u1) at (60:0.6) {};
			\node[] (u3) at (120:0.6) {};
			\node[] (u4) at (-60:0.6) {};
			\node[] (u2) at (-120:0.6) {};
			
			\node[fill=none] (P1) at (0.33, 1) {}
			node[draw=none, fill=none, above=1mm of P1] {\large $P'_1$};
			\node[] (P3) at (-0.33, 1) {}
			node[draw=none, fill=none, above=1mm of P3] {\large $P'_3$};
			\node[] (P4) at (0.33, -1) {}
			node[draw=none, fill=none, below=1mm of P4] {\large $P'_4$};
			\node[fill=none] (P2) at (-0.33, -1) {}
			node[draw=none, fill=none, below=1mm of P2] {\large $P'_2$};
			
			\draw[edge]
			(v) -- (north)
			(west) -- (east)
			(u1) -- (P1)
			(u4) -- (P4)
			(u2) -- (P2)
			(u3) -- (P3);
			
			\begin{scope}[xshift=4cm,yshift=0cm]
			
			\node[draw=none, fill=none] (v) at (0,1) {};
			\draw[very thick, black] (0,0) circle (0.6) {}
			node[draw=none, fill=none, above left = 4mm and 4mm] {\large 
				$B$};
			\node[] (north) at (0,0.6) {};
			\node[] (west) at (180:0.6) {};
			\node[] (east) at (0:0.6) {};
			\node[] (u1) at (216:0.6) {};
			\node[] (u3) at (252:0.6) {};
			\node[] (u4) at (288:0.6) {};
			\node[] (u2) at (324:0.6) {};
			
			\node[fill=none] (P1) at (-1, -1) {}
			node[draw=none, fill=none, below=1mm of P1] {\large $P_3$};
			\node[fill=none] (P2) at (1, -1) {}
			node[draw=none, fill=none, below=1mm of P2] {\large $P_4$};
			\node[] (P3) at (-0.33, -1) {}
			node[draw=none, fill=none, below=1mm of P3] {\large $P_1$};
			\node[] (P4) at (0.33, -1) {}
			node[draw=none, fill=none, below=1mm of P4] {\large $P_2$};
			
			\draw[edge]
			(v) -- (north)
			(west) -- (east)
			(u3) -- (P3)
			(u4) -- (P4);
			
			\draw[edge, bend right]
			(u1) edge (P1);
			
			\draw[edge, bend left]
			(u2) edge (P2);
			
			\end{scope}
			
			\end{tikzpicture}
			\caption{Case 2. (b)}
		\end{subfigure}
		\caption{The different cases from the proof of Claim~\ref{cla:L2L2Bad}. 
			Pendant blobs are indicated by filled and unfilled leaves with the 
			label~$P_i$ for some~$i$. The filled vertices indicate a pendant 
			level-$2$ blob of the form~$(2,0,0,0)$, whereas unfilled vertices 
			indicate a pendant level-$1$ blob on two leaves.}
		\label{fig:thmclaim2}
	\end{figure}
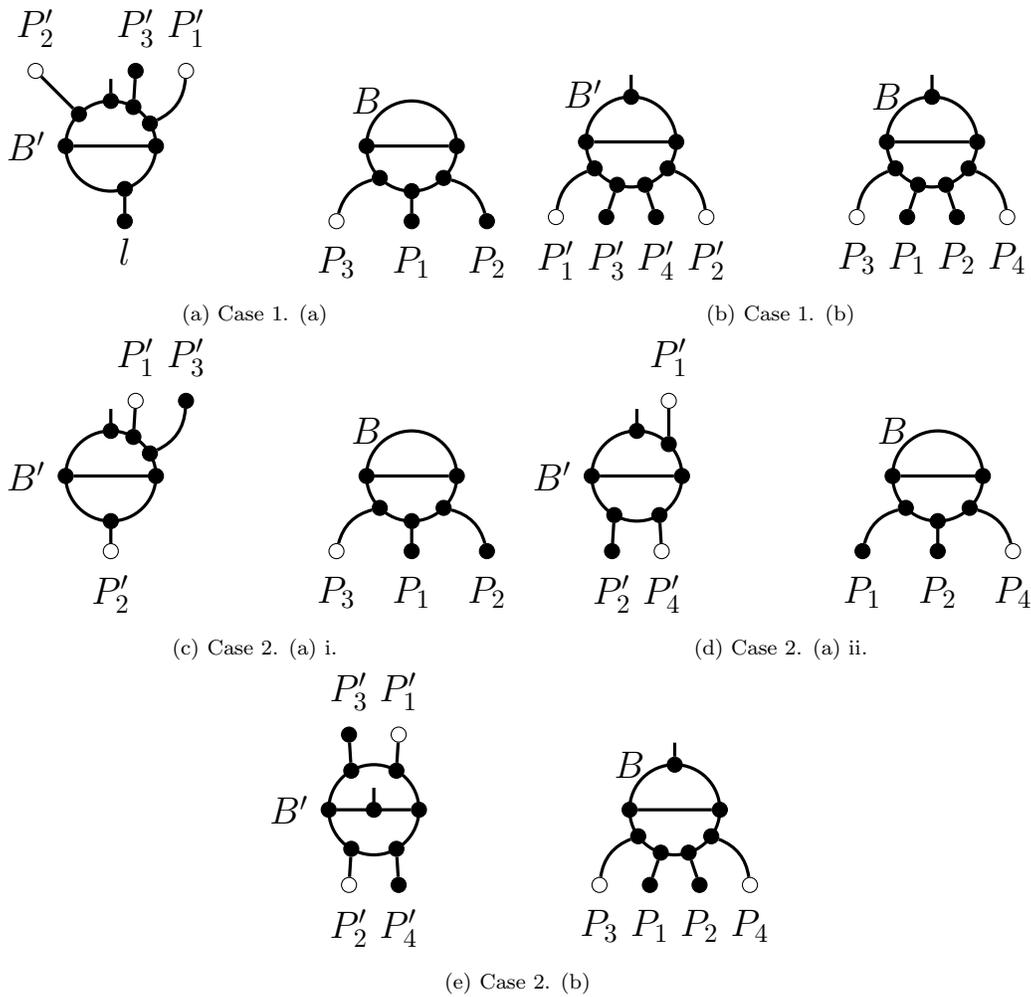
	
	\begin{claim}
		A pendant level-$2$ blob may not be adjacent to both a pendant 
		level-$1$ blob and a leaf simultaneously in~$\bar{B}$ (and 
		in~$\bar{B'}$).
	\end{claim}
	\begin{proof}
		Suppose that a pendant level-$1$ blob~$P_1$ is adjacent to a pendant 
		level-$2$ blob~$P_2$, and some leaf~$l$ is adjacent to~$P_2$ 
		in~$\bar{B}$. To realize these distances in~$\bar{B'}$, we must 
		have that~$P'_1$ and~$P'_2$, the pendant level-$2$ and the level-$1$ 
		blobs that correspond to~$P_1$ and~$P_2$ must be adjacent, and that~$l$ 
		must be adjacent to~$P'_1$. Since level-$1$ blobs may be adjacent to at 
		most one level-$2$ blob, in both~$B$ and~$B'$,~$P_1$ 
		and~$P'_2$ must be adjacent to a pole or~$u$ or~$u'$. We now split into 
		cases depending on the position of~$u$ in~$\bar{B}$.
		\begin{enumerate}
			\item \textbf{$\bm{u}$ is adjacent to~$\bm{p_1}$:} We have
			\[d^N(P_1,u)-d^N(P_2,u) = 4-6 = -2.\]
			In~$\bar{B'}$, we have~$d^{N'}(P'_1,p'_1) = d^{N'}(P'_2,p'_1) = 4$. 
			It follows that
			\[d^{N'}(P'_1,u') - d^{N'}(P'_2,u') \ge 0,\]
			which contradicts Claim~\ref{cla:DistToCE}.
			\item \textbf{$\bm{u}$ is not adjacent to~$\bm{p_1}$:} Let~$v$ be a 
			vertex in~$\bar{B}$ such that~$d^N(p_1,v) = 2$ and~$v$ is not the 
			neighbor of~$l$. We claim that~$v$ is either a pole or equal 
			to~$u$. 
			
			Note that~$p_1$ must be adjacent to a pole since~$P_1$ can be 
			adjacent to at most one pendant level-$2$ blob, and~$P_1$ 
			cannot be adjacent to leaves or other pendant level-$1$ blobs by 
			Claim~\ref{cla:4Statements}. The shortest path from~$p_1$ to~$v$ 
			must contain this pole. Suppose for a contradiction that~$v$ is 
			either a neighbor of a leaf or that~$v$ is a connection of some 
			pendant blob. If~$v$ is a neighbor of a leaf~$l'$, then
			\[d^N(P_1,l) = 6,\]
			meaning that~$P'_1$ must be adjacent to~$l'$ in~$\bar{B'}$. But 
			this is impossible since~$P'_1$ is adjacent to~$P'_2$ and~$l$. 
			Secondly, if~$v$ is a connection of a pendant level-$1$ blob~$P_3$, 
			then
			\[d^N(P_1,P_3) = 8,\]
			but this is impossible since two pendant level-$1$ blobs must have 
			shortest distance at least 10 by Claim~\ref{cla:L2L2Bad}. Finally, 
			if~$v$ is a connection of a pendant level-$2$ blob~$P_4$, then
			\[d^N(P_1,P_4) = 9,\]
			which means that the corresponding pendant level-$1$ blob must be 
			adjacent either to~$P'_2$ or~$l$ in~$\bar{B'}$. But both of these 
			are forbidden by Claim~\ref{cla:4Statements}. Therefore~$v$ must be 
			a pole or~$u$.
			
			But this means that if~$u$ was not placed in the main path 
			of~$\bar{B}$ that did not contain~$p_1$, then~$\bar{B}$ would have 
			parallel edges. So~$u$ must be contained in one of these two main 
			sides. This means that
			\[d^N(P_1,u) - d^N(P_2,u) = 5- 7 = -2.\]
			In~$B'$, as in the previous case, we have that
			\[d^{N'}(P'_1,u') - d^{N'}(P'_2,u') \ge 0,\]
			which clearly contradicts Claim~\ref{cla:DistToCE}
		\end{enumerate}
		These cover all possibilities for a pendant level-$2$ blob to be 
		adjacent to a pendant level-$1$ blob and to a leaf. In all cases, we 
		reach a contradiction.
	\end{proof}

	\begin{claim}
		$\bar{B}$ (and~$\bar{B'}$) contain at most one pendant level-$1$ blob.
	\end{claim}
	\begin{proof}
		Suppose for a contradiction that~$\bar{B}$ contained two pendant 
		level-$1$ blobs~$P_1$ and~$P_2$. We know that they must be shortest 
		distance at least~$10$ apart. Since~$d^N(P_1,p_1) = d^N(P_2,p_2) = 3$, 
		we require that~$d^N(p_1,p_2) \ge 4$. 
		\begin{enumerate}
			\item \textbf{$\bm{p_1}$ and~$\bm{p_2}$ are contained in the same 
				main path~$s$ of~$\bm{\bar{B}}$:} Consider the path from~$p_1$ 
			to~$p_2$ that contains only the vertices of the main path~$s$. 
			Since we require~$d^N(p_1,p_2)\ge 4$, we need at least three 
			vertices in this path excluding~$p_1$ and~$p_2$. By 
			Claims~\ref{cla:4Statements} and~\ref{cla:L1LeafBad}, these three 
			vertices must be two connections~$p_3,p_4$ of pendant level-$2$ 
			blobs and the vertex~$u$. In particular,~$p_1$ and~$p_2$ must be 
			adjacent to~$p_3$ and~$p_4$, and~$p_3$ and~$p_4$ must be adjacent 
			to~$u$. It follows from Claim~\ref{cla:4Statements} that~$s$ 
			contains only these five vertices.
			
			Now consider the two main paths~$s_1$ and~$s_2$ of~$B$ that is 
			not~$s$. If these main paths are both empty, then~$d^N(P_1,P_2) = 
			9$, which contradicts the fact that~$d^N(P_1,P_2) \ge 10$. So they 
			must both contain vertices. Let~$v_i$ denote the vertex in~$s_i$ 
			such 
			that~$d^N(v_i,p_1) = 2$, for~$i=1,2$. Firstly, if~$v_1$ is a 
			neighbor of some leaf~$l$, then~$d^N(P_1,l) = 6$, meaning that~$l$ 
			must be adjacent to~$P'_1$ in~$N'$. But this would imply 
			that~$P'_1$ is adjacent to a leaf~$l$ and a pendant level-$2$ 
			blob~$P'_3$ in~$\bar{B'}$, which contradicts 
			Claim~\ref{cla:L1LeafBad}. Secondly, if~$v_1$ is a connection of a 
			pendant level-$1$ blob~$P_5$, then~$d^N(P_1,P_5) = 8$. But two 
			pendant level-$1$ blobs must be shortest distance at least~$10$ 
			apart. So~$v_1$ must be a connection of a pendant level-$2$ blob. 
			Similarly,~$v_2$ must be a connection of a pendant level-$2$ blob. 
			But this implies that~$\bar{B'}$ contains four pendant level-$1$ 
			blobs. 
			
			The main path of~$\bar{B'}$ with~$u'$ contains exactly two pendant 
			level-$1$ blobs; the other two main paths of~$\bar{B'}$ contains 
			exactly one pendant level-$1$ blob each. Consider these two latter 
			main paths. By Claims~\ref{cla:4Statements},~\ref{cla:L2L2Bad}, 
			and~\ref{cla:L1LeafBad} these two main paths may contain an 
			additional pendant level-$2$ blob each, but no other leaves. This 
			means that the pendant level-$1$ blobs in these two main paths are 
			shortest distance at most~$9$ to one another, which is a 
			contradiction. Therefore, the vertices~$v_i$ for~$i=1,2$ cannot be 
			neighbors of leaves / connections of pendant blobs, meaning 
			that~$\bar{B}$ contains parallel edges, which is a contradiction.
			
			\item \textbf{$\bm{p_1}$ and~$\bm{p_2}$ are not contained in the 
				same main path of~$\bm{\bar{B}}$:} Let~$s_1$ and~$s_2$ denote 
				the 
			main paths of~$\bar{B}$ that contain~$p_1$ and~$p_2$, respectively. 
			If~$s_1$ and~$s_2$ do not contain the vertex~$u$, 
			then~$d^N(P_1,P_2) \le 9$ and we are done. 
			
			So suppose without loss of generality that~$s_1$ contains the 
			vertex~$u$. Since we require~$d^N(p_1,p_2) \ge 4$, considering the 
			path between~$P_1$ and~$P_2$ that uses only the edges from~$s_1$ 
			and~$s_2$, without using the vertex~$u$, we see that~$P_1$ 
			and~$P_2$ must be adjacent to pendant level-$2$ blobs~$P_3,P_4$, 
			respectively. In particular, this path contains the 
			subpath~$p_1p_3vp_4p_2$ for some pole~$v$ of~$\bar{B}$. 
			Therefore~$p_1$ must be adjacent to~$u$ in~$\bar{B}$. Then,
			\[d^N(P_1,u) - d^N(P_3,u) = 4-6 = -2.\]
			In~$\bar{B'}$, we have~$d^{N'}(P'_1,p'_1) = d^{N'}(P'_3,p'_1) = 4$. 
			It follows that
			\[d^{N'}(P'_1,u') - d^{N'}(P'_3,u') \ge 0,\]
			which is a contradiction. 
		\end{enumerate}
		This covers all cases for when~$\bar{B}$ contains more than one pendant 
		level-$1$ blob, which all result in a contradiction. Therefore the 
		claim follows.
	\end{proof}
	
\end{appendices}




\begin{thebibliography}{}
	
	\bibitem[\protect\astroncite{Bapteste et~al.}{2013}]{bapteste2013networks}
	Bapteste, E., van Iersel, L., Janke, A., Kelchner, S., Kelk, S., McInerney,
	J.~O., Morrison, D.~A., Nakhleh, L., Steel, M., Stougie, L., et~al. (2013).
	\newblock Networks: expanding evolutionary thinking.
	\newblock {\em Trends in Genetics}, 29(8):439--441.
	
	\bibitem[\protect\astroncite{Bordewich 
	et~al.}{2018a}]{bordewich2018recovering}
	Bordewich, M., Huber, K.~T., Moulton, V., and Semple, C. (2018a).
	\newblock Recovering normal networks from shortest inter-taxa distance
	information.
	\newblock {\em Journal of mathematical biology}, pages 1--24.
	
	\bibitem[\protect\astroncite{Bordewich and
		Semple}{2016}]{bordewich2016determining}
	Bordewich, M. and Semple, C. (2016).
	\newblock Determining phylogenetic networks from inter-taxa distances.
	\newblock {\em Journal of mathematical biology}, 73(2):283--303.
	
	\bibitem[\protect\astroncite{Bordewich
		et~al.}{2018b}]{bordewich2018constructing}
	Bordewich, M., Semple, C., and Tokac, N. (2018b).
	\newblock Constructing tree-child networks from distance matrices.
	\newblock {\em Algorithmica}, 80(8):2240--2259.
	
	\bibitem[\protect\astroncite{Bordewich and
		Tokac}{2016}]{bordewich2016algorithm}
	Bordewich, M. and Tokac, N. (2016).
	\newblock An algorithm for reconstructing ultrametric tree-child networks 
	from
	inter-taxa distances.
	\newblock {\em Discrete applied mathematics}, 213:47--59.
	
	\bibitem[\protect\astroncite{Bryant et~al.}{2007}]{bryant2007consistency}
	Bryant, D., Moulton, V., and Spillner, A. (2007).
	\newblock Consistency of the neighbor-net algorithm.
	\newblock {\em Algorithms for Molecular Biology}, 2(1):8.

    \bibitem[\protect\astroncite{Buneman}{1971}]{buneman1971recovery}
    Buneman, P. (1971).
    \newblock The recovery of trees from measures of dissimilarity.
    \newblock {\em Mathematics in the archaeological and historical sciences}, pages 387--395.
	
	\bibitem[\protect\astroncite{Cunningham}{1978}]{cunningham1978free}
	Cunningham, J.~P. (1978).
	\newblock Free trees and bidirectional trees as representations of
	psychological distance.
	\newblock {\em Journal of mathematical psychology}, 17(2):165--188.
	
	\bibitem[\protect\astroncite{Dewdney}{1979}]{dewdney1979diagonal}
	Dewdney, A. (1979).
	\newblock Diagonal tree codes.
	\newblock {\em Information and Control}, 40(2):234--239.
	
	\bibitem[\protect\astroncite{Forcey and 
	Scalzo}{2020}]{forcey2020phylogenetic}
	Forcey, S. and Scalzo, D. (2020).
	\newblock Phylogenetic networks as circuits with resistance distance.
	\newblock {\em Frontiers in Genetics}, 11.
	
	\bibitem[\protect\astroncite{Hakimi and Yau}{1965}]{hakimi1965distance}
	Hakimi, S.~L. and Yau, S.~S. (1965).
	\newblock Distance matrix of a graph and its realizability.
	\newblock {\em Quarterly of applied mathematics}, 22(4):305--317.
	
	\bibitem[\protect\astroncite{Huson et~al.}{2010}]{huson2010phylogenetic}
	Huson, D.~H., Rupp, R., and Scornavacca, C. (2010).
	\newblock {\em Phylogenetic networks: concepts, algorithms and 
	applications}.
	\newblock Cambridge University Press.
	
	\bibitem[\protect\astroncite{van Iersel
	et~al.}{2020}]{van2020reconstructibility}
	van Iersel, L., Moulton, V., and Murakami, Y. (2020).
	\newblock Reconstructibility of unrooted level-$k$ phylogenetic networks 
	from
	distances.
	\newblock {\em Advances in Applied Mathematics}, 120:102075.
	
	\bibitem[\protect\astroncite{Jones et~al.}{2019}]{jones2019cutting}
	Jones, M., Gambette, P., van Iersel, L., Janssen, R., Kelk, S., Pardi, F., 
	and
	Scornavacca, C. (2019).
	\newblock Cutting an alignment with ockham's razor.
	\newblock {\em arXiv preprint arXiv:1910.11041}.
	
	\bibitem[\protect\astroncite{Mol}{2020}]{Mol2020}
	Mol, R. (2020).
	\newblock Reconstruction of phylogenetic networks: An algorithm for
	reconstructing level-2 binary networks based on their distances.
	
	\bibitem[\protect\astroncite{Morrison}{2011}]{morrison2011introduction}
	Morrison, D.~A. (2011).
	\newblock {\em An introduction to phylogenetic networks}.
	\newblock RJR productions.
	
	\bibitem[\protect\astroncite{Schvaneveldt
		et~al.}{1989}]{schvaneveldt1989network}
	Schvaneveldt, R.~W., Durso, F.~T., and Dearholt, D.~W. (1989).
	\newblock Network structures in proximity data.
	\newblock In {\em Psychology of learning and motivation}, volume~24, pages
	249--284. Elsevier.
	
	\bibitem[\protect\astroncite{Willson}{2006}]{willson2006unique}
	Willson, S.~J. (2006).
	\newblock Unique reconstruction of tree-like phylogenetic networks from
	distances between leaves.
	\newblock {\em Bulletin of mathematical biology}, 68(4):919--944.
	
\end{thebibliography}





\end{document}